\documentclass[11pt]{article}
\usepackage[margin=1in]{geometry}
\usepackage{amssymb,amsmath,latexsym,mathtools}
\usepackage{mathbbol}
\usepackage{enumerate}
\usepackage{graphicx}            
\usepackage{amsfonts}           
\usepackage{amsthm}               
\usepackage{url}
\usepackage{marvosym}
\usepackage{color}
\usepackage[colorlinks,pdfencoding=auto,pdftex,pdfpagelabels,bookmarks,hyperindex,hyperfigures]{hyperref}
\usepackage{soul}
\usepackage{bm}
\usepackage{tikz-cd}
\usepackage[inline]{enumitem}
\usepackage{tasks}
\usepackage{stmaryrd}

\usepackage{blkarray}
\newcommand{\matindex}[1]{\mbox{\scriptsize#1}}

\DeclareSymbolFontAlphabet{\mathbb}{AMSb}
\DeclareSymbolFontAlphabet{\mathbbl}{bbold}
\DeclareMathAlphabet\mathbfcal{OMS}{cmsy}{b}{n}

\newtheorem{intro_thm}{Theorem}

\newtheorem{theorem}{Theorem}[section]
\newtheorem*{theorem*}{Theorem}
\newtheorem{lemma}[theorem]{Lemma}
\newtheorem{proposition}[theorem]{Proposition}
\newtheorem{corollary}[theorem]{Corollary}

\theoremstyle{definition}
\newtheorem{definition}[theorem]{Definition}
\newtheorem{example}[theorem]{Example}
\newtheorem{remark}[theorem]{Remark}
\newtheorem{question}[theorem]{Question}

\newtheorem{notation}[theorem]{Notation}

\newtheorem*{open}{Open Problem}

\newcommand{\op}[1]{\operatorname{#1}}

\newcommand{\wt}{\mathbf{w}^{}_{P}}
\newcommand{\towerchi}{\underline{\bm\chi}}
\newcommand{\towerwt}{\mathbf{w}^{}_T}
\newcommand{\towerequiv}{\underset{\op{Tower}}{\displaystyle\equiv\joinrel\equiv}}
\newcommand{\fac}{\mathcal{F}}

\newcommand{\CC}{\mathbb{C}}
\newcommand{\QQ}{\mathbb{Q}}

\newcommand{\hook}[2]{\mathfrak{h}_{#2}^{#1}}
\newcommand{\qhook}[2]{\mathfrak{qh}_{#2}^{#1}}
\newcommand{\ghook}[3]{{}_#3\hook{#1}{#2}} 
\newcommand{\parA}{([n],\dots)}
\newcommand{\parB}{([1^n],\dots)}
\newcommand{\parC}{(\qhook{n}{k},\dots)}
\newcommand{\parD}{(\hook{n-1}{k},\dots,1,\dots)}

\begin{document}

\title{Coxeter factorizations with generalized Jucys-Murphy weights and Matrix Tree theorems for reflection groups}

\author{\Large Guillaume Chapuy,\ \Large Theo Douvropoulos\thanks{Both authors are supported by the European Research Council, grant ERC-2016- STG 716083 ``CombiTop''.}}

\newcommand{\Address}{{
  \bigskip
  \footnotesize
	
  Guillaume~Chapuy, \textsc{IRIF, UMR CNRS 8243, Universit\'e Paris Diderot, Paris 7, France}\par\nopagebreak
  \textit{E-mail address:} \texttt{guillaume.chapuy@irif.fr}	
  
   Theo~Douvropoulos, \textsc{University of Massachusetts, Amherst, MA, USA}\par\nopagebreak
  \textit{E-mail address:} \texttt{douvropoulos@math.umass.edu}

}}

\date{}
\maketitle
\begin{abstract}
	We prove universal (case-free) formulas for the \emph{weighted} enumeration of factorizations of Coxeter elements into products of reflections valid in any well-generated reflection group $W$, in terms of the spectrum of an associated operator, the $W$-Laplacian. This covers in particular all finite Coxeter groups. The results of this paper include generalizations of the Matrix Tree and Matrix Forest theorems to reflection groups, and cover reduced (shortest length) as well as arbitrary length factorizations.
	
	Our formulas are relative to a choice of weighting system that consists of $n$ free scalar parameters and is defined in terms of a tower of parabolic subgroups.
	To study such systems we introduce (a class of) variants of the Jucys-Murphy elements for every group, from which we define a new notion of `tower equivalence' of virtual characters. A main technical point is to prove the tower equivalence between virtual characters naturally appearing in the problem, and exterior products of the reflection representation of $W$.

Finally we study how this $W$-Laplacian matrix we introduce can be used in other problems in Coxeter combinatorics; for instance, we explain how it defines analogues of trees for $W$ and how it relates them to Coxeter factorizations. Moreover, we build a parabolic recursion for the $W$-Laplacian that proves, without relying on the classification, new numerological identities between the Coxeter number of $W$ and those of its parabolic subgroups, and, when $W$ is a Weyl group, a new, explicit formula for the volume of the corresponding root zonotope.

\end{abstract}

\section{Introduction and main results}

More than a century after its genesis\footnote{In his 1847 paper, Kirchhoff was interested in the theory of electrical networks, whose fundamental laws he had introduced two years earlier, and the linear equations that govern them. Nonetheless, the main objects and ideas for the modern Matrix Tree theorem do indeed appear there; for a captivating telling of the story, see \cite{history_mt}.} \cite{kirchhoff} and despite its simplicity, Kirchoff's ``Matrix Tree'' theorem is still one of the most beautiful and remarkable enumeration formulas in mathematics. It expresses the weighted number of labeled trees on the vertex-set $\{1,2,...,N\}$, where each edge $\{i,j\}$ receives an arbitrary complex weight $\omega_{ij}$, as the determinant of (any first minor of) the so-called \emph{Laplacian matrix} constructed form these weights.
The influence of Kirchoff's theorem, relating spectral properties of the Laplacian to a combinatorial problem, has extended well beyond combinatorics, from mathematical physics to probability theory. 
In this work we will show how, for any given finite reflection group $W$, the spectrum of a natural ``$W$-Laplacian'' operator plays a key role in enumerative and geometric problems related to $W$.

When all weights are equal to one, Kirchoff's determinant evaluates to Cayley's famous $N^{N-2}$ formula that counts (unweighted) labeled trees of size $N$. Shortly after Cayley, Hurwitz knew \cite{Hurwitz} that this same number counts factorizations $\tau_1\cdots \tau_{N-1}=(1,2,\cdots ,N)$ of the increasing long cycle of the symmetric group $\mathfrak{S}_N$ into a product of $N-1$ transpositions $\tau_i$. 
There is a correspondence between factorisations and trees (easier to describe if one considers all long cycles simultaneously, rather than the increasing one, see\cite{Denes}, but also one that can be realized geometrically via the interpretation of trees as branched coverings of the sphere by itself \cite[Corol.~5.1.4]{LZ-graphs-on-surfaces}). In this way, Kirchoff's theorem can be viewed as a far-reaching generalization of Hurwitz's theorem, counting factorizations with weights.

But Cayley's formula has at least two other sorts of generalizations. The first one, which lies most naturally in the field of enumerative geometry and the study of Hurwitz numbers, considers coverings of the sphere by a surface of higher genus $g$. Combinatorially, this amounts to counting factorizations of arbitrary length $\tau_1\cdots \tau_{\ell}=(1,2,...,N)$ with $\ell=N-1+2g$. This was done by Jackson~\cite{Jackson} (see also~\cite{SSV} who further noted the geometric interpretation) in a paper pioneering the use of representation-theoretic techniques in enumeration. He obtained a remarkable product form for their generating function ($\mathcal{R}$ denotes the set of transpositions $\tau_i$ in $\mathfrak{S}_N$):
\begin{equation}
\sum_{\ell\geq 0}\#\Big\{(\tau_1,\dots,\tau_{\ell})\in\mathcal{R}^{\ell}:\tau_1\cdots\tau_{\ell}=(1,2,\dots,N)\Big\}\cdot \dfrac{t^{\ell}}{\ell!}=\dfrac{e^{t\binom{N}{2}}}{N!}\cdot \big(1-e^{-tN}\big)^{N-1}.\label{EQ: Jackson-formula}
\end{equation}
The first term in the $t=0$ expansion of this formula gives back the Cayley/Hurwitz formula.

Another generalization that has become very popular after the development of Coxeter combinatorics, is to replace the symmetric group $\mathfrak{S}_n$ by a (real or complex) reflection group $W$ of rank $n$.
There, analogues of transpositions and long cycles are respectively reflections and Coxeter elements (see Section~\ref{sec:JM}). The enumeration of minimal length reflection factorizations of a fixed Coxeter element $c$ was first carried out by Deligne \cite{Deligne,kluitmann-thesis} (crediting discussions with Tits and Zagier) for simply-laced Weyl groups, confirming a conjecture of Looijenga. The answers take a remarkable product form as 
\begin{align}\label{eq:introDeligne}
	\#\Big\{(\tau_1,\dots,\tau_{\ell})\in\mathcal{R}^{\ell}:\tau_1\cdots\tau_{\ell}=c\Big\}=	\frac{h^nn!}{|W|},
\end{align}
where  $h$ is the Coxeter number of $W$ (and equals the order of $c$), and $\mathcal{R}$ is the set of reflections in~$W$. This uniform description of the counts was not observed originally and might be best credited as the Arnol'd-Bessis-Chapoton formula\footnote{Looijenga's conjecture \cite{Looijenga}, slightly reinterpreted, was that the left hand side of \eqref{eq:introDeligne} equals the degree of a covering map onto the configuration space of $n$ points in the plane. Arnol'd \cite[Thm.~11]{arnold_ICM} was the first to notice that this degree had a uniform formula, given by the right hand side of \eqref{eq:introDeligne}.}. 
In the case $W=\mathfrak{S}_N$ one has $(h,n)=(N,N-1)$ and \eqref{eq:introDeligne} is nothing but Cayley's $N^{N-2}$ formula again! 

In the last few years, there has been significant evidence that these three directions of generalization (weighted enumeration, higher genus, arbitrary reflection groups) are compatible with each other. Burman and Zvonkine in \cite{BZ}, which was a major influence for this work, gave a striking formula combining the first and second direction (see also~\cite{AK}). If $\mathcal{C}$ denotes the class of long cycles in $\mathfrak{S}_N$ and $\mathcal{R}$ its set of transpositions, and if we consider a weight system $\mathbf{w}$ on $\mathcal{R}$ with $\binom{N}{2}$ scalar parameters $\bm\omega:=(\omega_{ij})$ (that is, given as $\mathbf{w}\big( (ij)\big):=\omega_{ij}$) the weighted analogue of Jackson's formula \eqref{EQ: Jackson-formula} becomes 

\begin{equation}
\sum_{(\tau_1,\tau_2,\dots,\tau_\ell,c) \in \mathcal{R}^\ell \times \mathcal{C}\atop \tau_1\tau_2\dots\tau_\ell =c,\ \ell\geq 0}\mathbf{w}(\tau_1)\cdots\mathbf{w}(\tau_{\ell})\cdot\dfrac{t^{\ell}}{\ell!}=\dfrac{e^{t\cdot \mathbf{w}(\mathcal{R})}}{n}\cdot \prod_{i=1}^{N-1}(1-e^{-t\lambda_i}),\label{EQ: Burman-Zvonkine}
\end{equation}
where $\mathbf{w}(\mathcal{R}):=\sum_{\tau\in\mathcal{R}}\mathbf{w}(\tau)$ and where the $\lambda_i:=\lambda_i(\bm\omega)$ are the nonzero eigenvalues of the (weighted) Laplacian matrix of the complete graph $K_N$. It is easy to see that the leading term of this expression readily recovers Kirchoff's theorem after the factorizations-trees correspondence, and that the unweighted case ($\omega_{ij}=1$) gives back Jackson's formula~\eqref{EQ: Jackson-formula} since all eigenvalues are then equal to~$N$. Shortly after Burman-Zvonkine, Stump and the first author \cite{CS} unified the second and third directions of generalization. For a reflection group $W$ of rank $n$ with set of reflections $\mathcal{R}$, and for any Coxeter element $c\in W$, they showed

\begin{equation}
\sum_{\ell\geq 0}\#\Big\{(\tau_1,\dots,\tau_{\ell})\in\mathcal{R}^{\ell}:\tau_1\cdots\tau_{\ell}=c\Big\}\cdot \dfrac{t^{\ell}}{\ell!}=\dfrac{e^{t\cdot|\mathcal{R}|}}{|W|}\cdot \big(1-e^{-th}\big)^{n},\label{EQ: CS-formula}
\end{equation}
which simultaneously generalizes~\eqref{EQ: Jackson-formula} and~\eqref{eq:introDeligne} (the latter being the leading term of its expansion). 

A striking common point becomes immediately evident in these generalizations; all formulas have an unexpectedly nice product form. Far beyond their apparent appeal however, they encode non-trivial representation theoretic properties (see \S~\ref{sec:Intro-2}) and combine seemingly disparate objects. In particular, the emergence of the Laplacian matrix and its spectrum in the Burman-Zvonkine formula \emph{after a character calculation} in \cite[Prop.~2.1]{BZ} was (to the authors and a priori) nothing less than astonishing.

One is thus led to several natural questions: can these results be put under a common roof, i.e. are they shadows of a more general universal formula? Is there an object that plays the role of the ``Laplacian'' for reflection groups $W$? If so, is there also a notion of ``$W$-trees" and does this new Laplacian still manage to relate them to factorizations? And, is there a conceptual \emph{explanation} for the common product forms?

In this paper we will see that the answers to these questions are very much related. In all cases the core object is indeed a new notion of a Laplacian matrix for reflection groups. Its spectrum renders enumerative formulas, simultaneously unifying the three generalizations discussed above, for factorizations (Thm.~\ref{thm:main}) as well as W-trees (Lemma~\ref{Lem:Burman M-T} and \S~\ref{sec: Trees vs Facns}). We give its definition below as it will bring together the new ideas in what follows. It can be readily seen (\S~\ref{subsec:Laplacians}) to reduce to the usual graph Laplacian when $W=\mathfrak{S}_N$. 

\begin{definition}[The $W$-Laplacian]\label{Defn:Intro-Laplacian}\ \newline
	Let $W$ be a complex reflection group of rank $n$, and equip its set of reflections $\mathcal{R}=(\tau_i)_{i=1}^{|\mathcal{R}|}$ with a weight system $\mathbf{w}$  given by $\mathbf{w}(\tau_i)=\omega_i$, for indeterminates $\bm\omega:=(\omega_i)_{i=1}^{|\mathcal{R}|}$. We define the (weighted) $W$-Laplacian matrix as 
$$\op{GL}(V)\ni L_W(\bm\omega):=\sum_{\tau\in\mathcal{R}}\mathbf{w}(\tau)\cdot \big(\mathbf{I}_n-\rho^{}_V(\tau)\big),$$ where $\mathbf{I}_n$ is the $n\times n$ identity matrix and $\rho^{}_V$ the reflection representation of $W$.
\end{definition}

In the rest of this introduction we present the main results of our paper. In particular,  Theorem~\ref{thm:main} in \S~\ref{sec:Intro-1} below gives our common generalization of the formulas~\eqref{EQ: Jackson-formula}, \eqref{eq:introDeligne}, \eqref{EQ: Burman-Zvonkine}, and~\eqref{EQ: CS-formula}. In \S~\ref{sec:Intro-2} we give an explanation of their product structure which comes down to a correspondence between a special virtual character of $W$ and the exterior powers of its reflection representation. To define this  correspondence we introduce Jucys-Murphy elements for $W$. Finally, in \S\ref{sec:Intro-3} we discuss a family of further applications of the $W$-Laplacian matrix in Coxeter combinatorics.

\subsection{Enumerative results}
\label{sec:Intro-1}

A first attempt to extend the Chapuy-Stump formula was done in \cite{DHR} where the authors assign weights to the reflections $\tau\in \mathcal{R}$ depending on the orbits of the fixed hyperplanes $V^{\tau}$. In this way however, at most two different weights can appear in an irreducible, well-generated group $W$ (corresponding to long and short roots in the Weyl case). On the other hand it is easy to see that a full generalization of the Burman-Zvonkine formula, where each reflection gets its own weight, is impossible (already for dihedral groups, see \S\ref{sec:dihedral}). We achieve a \emph{maximally good} (see \S\ref{sec: end: better version?}) intermediate version.

Our main enumerative result allows weights induced from a parabolic filtration of the set of reflections. To state it, we first set up some terminology refering to Section~\ref{sec:JM} for precise definitions. For a finite reflection group $W\subset GL(V)$ of rank $n$ (i.e. such that $\dim_\mathbb{C} V=n$),  we consider \emph{parabolic towers} of subgroups, of the form
\begin{align}\label{eq:towerintro}
T=(\{1\}=W_0 \subset W_1\subset \dots \subset W_n=W),
\end{align}
where each group $W_i$ is a parabolic subgroup of $W$.
We let $\mathcal{R}$ be the set of reflections of $W$ and $\mathcal{C}$ the set of its Coxeter elements, and we consider a weighting system in which each $\tau\in \mathcal{R}$ is assigned one of $n$ complex weights $\bm\omega:=(\omega_i)_{i=1}^n$ based only on the index $i$ such that $\tau\in W_{i}\setminus W_{i-1}$, 
\begin{align}\label{eq:weightintro}
	\mathbf{w}^{}_T(\tau) = \omega_i \in \mathbb{C} \mbox{ if }	\tau\in W_{i}\setminus W_{i-1}.
\end{align}
\begin{intro_thm}[Main combinatorial result]\label{thm:main}
	Let $W \in GL(V)$ be a well-generated complex reflection group of rank $n$ and Coxeter number $h$, and let 
	$T$ be a parabolic tower as in~\eqref{eq:towerintro}.
	Given a corresponding weighting system $\mathbf{w}^{}_T$ as in~\eqref{eq:weightintro}, 
	consider the generating function 
	$$\fac_W^T(t,\bm\omega) := \sum_{\ell \geq 1 } \frac{t^\ell}{\ell!}
	\sum_{(\tau_1,\tau_2,\dots,\tau_\ell,c) \in \mathcal{R}^\ell \times \mathcal{C}\atop \tau_1\tau_2\dots\tau_\ell =c} \mathbf{w}^{}_T(\tau_1)\mathbf{w}^{}_T(\tau_2)\cdots\mathbf{w}^{}_T(\tau_\ell),$$
	of weighted factorisations of Coxeter elements into reflections. Then we have that
	\begin{align}\label{eq:mainThm}
	\fac_W^T(t,\bm\omega) = \dfrac{e^{t\cdot \mathbf{w}^{}_T(\mathcal{R})}}{h} \prod_{i=1}^n (1-e^{-t\lambda_i}),
	\end{align}
where $\lambda_i=\lambda_i(\bm\omega)$ are the eigenvalues, with multiplicity, of the \emph{$W$-Laplacian} matrix $L_W^T(\bm\omega)$ of Defn.~\ref{Defn:Intro-Laplacian} (weighted by $\mathbf{w}^{}_T$) and where $\mathbf{w}^{}_T(\mathcal{R})=\sum_{\tau \in \mathcal{R}}\mathbf{w}^{}_T(\tau)$ is the total weight of all reflections.
\end{intro_thm}

When all weights are equal to $1$, all the $\lambda_i$ are equal to $h$ (Prop.~\ref{Prop: L_W has all eigenvalues =h}) and Theorem~\ref{thm:main} reduces to the Chapuy-Stump formula~\eqref{EQ: CS-formula}, and thus also recovers the Arnol'd-Bessis-Chapoton formula~\eqref{eq:introDeligne} (see Figure~\ref{Fig: poset of formulas}). In general, the $\lambda_i$ are always non-negative integer combinations of the weights $\omega_i$ and are completely determined by the Coxeter-theoretic data of the tower $T$ (see \S~\ref{sec:combinatorial}). Many combinatorial  corollaries of Theorem~\ref{thm:main} may arise this way; we describe a few possibilities in \S~\ref{Sec: end: W-Lapl-combinatorics}.

Still assuming arbitrary values for the $\omega_i$'s, extracting the leading coefficient in $t$ in Theorem~\ref{thm:main} shows that the (weighted) enumeration of  factorizations of minimal length takes the particularly nice form of a determinant. For the symmetric group $\mathfrak{S}_N$, this is an instance of Kirchoff's Matrix-Tree theorem, via the correspondence between labelled trees and factorizations mentioned above.
\begin{corollary}[$W$-Matrix-tree theorem with parabolic weights]\label{cor:reduced}
	Let $W$ be a well-generated complex reflection group, $T$ a parabolic tower, and assume the notations and hypotheses of Theorem~\ref{thm:main}. 
Then the weighted number of reduced reflection factorisations of all Coxeter elements is given by 
	\begin{align}\label{eq:reduced}
\sum_{(\tau_1,\tau_2,\dots,\tau_n,c) \in \mathcal{R}^n \times \mathcal{C}\atop \tau_1\tau_2\dots\tau_n =c} \mathbf{w}^{}_T(\tau_1)\mathbf{w}^{}_T(\tau_2)\cdots\mathbf{w}^{}_T(\tau_n)=
		\dfrac{n!}{h} \det L_W^T(\bm\omega),
	\end{align}
where $L^T_W$ is the $W$-Laplacian matrix of Defn.~\ref{Defn:Intro-Laplacian} (weighted by $\mathbf{w}^{}_T$).
\end{corollary}

\begin{figure}[h]
\center
\includegraphics[scale=1]{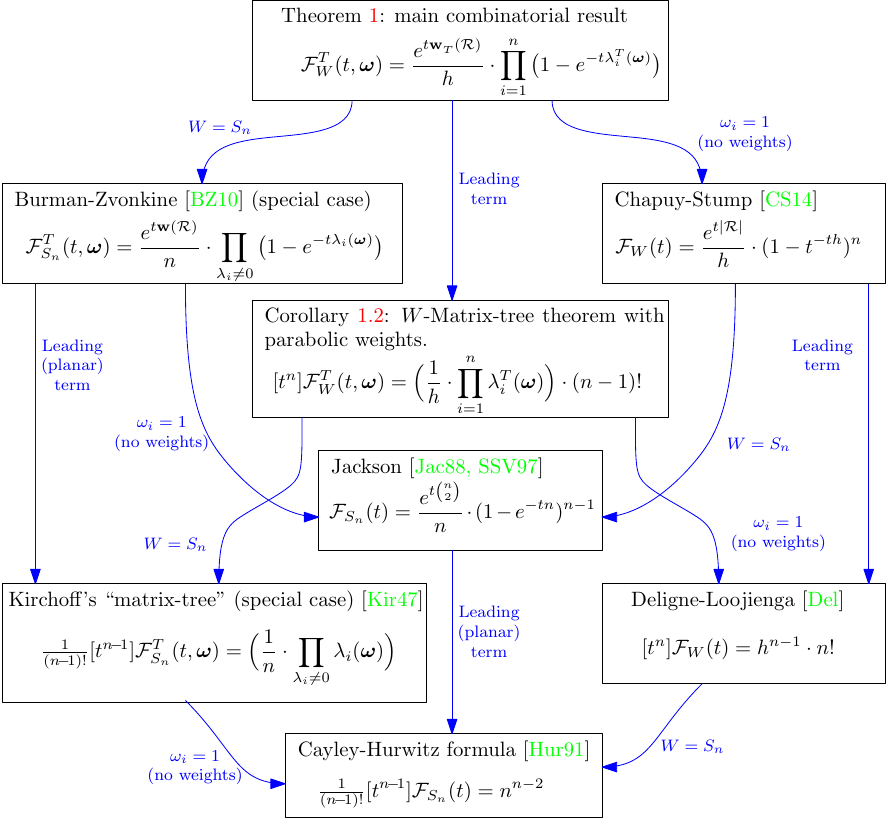}
\caption{The poset of implications from Theorem~\ref{thm:main} down to formulas  \eqref{EQ: Jackson-formula}, \eqref{eq:introDeligne}, \eqref{EQ: Burman-Zvonkine}, and \eqref{EQ: CS-formula}.}
\label{Fig: poset of formulas}
\end{figure}

\begin{remark}[Prefactors]\label{rem:prefactors}

	In the case $W=\mathfrak{S}_N$, the prefactor $\frac{n!}{h}=\frac{(N-1)!}{N}$ in formula~\eqref{eq:reduced} might come as a surprise in comparison to the usual formulation of Kirchoff's theorem. It is related to two things.
	First, the usual practice is to define the Laplacian matrix as an $N\times N$ matrix having a zero eigenvalue. Here our matrix is defined directly as a generically invertible $n\times n$ matrix but it is essentially the same object.
	In particular, the determinant of our Laplacian  differs from any first-minor of the usual Laplacian  by a factor of $\frac{1}{N}$, as is well known~\cite[Lem. 5.6.5 \& Cor. 5.6.6]{Stanley:EC2}.
	Secondly, in~\eqref{eq:reduced} we are summing over \emph{all} long cycles $c$, and not fixing a particular one, which amounts to considering edge-labelled spanning trees of $K_N$, hence the extra-factor $(N-1)!$. Corollary~\ref{cor:reduced} would \emph{not} be true, even for $\mathfrak{S}_N$, if we fixed a particular Coxeter element $c$ and divided by the size of the Coxeter class. 
	
\noindent In a different vein, the prefactor $n!$ might be justified as interpreting reorderings of the terms in the factorizations of \eqref{eq:reduced}. Indeed, among real $W$, in $\mathfrak{S}_N$, $B_n$, and $H_3$, the so called \emph{coincidental} types, such reorderings always result in Coxeter elements. The general case however is not understood combinatorially (see Question~\ref{Q: multiple of n!}).
\end{remark}

\begin{remark}[Stronger formulas]\label{rem:strongerFormulas}
	In the case of the symmetric group, $W=\mathfrak{S}_{N}$, formulas~\eqref{eq:reduced} and~\eqref{eq:mainThm}, respectively the Kirchoff and Burman-Zvonkine theorems, are true for \emph{any} choice of weights ${\bf w}(\tau), \tau \in \mathcal{R}$, without reference to any tower structure. 
	As we mentioned, the same is \emph{not} true for general reflection groups --  already for the cyclic and dihedral groups for which computations can be done explicitly.
	We will see in this paper that Theorem~\ref{thm:main} has in fact a representation-theoretic meaning, and this will also explain why a stronger statement holds for $\mathfrak{S}_N$. It will also show that for some particular groups (in particular $B_n$) formulas exist under more general weight systems, at the cost of defining the ``Laplacian'' matrix over representations other than the reflection representation $\rho^{}_V$. See Remark~\ref{rem:strongerStatements} below, and Section~\ref{Sec: product formulas} for these generalizations.
\end{remark}

\subsection{Representation theoretic intepretation}
\label{sec:Intro-2}

It was already known to Hurwitz \cite{hurwitz-1901-paper} that, in any finite group $G$, generating functions of factorizations as the ones in \eqref{EQ: Jackson-formula} and \eqref{EQ: CS-formula} can be computed via representation theoretic methods. They can always be expressed as finite sums of exponentials, a technique known as the \emph{Frobenius Lemma} (see \S~\ref{Section: Rep-theoretic techniques on enumeration}). This procedure applies both to weighted and unweighted enumeration, but in the latter case it is particularly difficult to compute those exponentials (save from certain well behaved cases such as in the Burman-Zvonkine proof of \eqref{EQ: Burman-Zvonkine}). In Section~\ref{Section: Rep-theoretic techniques on enumeration} we build a general framework (Corol.~\ref{Corol: FAC^(A,P)_(G,c) final}) for the calculation of such generating functions when the weights arise from a filtration of a conjugacy class by a tower of subgroups as in \eqref{eq:towerintro} and \eqref{eq:weightintro}. Our approach is via the Gelfand-Tsetlin decomposition of representations of $G$, induced by such a tower.

We were particularly influenced by the pioneering work of Okounkov and Vershik \cite{OV} who used the standard tower of subgroups $
\mathfrak{S}_1\subset \mathfrak{S}_2 \subset \dots\subset\mathfrak{S}_n
$ to inductively develop the whole representation theory of the symmetric group $\mathfrak{S}_n$. A key ingredient of their work was that this tower has  simple (multiplicity-free) branching and therefore its Gelfand-Tsetlin algebra is a maximal commutative subalgebra of $\mathfrak{S}_n$. 
For other reflection groups, similar towers have been considered but are usually less enjoyable, for example in $E_8$ there exists no tower of \emph{reflection} subgroups that gives simple branching (see \S\ref{Section End: Okounkov-Vershik approach}).
However, in this paper we will see that the class of parabolic towers \eqref{eq:towerintro} \emph{do have} deep representation theoretic properties and deserve further interest even if they fail to extend the Okounkov-Vershik approach in its entirety.

To apply the Frobenius Lemma in our context one needs to know the values of irreducible characters of $W$ over large commutative subalgebras of the group algebra $\CC[W]$. These subalgebras are generated by natural analogs of the Jucys-Murphy elements, that we study in detail (Sec.~\ref{Section: generalized JM elts}). Given a tower $T$ as in~\eqref{eq:towerintro}, we introduce elements $J_{T,1},\dots,J_{T,n}$ in $\mathbb{C}[W]$, where $J_{T,i}$ is the sum of all reflections of $W$ that belong to $W_i\setminus W_{i-1}$:
\begin{align}\label{eq:GJMintro}
	J_{T,i} := \sum_{\tau \in \mathcal{R} \atop \tau \in W_i, \tau \not\in W_{i-1}} \tau \quad  \in \mathbb{C}[W].
\end{align}
We call these $J_{T,i}$'s {\color{blue} \emph{generalized Jucys-Murphy elements}} and show that they have an integer spectrum and natural bounds given by the Coxeter-theoretic data of the tower $T$ (Prop.~\ref{Prop: bounds on spectrum of JM elements}).

	The cornerstone of the Okounkov-Vershik approach to the representation theory of $\mathfrak{S}_n$ is that, for the standard tower of $\mathfrak{S}_n$, the Jucys-Murphy elements \emph{generate} the Gelfand-Tsetlin algebra, which in particular contains the center of the group algebra $\CC[\mathfrak{S}_n]$ and separates all conjugacy classes, hence all characters. For a general tower as~\eqref{eq:towerintro}, this is no longer the case, which leads us to the following question: which characters does the algebra $\CC[\bm J_T]:=\langle J_{T,i}\rangle$ separate? And what does it \emph{mean} for characters to be separated, or not? We will not answer these questions here, but our work shows that they deserve further exploration.

We say that two (virtual) characters of the group $W$ are {\color{blue} \emph{tower equivalent}} if they are equal on the subalgebra $\CC[\bm J_T]$ of $\CC[W]$ generated by the generalized Jucys-Murphy elements~\eqref{eq:GJMintro}, for \emph{any} choice of the  parabolic tower $T$.
As we will see, Theorem~\ref{thm:main} is a direct consequence of the following result, which is the main representation-theoretic product of this paper and, along with the theory of Lie-like elements \S~\ref{sec:LieLike}, the {\it de facto} explanation for the nice product forms of the enumerative formulas discussed so far.

\begin{intro_thm}[Tower equivalence, main result]\label{thm:equivalence}\ \newline
	Let $W \subset GL(V)$ be a well-generated complex reflection group, $T$ a parabolic tower, and $c\in W$ a Coxeter element. If $\widehat{W}$ denotes the collection of irreducible characters of $W$, we have that
	\begin{align}\label{eq:equivThm}
		\sum_{\chi\in\widehat{W}}\chi(c^{-1})\cdot \chi \towerequiv\sum_{k=0}^n(-1)^k\chi^{}_{\wedge^k(V)},
	\end{align}
	where $\towerequiv$ denotes tower equivalence, and where $\chi^{}_{\wedge^k V}$ is the character of the $k$-th exterior power of the reflection representation  $W\subset GL(V)$. 
\end{intro_thm}

This equivalence sheds a new light on the numerological formulas that characterize the world of complex reflection groups, and that are still not well understood.
It is interesting to recall here the history of the 50 years old Arnol'd-Bessis-Chapoton formula~\eqref{eq:introDeligne} stated earlier. It received a case-free proof only very recently, for Weyl groups by Michel~\cite{Michel} and for arbitrary real groups by the second author~\cite{Douvr}, with both papers proving in fact the character-theoretic identities, equivalent to the more general formula \eqref{EQ: CS-formula}, that Chapuy-Stump \cite{CS} had proved relying on the classification. As it turns out, both these proofs also involve, in different forms and in a rather ad-hoc way, the exterior products $\wedge^k V$ (more explicitly in the first, see \cite[Lemma~5]{Michel}). Adding to this, we show in Prop.~\ref{Prop: formal stronger tower equivalence} that Theorem~\ref{thm:equivalence} is actually equivalent to a seemingly stronger version that specifies tower equivalences for each exterior power $\wedge^k(V)$. It can therefore be seen more as a statement about the whole exterior algebra of $V$.

These two case-free proofs mentioned above still contain a lot of mystery and for example they do not explain -- at least that we know of -- our more general results. We hope that Theorem~\ref{thm:equivalence} clearly identifies the role of the exterior powers $\wedge^k(V)$, even if our proof does not fully explain it, as it relies on the classification. 
We also hope that the study of the generalized Jucys-Murphy elements can lead to interesting progress on these questions, see also Section~\ref{Section End: Okounkov-Vershik approach}.

\begin{remark}[Stronger statements]\label{rem:strongerStatements}
	In the case of the symmetric group, $W=A_n$, it is well known that the equivalence in Theorem~\ref{thm:equivalence} is an equality, as only ``hook'' characters do not vanish on Coxeter elements, and they coincide with exterior powers. As we will see, this directly implies that Theorem~\ref{thm:main} holds for \emph{any} choice of weights (Remark~\ref{rem:strongerFormulas}). It also turns out that for certain groups, the equivalence in~\eqref{eq:equivThm} can be strengthened (so that the virtual characters agree on larger subalgebras) by replacing the representation $\wedge^k(V)$ in the right-hand side by other representations $\wedge^k(V')$; in particular, it becomes an equality for the groups $B_n$. This leads to versions of Theorem~\ref{thm:main} with finer weights, but in which the Laplacian has to be defined over the less natural representations $V'$. See Section~\ref{Sec: product formulas}.
\end{remark}

\subsection{Other applications of the $W$-Laplacian in Coxeter combinatorics}
\label{sec:Intro-3}

The $W$-Laplacian matrix $L_W$ was constructed (Defn.~\ref{Defn:Intro-Laplacian}) as a sum of rank-1 operators. Burman et al. \cite{BPT} have developed a general framework for expressing the characteristic polynomials of such (and more general) matrices. In the case of the usual graph Laplacian $L_G$, this corresponds precisely to the so called Matrix-Forest theorem, where the coefficients of $\op{det}(x+L_G)$ count rooted forests in $G$. In the case of our $W$-Laplacian this can be interpreted as counting $\CC$-independent subsets of hyperplanes, which could be seen as $W$-trees and $W$-forests, weighted by a Grammian statistic (see Lemma~\ref{Lem:Burman M-T}). 

Building on the work of \cite{BPT} we prove, \emph{without relying on the classification}, a parabolic recursion (Prop.~\ref{prop:arr_rec}) for the $W$-Laplacian (that makes sense for arbitrary hyperplane arrangements and their Laplacians) which is the key ingredient for the following two theorems. The first (which however does not have a uniform proof as it also relies on Corol.~\ref{cor:reduced}) is an analogue of the Matrix Forest theorem for reflection groups, where however the combinatorial object that appears is reflection factorizations instead of trees. It turns out that the spectrum of the $W$-Laplacian encodes both \emph{arbitrary length} and \emph{reduced} Coxeter factorizations at the same time (see Remark~\ref{Rem: W-M-F thm}).

\begin{intro_thm}[$W$-Matrix-forest theorem, see Theorem~\ref{Thm: WMFT} page~\pageref{Thm: WMFT}]\label{thm:forestIntro}\ \newline
For a well generated complex reflection group $W$, its intersection lattice $\mathcal{L}_W$ and a parabolic tower $T$, the characteristic polynomial of its $T$-weighted $W$-Laplacian is given via
$$\op{det}\big(x+L_W^T(\bm\omega)\big)=\sum_{\tau_1\cdots\tau_{n-k}=c^{}_X\atop  X\in\mathcal{L}_W,\ k=\op{dim}(X)}|C^{}_{W_X}(c^{}_X)|\cdot\mathbf{w}^{}_T(\tau_1)\cdots\mathbf{w}_T^{}(\tau_{n-k})\cdot\dfrac{x^k}{(n-k)!},$$ where the sum is over all (reduced) reflection factorizations of \emph{any} Coxeter element $c^{}_X$ of \emph{any} parabolic subgroup $W_X$ of $W$ and $C_{W_X}(c_X)$ is the centralizer of $c_X$ in $W_X$.
\end{intro_thm}

The second important application (now with a completely uniform proof) of the parabolic recursion for $L_W$ is an identity between the Coxeter number of $W$ and those of its parabolic subgroups. For a reducible reflection group $W=W_1\times \cdots \times W_m$, we define the {\color{blue} \emph{multiset $\{h_i(W)\}$ of Coxeter numbers}} of $W$ as the collection of each Coxeter number $h_i$ of $W_i$ taken with multiplicity equal to $\op{rank}(W_i)$. For instance, we have $\{h_i(\mathfrak{S}_4\times B_3)\}=\{6,6,6,4,4,4\}$. The following is Theorem~\ref{Prop: identity among coxeter numbers} page~\pageref{Prop: identity among coxeter numbers}.

\begin{intro_thm}
For any irreducible complex reflection group $W$ with Coxeter number $h$, we have
$$(h+x)^n=\sum_{X\in\mathcal{L}_W}\prod_{i=1}^{\op{codim}(X)}h_i(W_X)\cdot x^{\op{dim}(X)},$$ where $\{ h_i(W_X)\}$ denotes the \emph{multiset} of Coxeter numbers of $W_X$ as described previously.
\end{intro_thm}

The Grammian statistic that appears in the  abstract Matrix Forest theorems of \cite{BPT} (see Lemma~\ref{Lem:Burman M-T}) behaves particularly well for Weyl groups $W$. It allows us in fact to calculate (also with a uniform proof) the volumes of root zonotopes via new formulas involving the Coxeter numbers of reflection subgroups of $W$, arguably simpler than existing works, see Section~\ref{Sec: root polytopes}. Still in the setting of Weyl groups, the $W$-Laplacian $L_W$ seems to be an excellent object to relate $W$-trees and Coxeter factorizations, generalizing the situation in type $A$. Remarkably, for $W$-trees, the role of the Coxeter number $h$ is in some sense replaced by the connection index of $W$, see Section~\ref{sec: Trees vs Facns}.

\subsection{Plan of the paper}
The paper is organized as follows. Section~\ref{Section: Rep-theoretic techniques on enumeration} recalls and extends some standard techniques of enumeration of factorizations in groups using representation theory, circling around ideas related to the ``Frobenius lemma''. In Section~\ref{sec:JM} we present the notions we need about complex reflection groups, we introduce parabolics towers,  generalized Jucys-Murphy elements, the $W$-Laplacian and its group-algebra counterpart. 
In Section~\ref{sec:exceptional} we use the previous results to show how our main theorem can be checked algorithmically provided one has access to the character table of the group, which we used in the (computer aided) proof for exceptional groups.
In Section~\ref{sec:combinatorial} we describe an explicit combinatorial rule to compute the eigenvalues of the $W$-Laplacian for a given group and parabolic tower.
In Section~\ref{sec:LieLike} we show the equivalence between Theorems~\ref{thm:main} and~\ref{thm:equivalence} via the concept of Lie-like elements due to Burman and Zvonkine. In Section~\ref{sec:infiniteFamilies} we prove the main result for the infinite families $G(r,1,n)$ and $G(r,r,n)$. This relies on an induction based on parabolic subgroups.
In Section~\ref{Section: Matrix-Forest}, we study the Laplacian for hyperplane arrangements, and use it to prove the $W$-matrix forest theorem, the identities for parabolic Coxeter numbers, and more. 
In Section~\ref{sec:further} we discuss further directions of research, natural questions, and comments. Finally, the appendices deal with certain applications of the Littlewood-Richardon rule which are used in the computations of Section~\ref{sec:infiniteFamilies}.

\section{Representation theoretic techniques on enumeration}
\label{Section: Rep-theoretic techniques on enumeration}

A well-trodden path in the enumeration of factorizations in groups is via the following lemma, that goes back to Frobenius and the beginnings of representation theory \cite{frobenius-collected-works}. It provides an answer in terms of group characters for factorization problems where each factor belongs to a prescribed conjugacy class (or a union of them). A proof  can be found in \cite[App.~A.1.3]{LZ-graphs-on-surfaces}.

\begin{lemma}[The Frobenius lemma] \label{Thm: Frobenius lemma}
Let $G$ be a finite group, $\mathcal{C}\subset G$ a conjugacy class, and $\mathcal{A}_i, i=1\dots \ell$, subsets of $G$ that are closed under conjugation. Then the number of factorizations $\sigma_1\cdots \sigma_{\ell}=g$, where $\sigma_i\in \mathcal{A}_i$ and  $g\in\mathcal{C}$, is given by$$
\dfrac{|\mathcal{C}|}{|G|}\sum_{\chi\in \widehat{G}}\chi(1)\cdot\chi(g^{-1})\cdot\dfrac{\chi(\mathcal{A}_1)}{\chi(1)}\cdots\dfrac{\chi(\mathcal{A}_{\ell})}{\chi(1)},$$ where $\widehat{G}$ denotes the (complete) set of irreducible characters of $G$, $g$ is any element of $\mathcal{C}$, and $\chi(\mathcal{A}_i):=\sum_{a\in \mathcal{A}_i}\chi(a)$.  
\end{lemma}

When all factors $\sigma_i$ are chosen from the same (closed under conjugation) set $\mathcal{A}$ the Frobenius lemma gives a particularly nice answer to the enumeration problem. If we write $$\fac^{\mathcal{A}}_{G,g}(t):=\sum_{\ell\geq 0}\dfrac{t^{\ell}}{\ell!}\cdot\#\{(\sigma_1,\cdots,\sigma_{\ell},g)\in \mathcal{A}^{\ell}\times\mathcal{C}\ |\ \sigma_1\cdots\sigma_{\ell}=g\},$$ for the exponential generating function, then the previous theorem allows us to compute its coefficients and rewrite the sum as \begin{equation} \fac_{G,g}^{\mathcal{A}}(t)=\dfrac{|\mathcal{C}|}{|G|}\cdot \sum_{\chi\in\widehat{G}}\chi(1)\cdot \chi(g^{-1})\cdot \op{exp}\big( t\widetilde{\chi}(\mathcal{A})\big)\label{EQ: Frobenious without weights},\end{equation} where $\widetilde{\chi}(\mathcal{A})$ denotes the normalized trace $\frac{\chi(\mathcal{A})}{\chi(1)}$. This observation, that the generating function becomes a finite sum of exponentials, goes all the way back to Hurwitz in 1901 \cite[\S\,3:(15)]{hurwitz-1901-paper}.

\subsection{A Frobenius Lemma for weighted enumeration}
\label{Subsection: A frobenius Lemma for weighted enumeration}

We recall here a variant of the Frobenius lemma, considered already in \cite[Proof of Prop.~2.1]{BZ}, which allows for the weighted enumeration of factorizations. 

\begin{definition}\label{Defn: weight systems}
Let $\mathcal{A}$ be a subset\footnote{We will not need $\mathcal{A}$ to be closed under conjugation until \S~\ref{Section:weight function by group tower}.} of $G$ and $P=\{P_i\}_{(i\in I)}$ a set partition of $\mathcal{A}$. A {\it \color{blue} weight system} indexed by $P$ is an assignment $\wt:\mathcal{A}\rightarrow \CC,$ such that the weights $\wt(\sigma)\in\CC,\ \sigma\in\mathcal{A}$, depend only on the block $P_i$ that contains $\sigma$. We write $(\omega_i)_{i\in I}$ for the tuple of these weights (so that $\wt(\sigma)=\omega_i$ if $\sigma\in P_i$) and treat them as complex variables (but see also Rem.~\ref{Rem: scalar weights vs variables}).
\end{definition}

We are interested in the enumeration of factorizations $\sigma_1\cdots \sigma_{\ell}=g$ in a group $G$ where now each term $\sigma_i$ is weighted via a system $\wt$ as above. The exponential generating function of such weighted factorizations is then given by:

\begin{equation}\fac_{G,\mathcal{C}}^{\mathcal{A},P}(t,\bm\omega):=\sum_{\ell\geq 0}\dfrac{t^{\ell}}{\ell!}\cdot\sum_{(\sigma_1,\sigma_2,\dots,\sigma_\ell,g) \in \mathcal{A}^\ell \times \mathcal{C}\atop \sigma_1\sigma_2\dots\sigma_\ell =g} \wt(\sigma_1)\cdots\wt(\sigma_{\ell}).\label{EQ: Defn of FAC, weighted, general}\end{equation} The notation is meant to suggest that we want to treat $\fac$ as a function on the weights $\bm\omega:=(\omega_i)_{i\in I}$. Since its coefficients (as a formal series on $t^{\ell}/\ell!$) are sums of monomials in the $\omega_i$'s, we see in fact that $\fac\in\QQ[\bm\omega][[t]]$.

\begin{notation}\label{Not. chi(series)}
To simplify our formulas, we extend characters $\chi$ to the power series $\CC[G][[t]]$ by defining $\displaystyle\chi\big(\sum a_i\cdot t^i\big):=\sum \chi(a_i)\cdot t^i$, for any coefficients $a_i$ from the group algebra $\CC[G]$.
\end{notation}

\begin{lemma}\label{Thm: Frobenius a la Guillaume}
With notation as above, the enumeration of weighted factorizations \eqref{EQ: Defn of FAC, weighted, general} is given by
$$\fac_{G,\mathcal{C}}^{\mathcal{A},P}(t,\bm\omega)=\dfrac{|\mathcal{C}|}{|G|}\cdot \sum_{\chi\in \widehat{G}}\chi(g^{-1})\cdot \chi\big(e^{t\mathbf{A}(\bm\omega)}\big),$$ where $\mathbf{A}(\bm\omega):=\sum_{\sigma\in\mathcal{A}}\wt(\sigma)\cdot\sigma$ belongs to the group algebra $\CC[G]$ and $g$ is any choice from $\mathcal{C}$.
\end{lemma}

\begin{proof}
Notice that if $[g](\alpha)$ denotes the coefficient of $g$ in an element $\alpha$ of the group algebra $\CC[G]$, we can express the weighted enumerations as $$\sum_{(\sigma_1,\sigma_2,\dots,\sigma_\ell,g) \in \mathcal{A}^\ell \times \mathcal{C}\atop \sigma_1\sigma_2\dots\sigma_\ell =g} \wt(\sigma_1)\cdots\wt(\sigma_{\ell})=\sum_{g\in\mathcal{C}}\ [g]\big(\mathbf{A}(\bm\omega)^{\ell}\big)=\sum_{g\in\mathcal{C}}[\mathbbl{1}]\big(\mathbf{A}(\bm\omega)^{\ell}\cdot g^{-1}\big),$$ where in the second equality $\mathbbl{1}$ denotes the identity in $G$. The elements $\{g^{-1}\ |\ g\in\mathcal{C}\}$ sum up to the {\it central}  $\mathbf{C}^{-1}:=\sum_{g\in\mathcal{C}}g^{-1}$ and we rewrite the last term as $[\mathbbl{1}]\big(\mathbf{A}(\bm\omega)^{\ell}\cdot\mathbf{C}^{-1}\big)$.

One can detect the coefficient $[\mathbbl{1}](\alpha)$ as the trace of $\alpha$, normalized by $1/|G|$, under the regular representation $\CC[G]$. This is because in $\CC[G]$ all group elements apart from $\mathbbl{1}$ act as traceless permutation matrices (they have no fixed points). That is, we have
\begin{align*}[\mathbbl{1}]\big(\mathbf{A}(\bm\omega)^{\ell}\cdot\mathbf{C}^{-1}\big)&=\dfrac{1}{|G|}\cdot\op{Tr}_{\CC[G]}\big(\mathbf{A}(\bm\omega)^{\ell}\cdot\mathbf{C}^{-1}\big)\\
&=\dfrac{1}{|G|}\sum_{\chi\in\widehat{G}}\chi(1)\cdot \chi\big(\mathbf{A}(\bm\omega)^{\ell}\cdot\mathbf{C}^{-1}\big)=\dfrac{|\mathcal{C}|}{|G|}\sum_{\chi\in\widehat{G}}\chi(g^{-1})\cdot\chi\big(\mathbf{A}(\bm\omega)^{\ell}\big),
\end{align*} where for the second line we are decomposing the group algebra as $\CC[G]\cong\sum_{\chi\in\widehat{G}}\chi(1)\cdot U_{\chi}$ with $U_{\chi}$ the representation afforded by $\chi$, and use that the central element $\mathbf{C}^{-1}$ acts on $U_{\chi}$ as multiplication by the scalar $|\mathcal{C}|\cdot\chi(g^{-1})/\chi(1)$ (for any $g\in\mathcal{C}$).

The generating function \eqref{EQ: Defn of FAC, weighted, general} becomes then $$\fac_{G,\mathcal{C}}^{\mathcal{A},P}(t,\bm\omega):=\dfrac{|\mathcal{C}|}{|G|}\sum_{\ell\geq 0}\dfrac{t^{\ell}}{\ell!}\cdot\sum_{\chi \in \widehat{G}}\chi(g^{-1})\cdot \chi\big(\mathbf{A}(\bm\omega)^{\ell}\big),$$ which is exactly the statement of the lemma if we notice by Notation~\ref{Not. chi(series)} that we write \begin{equation}\vspace{-0.3cm}\chi\big(e^{t\mathbf{A}(\bm\omega)}\big)=\sum_{\ell\geq 0}\chi\big(\mathbf{A}(\bm\omega)^l\big)\cdot\dfrac{t^l}{l!}.\label{EQ: chi(exp)}\end{equation}\end{proof}

\begin{remark}\label{Rem: scalar weights vs variables}
We could have treated the $\omega_i$'s as indeterminates and considered the group ring $\CC[\bm\omega][G]$ instead. This would not affect the representation theoretic arguments (the group algebra $\CC[G]$ is \emph{already} split, see \cite[\S~7.4]{geck-pfeiffer}). This dual intepretation of the weights as complex scalars or indeterminates is valid throughout the rest of the paper too.
\end{remark}

For the actual calculation of the generating function the following form is perhaps more useful. We write $\op{Spec}_{\chi}\big(\mathbf{A}(\bm\omega)\big):=\{\theta_1(\bm \omega),\cdots,\theta_s(\bm\omega)\}$ for the {\it multiset} of (generalized) eigenvalues $\theta_j(\bm\omega)$ of $\mathbf{A}(\bm\omega)$ on the representation afforded by $\chi$ (so that $s=\op{dim}(\chi)$). Then \eqref{EQ: chi(exp)} becomes
$$
\chi(e^{t\mathbf{A}(\bm\omega)})=\sum_{\ell\geq 0}\dfrac{t^{\ell}}{\ell!}\cdot\sum_{\theta_j(\bm\omega)\in\op{Spec}_{\chi}(\mathbf{A}(\bm\omega))}\theta_j(\bm\omega)^{\ell}=\sum_{\theta_j(\bm\omega)\in\op{Spec}_{\chi}(\mathbf{A}(\bm\omega))}e^{t\theta_j(\bm\omega)},
$$which then immediately gives:

\begin{corollary}\label{Corol: FAC, summing over eigenvalues}
The generating function \eqref{EQ: Defn of FAC, weighted, general} for the weighted enumeration is as follows:
$$\fac_{G,\mathcal{C}}^{\mathcal{A},P}(t,\bm\omega)=\dfrac{|\mathcal{C}|}{|G|}\cdot\sum_{\chi\in\widehat{G}}\chi(g^{-1})\sum_{\theta_j(\bm\omega)\in\op{Spec}_{\chi}(\mathbf{A}(\bm\omega))}e^{t\theta_j(\bm\omega)}.$$
\end{corollary}

The above expression, where a finite sum of exponentials encodes the answer to an enumeration problem, should be seen as a weighted analog of \eqref{EQ: Frobenious without weights}. It reveals however a computational difficulty one faces when using this version of the Frobenius lemma, as opposed to Lemma~\ref{Thm: Frobenius lemma}:

 In general, the eigenvalues $\theta_j(\bm\omega)$ are algebraic functions on the parameters $(\omega_i)_{i\in I}$ and they might be very hard to calculate explicitly. In \S\ref{Section: weight function gives commutative subalgebra} and \S\ref{Section:weight function by group tower} we study special weight systems for which this is remedied (see also Section~\ref{sec:LieLike} and particularly Lemma~\ref{lemma:LieLikeEigenvalues}).

\subsection{When the partition $P$ determines a commutative subalgebra}
\label{Section: weight function gives commutative subalgebra}

The partition $P$ that indexes the weight system $\wt$ as in Defn.~\ref{Defn: weight systems} determines a tuple $(J_i)_{i\in I}$ of elements in the group algebra $\CC[G]$ as $J_i:=\sum_{\sigma\in P_i}\sigma$. These elements generate a subalgebra \begin{equation}\CC[\bm J]:=\CC[J_1,\cdots,J_r]\subset \CC[G],\label{EQ: C[J] subalgebra of C[G]}\end{equation} which in particular contains $\mathbf{A}(\bm\omega)=\sum_{i\in I}\omega_i J_i$ for any choice of variables $\omega_i$. 

When the algebra $\CC[\bm J]$ is commutative, then for any representation $\rho^{}_U:\CC[G]\rightarrow U$ and via use of the Schur decomposition theorem, its elements can be simultaneously triangularized. That is, there is some basis of $U$ for which all matrices $\rho^{}_U(A),\ A\in\CC[\bm J]$, are upper-triangular.  

In this setting the eigenvalues of an element $A$ acting on $U$ are given by linear functionals\footnote{These are usually called {\it weights} as in the theory of Lie algebras, but they will be left innominate here to avoid confusion with the weight systems $\wt$.} $\bm\theta :\CC[\bm J]\rightarrow \CC$ that record some fixed diagonal entry of the matrix $\rho^{}_U(A)$. We will denote by $\op{Spec}_U(\CC[\bm J])$ the {\it multiset} of such functionals associated to the representation $U$. They are in fact multiplicative (since the matrices $\rho^{}_U(A)$ are upper-triangular) and are therefore determined by their values on the generators $J_i$. 

That is, we write $\bm\theta=(\ell_i)_{i\in I}$ where $\ell_i:=\bm\theta(J_i)$ and if $A\in\CC[\bm J]$ is expressed as $A=f(\bm J)$ for some polynomial $f$, then the multiset $\op{Spec}_U(A)$ of its eigenvalues on $U$ is given by $$\op{Spec}_U(A)=\big\{f(\ell_1,\dots,\ell_r)\ |\ \bm\theta\in\op{Spec}_U(\CC[\bm J])\big\}.$$ In particular, the eigenvalues of $\mathbf{A}(\bm \omega)=\sum_{i\in I}\omega_iJ_i$ on $U$ are of the form $\bm\theta\big(\mathbf{A}(\bm \omega)\big)=\sum_{i\in I}\ell_i\omega_i$ for each of the $\op{dim}(U)$-many functionals $\bm\theta$.

This is a particularly nice setting for Corol.~\ref{Corol: FAC, summing over eigenvalues}, since now the eigenvalues $\theta_j(\bm\omega)$ are linear functions on the weights $\omega_i$ (instead of algebraic). To compute them, one would need in each irreducible representation $\chi\in\widehat{G}$ to simultaneously triangularize the matrices $\rho^{}_U(J_i)$. This might still be very difficult to do, or too slow for a computer experiment; in the next section we study even more special weight systems where this calculation comes down to the character theory of $G$.   

\subsection{When the partition $P$ is prescribed by a tower of subgroups}
\label{Section:weight function by group tower}

An easy way to guarantee the commutativity of the algebra $\CC[\bm J]$ of \eqref{EQ: C[J] subalgebra of C[G]} is to partition $\mathcal{A}$ according to a filtration of the group $G$ by a tower of subgroups. From this point on, we will further require that the underlying set $\mathcal{A}$ is closed under conjugation.

\begin{definition}\label{Defn: partition P via tower of subgroups}
Let $\mathcal{A}$ be a subset of some group $G$, closed under conjugation, and let $$T:=\big(\{\mathbbl{1}\}=G_0\leq G_1\leq\cdots\leq G_n=G\big)$$ denote a chain of subgroups. The tower $T$ induces a filtration of $\mathcal{A}$ into sets $\mathcal{A}_i:=\mathcal{A}\bigcap G_i$ which are closed under conjugation in $G_i$. We denote by $P_T$ the associated partition of the set $\mathcal{A}$ into parts $P_i:=\mathcal{A}_i\setminus\mathcal{A}_{i-1}$. If $\mathbbl{1}\notin\mathcal{A}$, we disregard the empty part $P_0$ and consider only $i=1,\dots,n$.
\end{definition}

In this setting the assumption of commutativity for the algebra $\CC[\bm J]$ of \eqref{EQ: C[J] subalgebra of C[G]} holds automatically. Indeed, since both $\mathcal{A}_i$ and $\mathcal{A}_{i-1}$ are closed under conjugation by elements of $G_{i-1}$, we will have by definition that $g^{-1}P_ig=P_i$, for any $g\in G_{i-1}$. Recalling the definition $J_i=\sum_{\sigma\in P_i}\sigma$, this means also that $g^{-1} J_ig=J_i$ for all $g\in P_j\subset G_{i-1}$ when $j<i$, which in turn forces $J_iJ_j=J_jJ_i$ for all $i,j$.

Far more than ensuring commutativity of the $J_i$'s, this setting allows us to relate the functionals $\bm\lambda$ of the previous section with character-theoretic information (and thus produce an even easier to compute version of Lemma~\ref{Thm: Frobenius a la Guillaume}):

\subsubsection{The Gelfand-Tsetlin decomposition}

We follow here the presentation as in \cite{OV} but see also \S~\ref{Section End: Okounkov-Vershik approach}. A source with a longer elaboration of the necessary representation theory is \cite{OV-book}.

Given a group $G$ and a tower $T$ of subgroups as in Defn.~\ref{Defn: partition P via tower of subgroups}, the {\color{blue} \it branching graph} or {\it Bratteli diagram} of $G$ with respect to $T$ is a directed multi-graph defined by the following data. Its vertex set is indexed by the (disjoint union of the) irreducible characters of the groups $G_i$. Two such vertices $\psi\in\widehat{G_{i-1}}$ and $\chi\in\widehat{G_i}$ are connected by $m_{\psi,\chi}$ directed edges from $\psi$ to $\chi$ where $$m_{\psi,\chi}=\big(\psi,\big\downarrow^{G_i}_{G_{i-1}}\chi\big)_{G_{i-1}}$$ denotes the multiplicity by which $\psi$ appears in the restriction of the character $\chi$ on $G_{i-1}$. The notation $\psi\nearrow\chi$ indicates that $\psi$ and $\chi$ are characters of {\it consecutive} groups $G_{i-1}$ and $G_i$ in the tower $T$, and that the multiplicity $m_{\psi,\chi}$ is positive.

We can decompose the representation $\rho_{\chi}:G\rightarrow\op{GL}(V_{\chi})$ associated to a character $\chi\in\widehat{G_i}$ as \begin{equation}V_{\chi}=\bigoplus_{\psi\nearrow\chi}m_{\psi,\chi}V_{\psi},\label{Eq: decompose chi into isotypic components}\end{equation} where $m_{\psi,\chi}V_{\psi}$ denotes the $\psi$-isotypic component of $V_{\chi}$. This decomposition is canonical in the sense that there is a uniquely determined central projection operator $$E_{\psi}:=\dfrac{\psi(1)}{|G|}\sum_{g\in G}\overline{\psi}(g)\rho_{\chi}(g)\ \in\op{End}(V_{\chi}),$$ such that $m_{\psi,\chi}V_{\psi}$ may be defined as its image $E_{\psi}(V_{\chi})$.

\begin{notation}\label{Not: chi restricted down T}
We denote by $\towerchi$ a tuple of characters $(\chi_0,\chi_1,\cdots,\chi_n)$ with $\chi_i\in\widehat{G_i}$ and we write $$\op{Res}^{}_T(\chi):=\{\towerchi:\ \chi_i\nearrow \chi_{i+1}\text{ and }\chi_n=\chi\},$$ for the set (ignoring multiplicities) of chains $\towerchi$ that may appear as we restrict $\chi$ down the tower $T$.
\end{notation}

The canonical projection into isotypic components as in \eqref{Eq: decompose chi into isotypic components} can be used to construct a decomposition of $V_{\chi}$ into spaces $V_{\towerchi}$ indexed by the chains $\towerchi\in\op{Res}^{}_T(\chi)$ as defined above. Start by writing $V_{\towerchi}(n):=V_{\chi}$ and inductively define $V_{\towerchi}(i)$ as the $\chi_i$-isotypic component of $V_{\towerchi}(i+1)$. At the last level, the spaces $V_{\towerchi}:=V_{\towerchi}(0)$ satisfy \begin{equation}V_{\chi}=\bigoplus_{\towerchi\in\op{Res}_T^{}(\chi)}V_{\towerchi}.\label{Eq: Gelfand-Tsetlin decomposition}\end{equation}

We call the expression above the {\color{blue}\it Gelfand-Tsetlin decomposition} of $V_{\chi}$ with respect to the tower $T$. Notice that at each step $i$ of the inductive construction the multiplicity of $V_{\chi_{n-i}}$ in $V_{\towerchi}(n-i)$ is given by the product $\prod_{j=n-i}^{n-1}m_{\chi_j,\chi_{j+1}}$. That is, if we define \begin{equation}\op{mult}(\towerchi):=\prod_{i=1}^n m_{\chi_{i-1},\chi_{i}}=\prod_{i=1}^n\big( \chi_{i-1},\big\downarrow^{G_i}_{G_{i-1}}\chi_i\big)_{G_{i-1}},\label{Eq: mult(towerchi)}\end{equation} and since we have $\op{dim}(V_{\chi_0})=1$ (because $G_0=\{\mathbbl{1}\}$), we get that $\op{dim}(V_{\towerchi})=\op{mult}(\towerchi)$. In more combinatorial terms, \eqref{Eq: mult(towerchi)} counts the number of chains equal to $\towerchi$ in the branching graph for $T$. 

The significance of the Gelfand-Tsetlin decomposition \eqref{Eq: Gelfand-Tsetlin decomposition} in our setting is that it provides simultaneous eigenspaces for the elements $J_i:=\sum_{\sigma\in P_i}\sigma$. Indeed, consider in the group algebra $\CC[G]$ the partial sums $\mathbf{A}_i:=J_1+\cdots+J_i$ so that $J_i=\mathbf{A}_i-\mathbf{A}_{i-1}$. Each $\mathbf{A}_i$ is a central element of $\CC[G_i]$ and therefore acts as scalar multiplication by $\widetilde{\chi}_{i}(\mathcal{A}_i)$ on each space $V_{\towerchi}(i)$. 

Finally, we have by construction $V_{\towerchi}\subset V_{\towerchi}(i)$ for all $i$, which means that \eqref{Eq: Gelfand-Tsetlin decomposition} gives a decomposition of $V_{\chi}$ in simultaneous eigenspaces of the $\mathbf{A}_i$ and thus by definition, also of the $J_i$. In particular, the elements $J_i$ are simultaneously \emph{diagonalizable}. Each $J_i$ acts as scalar multiplication by $\widetilde{\chi}_i(\mathcal{A}_i)-\widetilde{\chi}_{i-1}(\mathcal{A}_{i-1})$ on the spaces $V_{\towerchi}$ (these are the numbers $\ell_i$ of \S\ref{Section: weight function gives commutative subalgebra}) and since $\bm A(\bm\omega):=\sum_{i=1}^n\omega_iJ_i$, we have completely determined its spectrum over $V_{\chi}$:

\begin{equation}
\op{Spec}_{\chi}\big(\bm A(\bm\omega)\big)=\Bigg\{\sum_{i=1}^n\omega_i\cdot\Big(\widetilde{\chi}_i(\mathcal{A}_i)-\widetilde{\chi}_{i-1}(\mathcal{A}_{i-1})\Big)\ \text{with multiplicity}\ \op{mult}(\towerchi)\ |\ \towerchi\in\op{Res}_T^{}(\chi)\Bigg\}
\label{Eq: Spec_chi(A(w)) description}\end{equation}

\noindent After Corol.~\ref{Corol: FAC, summing over eigenvalues}, this immediately gives the easiest-to-compute version of Lemma~\ref{Thm: Frobenius a la Guillaume}:

\begin{corollary}\label{Corol: FAC^(A,P)_(G,c) final}
If the partition $P$ is induced by a tower of groups $T$ as in Defn.~\ref{Defn: partition P via tower of subgroups} and $\mathcal{A}$ is closed under conjugation, then with Notation~\ref{Not: chi restricted down T} the enumeration of factorizations \eqref{EQ: Defn of FAC, weighted, general} is given by:
$$\fac_{G,\mathcal{C}}^{\mathcal{A},P}(t,\bm\omega)=\dfrac{|\mathcal{C}|}{|G|}\cdot\sum_{\chi\in\widehat{G}}\chi(g^{-1})\sum_{\towerchi\in \op{Res}^{}_T( \chi)}\op{mult}(\towerchi)\cdot\op{exp}\Big( t\cdot\sum_{i=1}^n\big(\widetilde{\chi}_i(\mathcal{A}_i)-\widetilde{\chi}_{i-1}(\mathcal{A}_{i-1})\big)\cdot\omega_i\Big).$$
\end{corollary}

\begin{remark}
This general formula can easily be used via a computer for arbitrary groups $G$ and towers $T$ given (more or less) complete access to the character tables. We hope the reader might enjoy experimenting in families of groups where such tower stratifications are natural (for instance $\op{GL}_n(F_q)$ seems a great first							 candidate after the work of \cite{LRS}).
\end{remark}

\section{Reflection groups, the $W$-Laplacian, and tower equivalence}
\label{sec:JM}

After discussing in the previous section a general framework for enumerating factorizations in groups, we go on now to introduce our particular objects of study. These are the complex reflection groups, their Coxeter elements and their reflection factorizations. The reader may consult the classical references \cite{HUM,kane-book-reflection-groups,broue-book-braid-groups,lehrer-taylor-book} for any material that is not provided here.

A finite subgroup $W\leq GL(V)$ acting on some space $V\cong\CC^n$ is called a {\color{blue}\it complex reflection group} of rank $n$ if it is generated by {\it unitary reflections}. These are ($\CC$-linear) maps $t$ whose fixed spaces are hyperplanes; that is, $\op{codim}(V^t)=1$. When there is no non-trivial linear subspace of $V$ stabilized by the action of the group, we say that $W$ is {\it irreducible}. 

Shephard and Todd \cite{shephard-and-todd} classified irreducible complex reflection groups into an infinite 3-parameter family $G(r,p,n)$ and 34 exceptional cases $G_4$ to $G_{37}$. Any such group or rank $n$ can always be generated by either $n$ or $n+1$ reflections (though there is no conceptual explanation of this fact). The groups in the first category are called {\color{blue}\it well generated} and they are generally better understood; in particular they share many properties with the subclass of real reflection groups.

\subsubsection*{Coxeter elements as Springer regular elements}

In the real case, there is indeed a geometrically defined generating set of $n$ reflections, called the simple generators, whose combinatorics determine the group. The product of these (in any order) is called a Coxeter element and is denoted by $c$, after the famous geometer who studied its properties. In particular, Coxeter observed \cite[Thm.~11,~16]{coxeter-annals} that all possible elements $c$ are conjugate and that their order $h$ determines the number $N$ of reflections of the group, via the formula $hn=2N$.
Steinberg \cite{steinberg-finite-reflection-groups} gave a uniform proof of this formula, by constructing a plane on which $c$ acts as a rotation of $2\pi /h$ radians. The key property of this so called {\it Coxeter plane} was that it intersects all reflection hyperplanes $H$ transversally. This can be rephrased by saying that in the complexified ambient space, the $e^{\pm 2\pi i/h}$-eigenvectors of $c$ are not contained in any $H$. 

In the general case (real or complex), a vector $v$ that lies outside any reflection hyperplane is called {\it \color{blue} regular} and an element of the group will be called a {\it \color{blue} $\zeta$-regular} element if it has a regular $\zeta$-eigenvector \cite{springer-regular-elements}. As it happens, all $\zeta$-regular elements are conjugate and their order equals the order of $\zeta$ \cite[Corol.~11.25]{lehrer-taylor-book}. If $\mathcal{R}$ and $\mathcal{R}^*$ respectively denote the set of reflections and reflecting hyperplanes of $W$, the number \begin{equation}h:=\dfrac{|\mathcal{R}|+|\mathcal{R^*}|}{n} \label{Eq: Coxeter number h} \end{equation} is an integer (see Prop.~\ref{Prop: coxeter numbers are integers}) and we call it the {\it \color{blue} Coxeter number} of $W$. The following definition is thus meant to generalize the properties of $c$ in real groups (where $|\mathcal{R}|=|\mathcal{R^*}|=N$ and $h=2N/n$).

\begin{definition}\label{Defn: Coxeter element}
If $h$ is the Coxeter number of a complex reflection group $W$ as above, we define a {\it \color{blue} Coxeter element} $c$ of $W$ to be a $e^{2\pi i/h}$-regular element. It turns out that the well-generated groups are precisely those that have Coxeter elements\footnote{This is easy to see from \cite[Prop.~4.2]{Bessis-Zariski} and known properties of regular numbers.}.
\end{definition}

It is a standard property of the invariant theory of reflection groups \cite[Thm.~11.24:(iii)]{lehrer-taylor-book} that for an irreducible such $W$, the centralizer of a Coxeter element $c\in W$ is the cyclic group $\langle c \rangle$. Since, as we mentioned, the order of $c$ is $h$, this implies the following useful formula \begin{equation}|\mathcal{C}|=\dfrac{|W|}{h},\label{EQ: |C|=|W|/h}\end{equation} for the size of the conjugacy class $\mathcal{C}$ of $c$. We might call it the \emph{Coxeter class} in what follows.

\begin{remark}\label{Rem: general defn Cox elts}
Notice that with the previous Definition~\ref{Defn: Coxeter element}, the inverse $c^{-1}$ of a Coxeter element $c$ is not \emph{always} a Coxeter element itself. This will happen when $c$ does not have the  complex number $e^{-2\pi i/h}$ as an eigenvalue. For real reflection groups $W$ this is never an issue, since we always have that $g$ and $g^{-1}$ are conjugate for any $g\in W$.

There is a slightly more general definition of Coxeter elements \cite{RRS} where one considers all $\zeta$-regular elements for any primitive $h$-th root of unity $\zeta$. These extra elements can be derived from the Coxeter elements defined above via certain automorphisms of $W$ that preserve its set of reflections. All our theorems hold in this case as well, without extra work, because of this property; we briefly return to it in Rem.~\ref{Remark: refl autom. and gen. Cox elts}.
\end{remark}

\subsection{Filtration by a parabolic tower} 

Almost as much as in the real case, a large part of the theory of complex reflection groups is affected by the combinatorics of the arrangement $\mathcal{A}_W$ of reflection hyperplanes of $W$. We write $\mathcal{L}_W$ for its {\it \color{blue}intersection lattice}, whose elements -arbitrary intersections of the hyperplanes in $\mathcal{A}_W$- are called {\it \color{blue} flats} and are denoted by $X\in\mathcal{L}_W$. 

The pointwise stabilizer $W_S$ of any subset $S\subset V$ will be called a {\it \color{blue} parabolic subgroup}. It is a theorem of Steinberg \cite[\S~4.2]{broue-book-braid-groups} that $W_S$ is also a reflection group, generated by those reflections which fix $S$. This implies that any parabolic subgroup is of the form $W_X$ for some flat $X\in\mathcal{L}_W$, and thus that the poset of parabolic subgroups of $W$ is isomorphic to the lattice of flats $\mathcal{L}_W$.

In Section~\ref{Section: Rep-theoretic techniques on enumeration} we studied a general approach towards enumerating weighted factorizations of group elements. We will apply these techniques in the case of a complex reflection group $W$ and factorizations of its Coxeter elements $c$ into factors which belong to the (closed under conjugation) set $\mathcal{R}$ of reflections. 

Moreover, we consider a weight system (see Defn.~\ref{Defn: weight systems}) $\towerwt$ that assigns different weights on the reflections according to the filtration of $\mathcal{R}$ by a {\it\color{blue} parabolic tower} $T$. The latter is defined as a maximal\footnote{This is chosen for clarity of exposition; Thm.~\ref{thm:main} clearly holds for non-maximal chains as well.} chain in the lattice of parabolic subgroups \begin{equation}
T:=\big(\{\mathbbl{1}\}=W_0\leq W_1\leq \cdots \leq W_n=W \big),\label{Eq: parabolic tower}
\end{equation} where we will therefore have $W_i=W_{X_i}$ for flats $X_i\in \mathcal{L}_W$ such that $\op{codim}(X_i)=i$ and $X_i\supset X_{i+1}$. That is, for a tuple of complex variables $\bm\omega:=(\omega_i)_{(1\leq i\leq n)}$ and as in Defn.~\ref{Defn: partition P via tower of subgroups}, we have for a reflection $\tau\in\mathcal{R}$ that $\towerwt(\tau)=\omega_i$ if and only if $\tau\in W_i\setminus W_{i-1}$.

\begin{remark}
Notice that for such parabolic weight systems $\towerwt$ and because of Steinberg's theorem, reflections $\tau$ with the same fixed hyperplane $H=V^{\tau}$ are assigned the same weight. One can then rephrase this construction by considering a filtration of $\mathcal{A}_W$ via the localized arrangements $\mathcal{A}_{X_i}:=\bigcup_{H\supset X_i} H$ and then weight the reflections accordingly. It would be interesting to see application of such filtrations in arbitrary hyperplane arrangements. 
\end{remark}

Since the (unique as in Defn.~\ref{Defn: Coxeter element}) conjugacy class $\mathcal{C}$ of Coxeter elements and the reflections $\mathcal{R}$ are always assumed for the products and factors sets, we drop the corresponding indices from the generating function \eqref{EQ: Defn of FAC, weighted, general}. That is, we study the exponential generating function \begin{equation}
\fac_W^T(t,\bm\omega) := \sum_{\ell \geq 1 } \frac{t^\ell}{\ell!}
	\sum_{(\tau_1,\tau_2,\dots,\tau_\ell,c) \in \mathcal{R}^\ell \times \mathcal{C}\atop \tau_1\tau_2\dots\tau_\ell =c} \towerwt(\tau_1)\towerwt(\tau_2)\dots\towerwt(\tau_\ell),\label{Eq: Defn of Fac, parabolic weights}
\end{equation} of reflection factorizations of Coxeter elements $c\in \mathcal{C}$, weighted via a system $\towerwt$ indexed by a parabolic tower $T$ as given in \eqref{Eq: parabolic tower}. 

\subsection{Generalized Jucys-Murphy elements}
\label{Section: generalized JM elts}

As we discussed in \S\,\ref{Section: weight function gives commutative subalgebra} and \S\,\ref{Section:weight function by group tower} an important ingredient in computing $\fac_W^T$ is the spectrum of the algebra $\CC[\bm J]$, generated by the sums $J_i\in\CC[W]$ of factors that are weighted by the same variable $\omega_i$. In the case of reflection groups $W$, the elements $J_i$ bare sufficient resemblance to the traditional Jucys-Murphy elements (as for instance in \cite{OV}) that we will call them by the same name:

\begin{definition}
For a complex reflection group $W$ and a parabolic tower $T$ as in \eqref{Eq: parabolic tower}, we define the {\it \color{blue} generalized Jucys-Murphy} elements $J_{T,i}$ as the following sums in the group algebra $\CC[W]$: $$J_{T,i}:=\hspace{-1cm}\sum_{\quad\quad\tau\in\mathcal{R}\bigcap(W_i\setminus W_{i-1})}\hspace{-0.5cm}\tau.\quad$$We will denote by $\CC[\bm J_T]$ the commutative, simultaneously diagonalizable (see above \eqref{Eq: Spec_chi(A(w)) description}) subalgebra of $\CC[W]$ they generate.
\end{definition}

We remark that the subalgebra $\CC[\bm J_T]$ always contains the identity $\mathbbl{1}\in W$. Indeed, the first Jucys-Murphy element is of the form $J_{T,1}=\tau^{}_H+\cdots+\tau^{s-1}_{H}$ where $s$ is the order of the unitary reflection $\tau^{}_H$ and where the hyperplane $H$ is the first flat $X_1$ of our filtration. Then we have $$J_{T,1}^2=(s-2)J_{T,1}+(s-1)\mathbbl{1}.$$ If we consider non-maximal parabolic towers, then $\CC[\bm J_T]$ is still an algebra with a unit (because the $J_{T,i}$'s are simultaneously \emph{diagonalizable}) but which might not equal the identity $\mathbbl{1}\in W$. To justify the terminology consider indeed the symmetric group $S_n=S_{[n]}$ and the (standard) parabolic tower $$T=(\{\mathbbl{1}\}\leq S_{\{1\}}\leq S_{\{1,2\}}\leq\cdots\leq S_{[n]}).$$ We can easily see that $J_{T,i}=(1i)+(2i)+\cdots +(i-1,i)$ is the usual $i$-th Jucys-Murphy element. The algebras $\CC[\bm J_T]$ share some but not all properties of the algebra generated by the usual Jucys-Murphy elements in $S_n$. In particular, we will see in the next proposition that the generators $J_{T,i}$ always have integer spectrum. However, they don't always generate the Gelfand-Tsetlin algebra of the tower $T$ (see \S\ref{Section End: Okounkov-Vershik approach}).

\begin{proposition}\label{Prop: bounds on spectrum of JM elements}
Over any representation of $W$ and for any tower $T$, the generalized Jucys-Murphy elements $J_{T,i}$ have integer eigenvalues $\lambda_j(J_{T,i})$ which, moreover, satisfy the bounds $$-|\mathcal{R}_i^*\setminus\mathcal{R}_{i-1}^*|\leq \lambda_j(J_{T,i})\leq |\mathcal{R}_i\setminus \mathcal{R}_{i-1}|,$$where as usual, $\mathcal{R}^*_i$ and $\mathcal{R}_i$ denote the sets of reflecting hyperplanes and reflections of $W_i$.
\end{proposition}
\begin{proof}

In Section~\ref{Section:weight function by group tower} we described how any irreducible representation $V_{\chi}$ of $W$ decomposes as a vector sum of simultaneous eigenspaces $V_{\towerchi}$ of the $J_{T,i}$, indexed by certain tuples of characters $\towerchi:=(\chi_0,\chi_1,\cdots,\chi_n=\chi)\in\op{Res}_T(\chi)$. On each such eigenspace $V_{\towerchi}$, we showed \eqref{Eq: Spec_chi(A(w)) description} that the generalized Jucys-Murphy element $J_{T,i}$ acts as multiplication by $\widetilde{\chi}_i(\mathcal{R}_i)-\widetilde{\chi}_{i-1}(\mathcal{R}_{i-1})$. 

To prove the integrality property, we must show then that the normalized traces $\widetilde{\chi}(\mathcal{R})$ are integers (for an arbitrary reflection group $W$ and character $\chi\in\widehat{W}$). Since $\mathcal{R}$ forms a union of conjugacy classes, these numbers are in fact algebraic integers \cite[Corol.~1,~p.~52]{Serre}. We need to show that they are additionally rational numbers, which then forces them to be integers.

To see this, notice that because {\it unitary} reflections are semisimple, the parabolic subgroups $W_H$ of $W$ have to be cyclic (they consist of the identity and all reflections that fix $H$). This allows us to decompose the set of reflections $\mathcal{R}$  as
\begin{equation}\mathcal{R}=\bigsqcup_{H\in\mathcal{R^*}}\{\tau_H,\cdots, \tau_H^{s_H-1}\},\label{EQ: decomposition of R into cyclic groups} \end{equation} where $\tau_H$ is any generator of $W_H$ and $s_H:=|W_H|$ is its order.

For each eigenvalue $\lambda$ of $\tau_H$ in the representation $V_{\chi}$ afforded by $\chi$, the contribution of the set of reflections $\{\tau_H,\cdots,\tau_H^{s_H-1}\}$ to the trace $\chi(\mathcal{R})$ is given by the quantity $\sum_{i=1}^{s_H-1}\lambda^i$, which equals $(s_H-1)$ or $-1$ depending on whether $\lambda=1$ or not. This of course implies that the traces $\chi(\mathcal{R})$ are integers and thus that the normalized traces $\widetilde{\chi}(\mathcal{R})$ are rational numbers and completes the proof of the integrality property.

To see that the inequalities hold, consider for each $H\in\mathcal{R}^*$ the element $P_H:=\sum_{i=1}^{s_H-1} \tau_H^i$ of the group algebra $\CC[W]$. The reflections $\tau_H^i$ have of course common eigenspaces over any $V_{\chi}$ and with the same argument as in the previous paragraph, the matrix $\rho^{}_{V_{\chi}}(P_H)$ has two possible eigenvalues, $(s_H-1)$ or $-1$. Since $V_{\chi}$ is (or may assumed to be) a unitary representation, the matrices $\rho_{V_{\chi}}(P_H)$ are Hermitian and therefore the Weyl inequalities \cite[\S~6.7]{Franklin} on their eigenvalues apply. That is, since we can write $J_{T,i}=\sum_{H\in\mathcal{R}_i^*\setminus\mathcal{R}^*_{i-1}}P_H$ as a sum of Hermitian matrices, we get the bounds $$-|\mathcal{R}_i^*\setminus\mathcal{R}_{i-1}^*|\leq \lambda_j(J_{T,i})\leq \sum_{H\in\mathcal{R}^*_i\setminus\mathcal{R}^*_{i-1}}(s^{}_H-1).$$ The lower bound is just that of the statement; the higher bounds agree immediately after \eqref{EQ: decomposition of R into cyclic groups}.
\end{proof}

\begin{remark}
Notice that this previous proposition gives a very broad generalization of the integral property and bounds $-i$ and $i$ for the spectrum of the standard Jucys-Murphy element $J_{i+1}$ in $S_n$. 
\end{remark}

\subsection{Tower equivalence}

Our Thm.~\ref{thm:equivalence} in the Introduction claims a character-theoretic equivalence between two virtual characters. This notion is very much related to the representation theory of the algebras $\CC[\bm J_T]$ and we elaborate on it here. To define the equivalence relation, recall first that a \emph{virtual} character is a sum $\psi:=\sum \kappa_i\psi_i$ where $\psi_i\in\widehat{W}$ are irreducible characters and the $\kappa_i$'s complex numbers. It does not necessarily correspond to a representation of $W$, reducible or not. 

\begin{definition}
Given two \emph{virtual} characters $\psi$ and $\xi$, we say that they are {\color{blue}\emph{tower equivalent}} and write $\psi\towerequiv\xi$, if they agree on all subalgebras $\CC[\bm J_T]\subset \CC[W]$ (that is, for any tower $T$).
\end{definition}

Virtual characters are after all linear functionals on the group algebra, so any two of them will agree on a codimension $1$ subspace of $\CC[W]$. This space is not expected to be \emph{nice} though (as in to contain basis elements $g\in W$ or integer sums of them, etc.). Tower equivalence implies that the two virtual characters agree on the \emph{vector sum} of all the subalgebras $\CC[\bm J_T]$.

\begin{lemma}\label{lem:tower equiv = equality on A^k}
For two virtual characters $\psi$ and $\xi$, we have that $\psi\towerequiv\xi$ if and only if they agree on all powers of the element $\mathbf{A}(\bm\omega):=\sum_{i=1}^n\omega_i J_{T,i}$, and for any tower $T$ and weights $\omega_i$.
\end{lemma}
\begin{proof}
The forward implication is clear since all powers of $\mathbf{A}(\bm\omega)$, including the empty product $\big(\mathbf{A}(\bm\omega)\big)^0:=\mathbbl{1}\in W$, belong to $\CC[\bm J_T]$. For the inverse implication it is beneficial to treat the $\omega_i$'s as parameters. Then, agreement of $\psi$ and $\xi$ on the power $\big(\mathbf{A}(\bm\omega)\big)^k$ means that they in fact agree on all coefficients $J_{T,1}^{a_1}\cdots J_{T,n}^{a_n}$ of the monomials $\omega_1^{a_1}\cdots\omega_n^{a_n}$ (with $\sum a_i=k$) in the expansion of $\big(\mathbf{A}(\bm\omega)\big)^k$. Since these elements $\prod J_{T,i}^{a_i}$ span the whole algebra $\CC[\bm J_T]$, the argument is complete.
\end{proof}

Tower equivalence is in some sense a statement about isomorphic \emph{virtual} $\CC[\bm J_T]$-representations. To see this, start with two virtual characters $\psi:=\sum_{i=1}^r\kappa_i\psi_i$ and $\xi:=\sum_{j=1}^s\mu_j\xi_j$, with $\psi_i,\xi_j\in\widehat{W}$ and $\kappa_i,\mu_j\in \CC$, that are tower equivalent. Consider the following two exponential generating functions, which by the previous lemma are equal: \begin{equation}\sum_{k\geq 0}\Big( \sum_{i=1}^r \kappa_i\psi_i \big( (\mathbf{A}(\bm\omega))^k\big)\Big)\cdot\dfrac{z^k}{k!}=\sum_{k\geq 0}\Big( \sum_{j=1}^s\mu_j\xi_j\big((\mathbf{A}(\bm\omega))^k\big)\Big)\cdot\dfrac{z^k}{k!}.\label{EQ: virtual isom part 1}\end{equation} As in \S\ref{Subsection: A frobenius Lemma for weighted enumeration} we can rewrite them as \begin{equation}\sum_{\lambda(\bm\omega)\in\op{Spec}_{\psi_i}(\mathbf{A}(\bm\omega))\atop i=1\dots r}\kappa_i\cdot e^{z\lambda(\bm\omega)}=\sum_{\lambda(\bm\omega)\in\op{Spec}_{\xi_j}(\mathbf{A}(\bm\omega))\atop j=1\dots s}\mu_j\cdot e^{z\lambda(\bm\omega)},\label{EQ: virtual isom part 2}\end{equation} where the $\lambda(\bm\omega)$ run through the eigenvalues of $\mathbf{A}(\bm\omega)$ on the representations $V_{\psi_i}$ and $V_{\xi_j}$ respectively. Recall from \S\ref{Section: weight function gives commutative subalgebra} that these eigenvalues are of the form $\lambda(\bm\omega)=\sum_{i=1}^n\ell_i\omega_i$, where the tuple $(\ell_i)_{i=1}^n$ is a ``weight" of the commutative algebra $\CC[\bm J_T]$ encoded via the eigenvalues of its generators $J_{T,i}$.  

The linear independence of exponential functions implies that each tuple $(\ell_i)_{i=1}^n$ should appear in both sides with the same total multiplicity (including the scaling by $\kappa_i$ and $\mu_j$). It is in this sense that we say that Tower equivalence is $\CC[\bm J_T]$-isomorphisms between virtual representations.

By grouping together these tuples $(\ell_i)_{i=1}^n$ with respect to a natural statistic we get the following proposition. It shows that tower equivalence implies a stronger relation between the constituents $\psi_i$ and $\xi_j$ than its definition suggests, and to which we will return in Corol.~\ref{cor:filterCoxeterNumbers}.

\begin{proposition}\label{Prop: formal stronger tower equivalence}
With notation as above, we have that $\psi\towerequiv\xi$ if and only if $$\sum \kappa_i\psi_i\cdot t^{\widetilde{\psi}_i(\mathcal{R})}\towerequiv\sum \mu_j\xi_j\cdot t^{\widetilde{\xi}_j(\mathcal{R})},$$ by which we mean that the corresponding coefficients of the two polynomials are tower equivalent. 
\end{proposition}
\begin{proof}
A priori, the parameter $t$ appears with algebraic integer powers, and this would not be a problem for us, but recall that we have shown in  Prop.~\ref{Prop: bounds on spectrum of JM elements} that the powers are in fact integers. 

As we discuss above, tower equivalence between $\psi$ and $\xi$ implies equation \eqref{EQ: virtual isom part 2}. There, we collect terms with respect to the statistic $|\lambda(\bm\omega)|:=\sum_{i=1}^n\ell_i$ and linear independence of exponential functions implies the formal relation $$\sum_{\lambda(\bm\omega)\in\op{Spec}_{\psi_i}(\mathbf{A}(\bm\omega))\atop i=1\dots r}\kappa_i\cdot e^{z\lambda(\bm\omega)}\cdot t^{|\lambda(\bm\omega)|}=\sum_{\lambda(\bm\omega)\in\op{Spec}_{\xi_j}(\mathbf{A}(\bm\omega))\atop j=1\dots s}\mu_j\cdot e^{z\lambda(\bm\omega)}\cdot t^{|\lambda(\bm\omega)|}.$$ The statistic $|\lambda(\bm\omega)|$ is an eigenvalue of the (unweighted) element $\mathbf{A}:=\sum J_i=\sum_{\tau\in\mathcal{R}}\tau$, which is central in $\CC[W]$. Therefore, all $\lambda(\bm\omega)$ associated to a particular constituent $\psi_i$ or $\xi_j$ have the same statistic $|\lambda(\bm\omega)|=\widetilde{\psi}_i(\mathcal{R})$ (or $\widetilde{\xi}_j(\mathcal{R})$ respectively). This completes the proof, by going backwards through \eqref{EQ: virtual isom part 2} and \eqref{EQ: virtual isom part 1} to lemma~\ref{lem:tower equiv = equality on A^k}.
\end{proof}

\subsection{The $W$-Laplacian, the \emph{group algebra} Laplacian, and Coxeter numbers}
\label{subsec:Laplacians}

Our Thm.~\ref{thm:main} gives a product form for the exponential generating function \eqref{Eq: Defn of Fac, parabolic weights} in terms of the eigenvalues of a matrix that is an analog of the usual graph Laplacian. In this section, we introduce this $W$-Laplacian, one of the main objects of study in this work, along with its group algebra counterpart. We will return to the $W$-Laplacian in \S\ref{Section: Matrix-Forest} where we give a combinatorial description of all coefficients of its characteristic polynomial and discuss how it unites two different analogs of the Matrix tree theorem for reflection groups. 

\begin{definition}[$W$-Laplacian]\label{Defn: W-Laplacian matrix}
For a complex reflection group $W$ of rank $n$ and any weight system (see Defn.~\ref{Defn: weight systems}) $\wt$ on its set $\mathcal{R}$ of reflections, we define the (weighted) {\color{blue} $W$-Laplacian matrix} as $$\op{GL}(V)\ni L_W(\bm\omega):=\sum_{\tau\in\mathcal{R}}\wt(\tau)\cdot\big(\mathbf{I}_n-\rho^{}_V(\tau)\big),$$ where $\mathbf{I}_n$ is the $n\times n$ identity matrix and $\rho^{}_V$ the reflection representation of $W$. If the weight system comes from a parabolic tower $T$ as in \eqref{Eq: parabolic tower}, we write $L_W^T(\bm\omega)$; the unweighted version is denoted by $L_W$.
\end{definition}

It is easy to see that for the symmetric group $S_n$, when the system $\wt$ is the finest possible (i.e. such that every reflection $(ij)$ gets its own weight $\omega_{ij}$), the $W$-Laplacian $L_{S_n}(\bm\omega)$ agrees (up to the projection $\mathbb{C}^n\rightarrow \mathbb{C}^{n-1}$, see Remark~\ref{rem:prefactors}) with the weighted Laplacian of the complete graph $K_n$. 
Indeed, recall (see e.g.~\cite{Stanley:EC2}) that the {\color{blue} Laplacian matrix of a weighted graph} $G$ is obtained by taking the opposite of its weighted adjacency matrix, and adjusting the diagonal elements in such a way that the sum over every line is equal to zero. Therefore, for the complete graph $K_n$ carrying the edge weight $\omega_{i,j}$ on the edge $\{i,j\}$ for each $i\neq j$, the Laplacian matrix has an $(i,j)$-coefficient equal to $-\omega_{i,j}$  if $i\neq j$, and to $\sum_{k\neq i} \omega_{i,k}$ if $i=j$.

For example, contracting notation, we have for $n=4$:

$$\underbrace{\begin{bmatrix}
\sum\omega_{1j} & -\omega_{12} &-\omega_{13} & -\omega_{14}\\
-\omega_{12} & \sum\omega_{2j} & -\omega_{23} & -\omega_{24}\\
-\omega_{13} & -\omega_{23} &  \sum\omega_{3j} & -\omega_{34}\\
-\omega_{14} & -\omega_{24} & -\omega_{34} & \sum\omega_{4j} \end{bmatrix}}_{L_{K_4}(\bm\omega)} = \sum_{i<j}\begin{bmatrix}
0 & 0 & 0 & 0 \\
0 & \omega_{ij} & -\omega_{ij} & 0\\
0 & -\omega_{ij} & \omega_{ij} & 0\\
0 & 0 & 0 & 0\end{bmatrix}=\sum_{i<j} \omega_{ij}\cdot \Bigg(\mathbf{I}_4- \underbrace{\begin{bmatrix}
1&0&0&0\\
0&0&1&0\\
0&1&0&0\\
0&0&0&1\end{bmatrix}}_{\rho^{}_V\big((ij)\big)}\Bigg).
$$
For a classical presentation of Laplacian matrices of a graph and of the matrix-tree theorem, see for example again~\cite{Stanley:EC2}.

The following proposition, which is not used in our proofs, is an analogue of the classical expression of the Laplacian in terms of the edge-vertex incidence matrix in graph theory~(see e.g. \cite{GodsilRoyle}).  
We consider an ordering $\rho_V(\tau_1),\dots,\rho_V(\tau_N)$ of the set of reflections of $W$. For $i\in \{1,\dots,N\}$ we let $\zeta_i$ be the eigenvalue of $\rho_V(\tau_i)$ different from $1$, and $r_i$ be a vector normal to the reflecting hyperplane of $\rho_V(\tau_i)$, normalized so that $\langle r_i,r_i\rangle = 1-\zeta_i$. Here $\langle \cdot, \cdot \rangle$ is a hermitian product on $GL(V)$ making reflections of $W$ orthogonal. Note that the vector $r_i$ is defined up to the choice of a unit complex number, and that for all $v\in V$ we have $\tau_i(v) = v - \langle v, r_i \rangle r_i$.
The proof of the following is straightforward (see the similar argument in the proof of Lemma~\ref{Lem:Burman M-T} in Section~\ref{Section: Matrix-Forest}). 
\begin{proposition}\label{prop:L=RR^T}
	Let $(e_i)_{i=1}^n$ be an orthormal basis of $V$, and $R$ the $n\times N$ matrix $R_{ij}:=\langle e_i,r_j\rangle$. Then the matrix of the weighted $W$-Laplacian $L_W(\bm\omega)$ in the basis $(e_i)_{i=1}^n$ is given by:
	$$
	Mat_{(e_i)}(L_W(\bm\omega)) = R \Omega \overline{R}^T,
	$$
	where $\Omega=\mathrm{diag}(\wt(\tau_1),\dots,\wt(\tau_N))$.
\end{proposition}

The unweighted $W$-Laplacian $L_W$ is the image under $\rho_V$ of a \emph{central} element of the group algebra (since $\mathcal{R}$ is closed under conjugation). This element $\mathbf{B}$ was considered by Beynon, Lusztig \cite{BL} and also Malle \cite{Malle}, but more explicitly by Gordon and Griffeth \cite[\S~1.3]{GG}. Its weighted version $\mathbf{B}(\bm\omega)$ has been considered in the symmetric group by various authors, including \cite[\S~2]{AK}, \cite{BZ}\footnote{Their element $\mathbf{B}(\bm\omega)$ \cite[p.~135]{BZ} is our element $\mathbf{A}(\bm\omega)$ of \S\ref{Section: Rep-theoretic techniques on enumeration}, a shift of our $\mathbf{B}(\bm\omega)$ of Defn.~\ref{Defn: group-algebra-Laplacian}.}, and \cite[p.~472]{Bacher}. 

It is not very difficult  to see that $\mathbf{B}(\bm\omega)$ is the weighted Laplacian for the Cayley graph of $W$, with the set of reflections $\mathcal{R}$ as generators. This leads to the following definition (Cesi \cite[\S~2]{cesi_remarks} calls a related, more general object the \emph{representation} Laplacian). 

\begin{definition}[\emph{group algebra} Laplacian]
\label{Defn: group-algebra-Laplacian}
For a complex reflection group $W$, and any weight system (Defn.~\ref{Defn: weight systems}) $\wt$ on the set $\mathcal{R}$ of reflections, we  define the (weighted) \emph{\color{blue}  group algebra Laplacian}: $$\CC[\bm\omega][W]\ni \mathbf{B}(\bm\omega):=\sum_{\tau\in\mathcal{R}}\wt(\tau)\cdot (\mathbbl{1}-\tau),$$ where $\mathbbl{1}$ denotes the identity in $W$. It induces the $W$-Laplacian via $\rho_V\big(\mathbf{B}(\bm\omega)\big)=L_W(\bm\omega)$; we write $\mathbf{B}$ for its unweighted version and $\mathbf{B}^T(\bm\omega)$ if the weights are given by a parabolic tower $T$ as in \eqref{Eq: parabolic tower}.
\end{definition}

The \emph{group algebra} Laplacian is a shift of the element $\mathbf{A}(\bm\omega):=\sum_{\tau\in\mathcal{R}}\wt(\tau)\cdot\tau$ of \S\ref{Section: Rep-theoretic techniques on enumeration}. Indeed, \begin{equation} \mathbf{B}(\bm\omega)=\wt(\mathcal{R})-\mathbf{A}(\bm\omega),\label{Eq: B(w)=wt(R)-A(w)}\end{equation} where we write $\wt(\mathcal{R}):=\sum_{\tau\in\mathcal{R}}\wt(\tau)$ as usual (and $\wt$ is \emph{any} weight system as in Defn.~\ref{Defn: weight systems}). It therefore can be used to give a formula for the generating function $\mathcal{F}_W^T(t,\bm\omega)$ of \eqref{Eq: Defn of Fac, parabolic weights} after the techniques of \S\ref{Section: Rep-theoretic techniques on enumeration}. We will see that the reason this shifted version is better behaved is that $\mathbf{B}(\bm\omega)$ is a \emph{Lie-like} element for the reflection representation of $W$ (see \S\ref{sec:LieLike}). 

The centrality of the unweighted element $\mathbf{B}$ allows us to give the first connection between $W$-Laplacians and the numerology of reflection groups (see also \cite[\S~1.3]{GG}). Recall our Definition~\eqref{Eq: Coxeter number h} of the Coxeter number $h$.

\begin{proposition}\label{Prop: L_W has all eigenvalues =h}
$L_W$ is the matrix of scalar multiplication by the Coxeter number $h$ of $W$.
\end{proposition}
\begin{proof}
As is immediate by the two definitions above, we have $L_W=\rho^{}_V(\mathbf{B})$. Since $\mathbf{B}$ is a central element it will act in the (irreducible) reflection representation $V$ as multiplication by a scalar; namely all eigenvalues of $L_W$ are equal to $\op{Tr}(L_W)/n$. We will show that $\op{Tr}(L_W)=hn$.

We repeat here the decomposition of the set of reflections $\mathcal{R}$ we gave in \eqref{EQ: decomposition of R into cyclic groups}. It states that
$$\mathcal{R}=\bigsqcup_{H\in\mathcal{R^*}}\{\tau_H,\cdots, \tau_H^{s_H-1}\},$$ where $\tau_H$ is any generator of $W_H$ and $s_H:=|W_H|$ is its order. We may assume that the non-1 eigenvalue of $\tau_H$ equals $e^{2\pi i/s_H}$ and write:
$$\op{Tr}(L_W)=\sum_{\tau\in\mathcal{R}}\op{Tr}\big(\mathbf{I}_n-\rho_V(\tau)\big)=\sum_{H\in\mathcal{R}^*}\sum_{k=1}^{s_H-1}1-e^{2\pi i k/s_H}=\sum_{H\in\mathcal{R}^*}s_H=|\mathcal{R}|+|\mathcal{R}^*|,$$ where the last equality is given again by the decomposition above. The coxeter number $h$ was defined in \eqref{Eq: Coxeter number h} as $(|\mathcal{R}|+|\mathcal{R}^*|)/n$ and the proof is complete.
\end{proof}

Gordon and Griffeth \cite{GG} used the element $\mathbf{B}$ to generalize Coxeter numbers to a statistic for each irreducible character of a complex reflection group $W$. 

\begin{definition}[generalized Coxeter numbers]\label{Defn: Coxeter numbers}
For a complex reflection group $W$ and a character $\chi\in\widehat{W}$, we define the {\color{blue} \emph{generalized Coxeter number}} $c_{\chi}$ as the normalized trace of the central, unweighted (Defn.~\ref{Defn: group-algebra-Laplacian}) element $\mathbf{B}=\sum_{\tau\in\mathcal{R}}(\mathbbl{1}-\tau)$. That is,$$c_{\chi}:=\dfrac{1}{\chi(1)}\cdot\big(|\mathcal{R}|\chi(1)-\chi(\mathcal{R})\big)=|\mathcal{R}|-\widetilde{\chi}(\mathcal{R}).$$
\end{definition}

In particular, for the reflection representation $V$ of $W$ and since $L_W=\rho^{}_V(\mathbf{B})$, Prop.~\ref{Prop: L_W has all eigenvalues =h} implies that the Coxeter number $c_{\chi^{}_V}$ agrees with the usual Coxeter number $h$ of $W$ defined in \eqref{Eq: Coxeter number h}.

\begin{proposition}\label{Prop: coxeter numbers are integers}
The (generalized) Coxeter numbers $c_{\chi}$ are integers and they satisfy $$0\leq c_{\chi}\leq hn,$$ where $h$ is the (usual) Coxeter number of $W$ as defined in \eqref{Eq: Coxeter number h}.\newline Moreover, for any tower $T$ and tuple $\towerchi:=(\chi_0,\chi_1,\cdots,\chi_n)\in\op{Res}_T(\chi)$, we have $c_{\chi_i}\geq  c_{\chi_{i-1}}$.
\end{proposition}
\begin{proof}
Everything here is a corollary of the analogous statements for the spectrum of the generalized Jucys-Murphy elements (Prop.~\ref{Prop: bounds on spectrum of JM elements}) and Defn.~\ref{Defn: Coxeter numbers} above. In particular, the $c_{\chi}$'s are integers because the normalized traces $\widetilde{\chi}(\mathcal{R})$ are integers (a fact shown in the proof of Prop.~\ref{Prop: bounds on spectrum of JM elements}).

To see the bounds, recall (again from the same proof) that the generalized Jucys-Murphy elements $J_{T,i}$ have eigenvalues of the form $\widetilde{\chi}_i(\mathcal{R}_i)-\widetilde{\chi}_{i-1}(\mathcal{R}_{i-1})$ (with some suitable multiplicity). Since by Defn.~\ref{Defn: Coxeter numbers} we have that $$c_{\chi_i}-c_{\chi_{i-1}}=|\mathcal{R}_i\setminus\mathcal{R}_{i-1}|-\big(\widetilde{\chi}_i(\mathcal{R}_i)-\widetilde{\chi}_{i-1}(\mathcal{R}_{i-1})\big),$$ the bounds of Prop.~\ref{Prop: bounds on spectrum of JM elements} imply that $$0\leq c_{\chi_i}-c_{\chi_{i-1}}\leq |\mathcal{R}_i\setminus \mathcal{R}_{i-1}|+|\mathcal{R}^*_i\setminus\mathcal{R}_{i-1}^*|.$$
This gives all bounds of the statement, trivially for the second, and by summing up over $i=1\dots n$ and writing $hn=|\mathcal{R}|+|\mathcal{R}^*|$ for the first.
\end{proof}

The Coxeter numbers of the exterior powers $\bigwedge^k(V)$ of the reflection representation $V$ of $W$ are spread out evenly across the bounds in the previous proposition. The following is a simple calculation, analogous to Prop.~\ref{Prop: L_W has all eigenvalues =h} (see \cite[proof of lemma~6]{Michel}) which we also recover in a different way as a consequence of Corollary~\ref{cor:BZ} in Section~\ref{subsec:LieLike}.

\begin{proposition}\label{prop:wedgeCoxeterNumbers}
The characters $\chi^{}_k:=\chi_{\bigwedge^k(V)}$ have Coxeter numbers given by $c^{}_{\chi^{}_k}=hk$.
\end{proposition}

This allows us to rephrase Thm.~\ref{thm:equivalence} in a way that separates each $k$-th exterior power $\bigwedge^k(V)$ of the reflection representation and does so almost suggesting that in fact the exterior \emph{algebra} of $V$ is at play here.

\begin{corollary}\label{cor:filterCoxeterNumbers}
We can rewrite Thm.~\ref{thm:equivalence} as the formal tower equivalence
	\begin{align}\label{eq:filterCoxeterNumbers}
		\sum_{\chi\in\widehat{W}}\chi(c^{-1})\cdot\chi\cdot t^{c_{\chi}/h}\towerequiv\sum_{k=0}^n(-1)^k\chi_{\bigwedge^k(V)}\cdot t^{k},
	\end{align}
	by which we mean that the corresponding coefficients of the two sides are tower equivalent. 
\end{corollary}
\begin{proof}
In Prop.~\ref{Prop: formal stronger tower equivalence} we showed that tower equivalence between any two virtual characters $\psi$ and $\xi$ implies a seemingly stronger formal relation. Applying the result of Prop.~\ref{Prop: formal stronger tower equivalence} to the virtual characters $\sum_{\chi\in\widehat{W}}\chi(c^{-1})\cdot\chi$ and $\sum_{k=0}^n(-1)^k\chi^{}_{\bigwedge^k(V)}$ of Thm.~\ref{thm:equivalence}, and rewriting the normalized traces $\widetilde{\chi}(\mathcal{R})$ in terms of the Coxeter numbers $c_{\chi}$, we immediately get the statement.
\end{proof}

\section{Proof of Theorem~\ref{thm:main} for the exceptional groups}
\label{sec:exceptional}

In this section, we discuss the computer-assisted verification of Thm.~\ref{thm:main} for the exceptional groups. It would be impossible (see discussion before Prop.~\ref{Prop: num of chains face lattice}) to check all possible parabolic towers, so we first prove that it is sufficient to check orbit representatives of them (a fact we will also use for the infinite families).

We write $N(W):=N_{\op{GL}(V)}(W)$ for the normalizer of $W$ in $\op{GL}(V)$. Then, for  an element $g\in N(W)$ and a parabolic tower $T$ as in \eqref{Eq: parabolic tower}, we consider the \emph{conjugate} tower \begin{equation}
^gT:=\big(\{\mathbbl{1}\}=\ \!\! ^gW_0\leq\ \!\! ^gW_1\leq \cdots \leq\ \!\! ^gW_n=W\big),\label{EQ: ^gT defn}
\end{equation}where as usual $ ^gW_i:=g^{-1}W_ig$. The collection of fixed sets $V^w$ of elements $w\in W$ is of course fixed under conjugation by $g\in N(W)$. Therefore, the set of parabolic subgroups is also respected and $ ^gT$ is itself a parabolic tower. With the following proposition, we show that it is sufficient to prove Thm.~\ref{thm:main} for one representative in each such class $\{\ \!\! ^gT\ |\ g\in N(W)\}$.

\begin{proposition}\label{Prop: Thm invariant under N(W)-orbits}
Both sides of equation \eqref{eq:mainThm} in Thm.~\ref{thm:main} are invariant in $N(W)$-orbits of towers~$T$.
\end{proposition}
\begin{proof}
One should first notice that for an element $g\in N(W)$, the two weight systems $\mathbf{w}^{}_T$ and $\mathbf{w}^{}_{^gT}$ as in \eqref{Eq: parabolic tower} determine each other via the equation $$\mathbf{w}^{}_T(\tau)=\mathbf{w}^{}_{^gT}(g^{-1}\tau g).$$ Indeed, conjugation by $g\in N(W)$ respects the set $\mathcal{R}$ of reflections and if for a reflection $\tau$ we have that $\tau\in W_i\setminus W_{i-1}$, then we also have that $g^{-1}\tau g\in\ \!\! ^gW_i\setminus\ \!\! ^gW_{i-1}$ and vice versa. 

As we have mentioned earlier, $\zeta$-regular elements (for a given $\zeta$) form a single conjugacy class \cite[Corol.~11.25]{lehrer-taylor-book} which means that conjugation by $g\in N(W)$ also respects the set of Coxeter elements. To see now the invariance of the left hand side of \eqref{eq:mainThm} notice that the contribution of a factorization $\tau_1\cdots \tau_{\ell}=c$ in $\mathcal{F}_W^T$ is the same as the contribution of $^g\tau_1\cdots\ \!\! ^g\tau_{\ell}=\ \!\! ^gc$ in $\mathcal{F}_W^{^gT}.$
Since this operation on factorizations is invertible and all such factorizations are considered under both towers, this shows that $\fac^T_W(t,\bm\omega)=\fac^{^gT}_W(t,\bm\omega)$. 

\noindent For the right hand side, it is easy to see by the Definition~\ref{Defn: W-Laplacian matrix} and the previous relation between the weight systems $\mathbf{w}^{}_T$ and $\mathbf{w}^{}_{^gT}$ that $$L^{^gT}_W(\bm\omega)=g^{-1}L_W^T(\bm\omega)g.$$ Indeed this means that the eigenvalues $\sigma_i(\bm\omega)$ of the (weighted) $W$-Laplacian are not altered, so neither is the product $\prod_{i=1}^n (1-e^{-t\sigma_i(\bm\omega)})$.
\end{proof}

Every real reflection group comes equipped with a set $S$ of standard generators and a presentation encoded in a Coxeter diagram, both determined by the chamber geometry. Analogous generators and Coxeter-like presentations have been defined \cite[Appendix~A]{broue-book-braid-groups} for all complex reflection groups $W$, albeit without having a uniform geometric interpretation (but see \cite{Bessis-Zariski} for a \emph{justification} of their existence). Any subgroup $W_I\leq W$ generated by a subset $I\subset S$ is the pointwise fixator of some flat of $\mathcal{L}_W$ and will be called a {\color{blue} \emph{standard parabolic subgroup}} of $W$. The following is well-known in the community but we elaborate on it here to correct a slight misstatement of the standard reference \cite[Fact~1.7]{BMR} (which has been noticed before at \cite[\S~3.7.2]{Wilson}).

\begin{lemma}\label{Lem: standard towers are enough}
For any complex reflection group $W$ different from $G_{27}$, $G_{29}$, $G_{33}$, and $G_{34}$, any parabolic tower is conjugate under $N(W)$ to a standard parabolic tower. For the remaining four groups, one has to consider an additional set of standard generators.
\end{lemma}
\begin{proof}
In most cases it is sufficient to consider conjugation under $W$ and in that setting it is easy to set up an inductive proof. That is, one first shows that any \emph{maximal-rank} parabolic subgroup is $W$-conjugate to a standard one and then proceeds by induction. In particular, this is true for real reflection groups where any parabolic subgroup is conjugate to a standard parabolic subgroup \cite[Prop.~1.15]{HUM}; this argument appears also in \cite[Prop.~2.4]{Matthieu}.
  
The situation is the same for all exceptional groups as well, if one considers two systems of standard generators for the four groups listed in the statement; this has been checked via SAGE-CHEVIE. Only for the infinite family $G(r,p,n)$, $W$-conjugation is no longer sufficient. The flats of the corresponding arrangements are described in \cite[\S~6.4]{OT} and one can see that to conjugate a maximal-rank parabolic subgroup of $G(r,p,n)$ to a standard one, you might need to use an element of $G(r,1,n)$. Since $G(r,1,n)\leq N\big( G(r,p,n)\big)$, this easily completes the argument.
\end{proof}

\subsection{Computing \texorpdfstring{$\fac^T_W(t,\bm\omega)$}{$\fac^T_W(t,\omega)$} via SAGE-CHEVIE and verifying Thm.~\ref{thm:main}}

The exponential generating function $\fac^T_W(t,\bm\omega)$ as given in \eqref{EQ: Defn of FAC, weighted, general} can be computed after the techniques of \S\ref{Section:weight function by group tower} as the following finite sum of exponentials (on which we elaborate below):
\begin{equation}
\fac_W^T(t,\bm\omega)=\dfrac{e^{t\mathbf{w}_T^{}(\mathcal{R})}}{h}\cdot\sum_{\chi\in\widehat{W}}\chi(c^{-1})\sum_{\towerchi\in \op{Res}_T( \chi)}\op{mult}(\towerchi)\cdot\op{exp}\Big( -t\cdot\sum_{i=1}^n(c_{\chi_i}-c_{\chi_{i-1}})\cdot\omega_i\Big).\label{Eq: F_W^T final}
\end{equation}As the reader may recall, the set $\op{Res}_T(\chi)$ consists of those tuples of irreducible characters that may appear as we restrict the character $\chi\in\widehat{W}$ down the tower $T$; $\op{mult}(\towerchi)$ counts their multiplicity. The numbers $c_{\chi_i}$ are the Coxeter numbers of the characters $\chi_i$ of $W_i$ and \eqref{Eq: F_W^T final} is just a rewriting of Corol.~\ref{Corol: FAC^(A,P)_(G,c) final} for $G=W$, using further that $|W|/|\mathcal{C}|=h$ and the Defn.~\ref{Defn: Coxeter numbers} of Coxeter numbers.

Each of the components of the above formula can be computed via SAGE \cite{sagemath} and its interface to CHEVIE \cite{chevie-michel,chevie-all}. Indeed, CHEVIE has hard-coded realizations of all exceptional complex reflection groups (as permutation groups on a set of roots) along with a set of standard generators as in \cite[Appendix~A]{broue-book-braid-groups}. Following that, each subgroup $W_i$ in a tower $T$ may be specified by a list of the reflections that generate it and CHEVIE can enumerate all irreducible characters $\chi\in\widehat{W}$ and restrict any character to any subgroup (also recording the multiplicity). 

Similarly, after computing $\op{Res}_T^{}(V_{ref})$ for the reflection representation $V_{ref}$, SAGE can produce the eigenvalues of the $W$-Laplacian matrix $L_W^T(\bm\omega)$ after \eqref{Eq: Spec_chi(A(w)) description} and recalling \eqref{Eq: B(w)=wt(R)-A(w)}. Now, one has two finite sums of exponentials in the two sides of equation \eqref{eq:mainThm} and we can immediately check that they agree. In fact, although this is not really necessary, Prop.~\ref{Prop: coxeter numbers are integers} guarantees that the exponents will be (positive) integer sums of the $\omega_i$'s so we can turn this to a required equality between two \emph{polynomials} on variables $z_i:=e^{-t\omega_i}$.

Notice that after Prop.~\ref{Prop: Thm invariant under N(W)-orbits} and Lemma~\ref{Lem: standard towers are enough}, we had to verify Thm.~\ref{thm:main} only for standard parabolic towers (and for two sets of standard generators for the four groups of the Lemma). That is, we needed to check at most $2\cdot n!$ towers for each well generated complex reflection group $W$ of rank $n$. The number of all possible parabolic towers on the other hand is much greater. For instance, for real $W$, it is given by $|W|\cdot n!/2^n$ as the next proposition (see also \cite{Drew}) shows. This formula doesn't generalize to the complex types but the gain there is comparable to the real cases.

It took about 24h (but 20 cpu days) to prove Thm.~\ref{thm:main} for $E_8$ by checking all standard parabolic towers. It would have taken 7560 (real) \emph{years} to verify it for arbitrary parabolic towers with our current computing capacity.

\begin{proposition}\label{Prop: num of chains face lattice}
If $\mathcal{X}$ is a simplicial, central hyperplane arrangement of rank $n$ and $Q$ denotes its set of chambers, then there exist $\big(|Q|\cdot n!\big)/2^n$ maximal chains in its intersection lattice $L_{\mathcal{X}}$. 
\end{proposition}
\begin{proof}
Since the arrangement is simplicial, there will be $|Q|\cdot n!$ simplices in its first barycentric subdivision. These are in bijection with the set of maximal chains (flags) in the {\it face} lattice $F_{\mathcal{X}}$ of the arrangement. Finally, by a theorem of Bayer and Sturmfels (see for instance \cite[Thm.~2.3]{ERS}), the natural map from chains in $F_{\mathcal{X}}$ to chains in $L_{\mathcal{X}}$ has fibers of constant size $2^n$ when the chains are maximal.  
\end{proof}

\begin{remark}[Reflection automorphisms and generalized Coxeter elements]\label{Remark: refl autom. and gen. Cox elts}\ \newline
Our Thm.~\ref{thm:main} covers (as is) the case of general Coxeter elements (of Remark~\ref{Rem: general defn Cox elts}), which are related to our usual Coxeter elements via reflection automorphisms of $W$ \cite{RRS}. Indeed, such reflection automorphisms preserve towers of parabolic subgroups (a parabolic subgroup is always maximal in any set of reflection subgroups of fixed rank). Then, the argument proceeds as in Prop.~\ref{Prop: Thm invariant under N(W)-orbits}.
\end{remark}

\section{Combinatorial description of the eigenvalues of the $W$-Laplacian}
\label{sec:combinatorial}

In this section, we give a combinatorial algorithm to compute the spectrum of the $T$-weighted $W$-Laplacian $L^T_W(\bm\omega)$ relying only on the Coxeter-theoretic data of the tower $T$. As opposed to the Burman-Zvonkine theorem (that however allows arbitrary weights) the eigenvalues of the $W$-Laplacian are always non-negative integer sums of the weights $\omega_i$. 

The $W$-Laplacian $L^T_W(\bm\omega)$ belongs to the algebra $\CC[\bm J_T]$ of generalized Jucys-Murphy elements and we have shown in sections~\ref{Section: weight function gives commutative subalgebra}~and~\ref{Section:weight function by group tower} how to compute the spectrum of such objects. Indeed, recalling the relation (Defn.~\ref{Defn: group-algebra-Laplacian}) between $L^T_W(\bm\omega)$ and the \emph{group algebra} Laplacian $\mathbf{B}^T(\bm\omega)$, we have after \eqref{Eq: B(w)=wt(R)-A(w)} and the spectral description \eqref{Eq: Spec_chi(A(w)) description} that \begin{equation}
\op{Spec}\big(L_W^T(\bm\omega)\big)=\Big\{ \sum_{i\geq 1}\omega_i\cdot (c^{}_{\chi_i}-c^{}_{\chi_{i-1}})\quad |\ \towerchi\in\op{Res}_T(V_{ref})\Big\}.\label{EQ: Spec(L_W^T(w)) via c_chi}
\end{equation}

We have to compute first the Bratteli diagram for the tower $T$ under the reflection representation $V_{ref}$ of $W$. To do that, notice that the restriction on a parabolic subgroup $W_X$ decomposes as a sum of reflection and trivial representations $$\op{Res}^W_{W_X}(V_{ref})=\chi_{\op{ref}(W_{X_1})}+\cdots +\chi_{\op{ref}(W_{X_s})}+\op{codim}(X)\cdot \op{triv},$$ where the $W_{X_i}$'s are the irreducible components of $W_X$ (that is, $W_X=W_{X_1}\times\cdots\times W_{X_s}$ and they are all parabolic subgroups of $W$). Now, each parabolic subgroup $W_{X'}\leq W_X$ determines also parabolic (but not always proper) subgroups of each of the components $W_{X_i}$ of $W_X$. In particular, the new flats satisfy $X'_{i}\supset X_i$ and this means that if $W_{X'}$ is of maximal rank in $W_X$, then there is a unique flat $X_i$ such that $X'_i\neq X_i$. 

This means that for a tower $T=(W_i)_{i=1}^n$, all tuples $\towerchi$ of characters that appear in $\op{Res}_T(V_{ref})$ consist of reflection representations of irreducible subgroups of the $W_i$. To be more precise, $T$ determines a sequence $(G_j)_{j=1}^n$ of \emph{irreducible} subgroups of $W$, such that $G_j$ is the unique component of $W_j$ not contained in $W_{j-1}$ (see the example of \eqref{EQ: Example of eigenvalues}). Moreover, each pair $(G_j,W_i)$ (with $i\geq j$) determines a subgroup $G^i_j\leq W$ as the unique irreducible component of $W_i$ that contains $G_j$. Now, the tuples $\towerchi$ are given precisely by $$\op{Res}_T(V_{ref})=\Big\{ (\underbrace{\op{triv},\cdots,\op{triv}}_{\op{rank}(G_j)\text{-times}},\op{ref}(G_j),\op{ref}(G_j^{j+1}),\cdots, \op{ref}(G_j^n)=V)\quad |\ j=1\dots n\Big\}.$$  
In particular all numbers $c_{\chi_i}$ that appear in \eqref{EQ: Spec(L_W^T(w)) via c_chi} are \emph{usual} Coxeter numbers, as in the definition \eqref{Eq: Coxeter number h}. We have now proven the following combinatorial interpretation for the eigenvalues of the (weighted) $W$-Laplacian.

\begin{proposition}\label{Prop: comb. descr. of Spec of L_W^T(w)} 
With notation as above, if we write $h^T_{j,i}$ for the Coxeter numbers of the irreducible reflection groups $G_j^i$ determined by $T$, then the eigenvalues of the $W$-Laplacian $L_W^T(\bm\omega)$ \mbox{are given by}
$$\op{Spec}\big(L_W^T(\bm\omega)\big)=\Big\{ \sum_{i\geq j} \omega_i\cdot (h^T_{j,i}-h^T_{j-1,i})\quad  |\ j=1\dots n\Big\}.$$
\end{proposition}

\subsection{A practical algorithm}

The above discussion leads to a practical algorithm to compute the eigenvalues of the $W$-Laplacian $L_W^T(w)$ for a given tower $T$. Indeed, Proposition~\ref{Prop: comb. descr. of Spec of L_W^T(w)} shows that it suffices to 
compute the $n\times n$, upper-triangular, integer matrix $H=(h^T_{j,i}-h^T_{j-1,i})_{1\leq j\leq i \leq n}$.

For practical applications, consider the case of a tower $T$ of \emph{standard} parabolic subgroups, specified by an ordering of the generators of $W$. That is, $W=\langle s_1,\dots,s_n\rangle$ and  $
W_i = \langle s_1,\dots, s_i \rangle$ for $i\in \{1,\dots,n\}$.
For each $i\in \{1,\dots,n\}$, we let as above $G_i$ be the irreducible component of $W_i$ containing $s_i$. To compute the matrix $H$, we fill  iteratively its columns from left to right as follows:
For each $j\in \{1,\dots,n\}$, and for each $i\in \{1,\dots,n\}$, if $j\leq i$ and $s_j$ does belongs to $G_i$, we choose the value of $H_{i,j}$ such that the partial sum of the $i$-th row up to this entry equals the Coxeter number of $G_i$. For other values of $j$, we set $H_{i,j}=0$.

To see that this algorithm indeed computes the wanted matrix $H$, just notice that the partial sum up to $i$ of the quantity $(h^T_{j,i}-h^T_{j-1,i})$ is equal to the Coxeter numer $h^T_{i,i}$ of $G_i$. We observe that in practice, specifying an ordering of the generators amounts to numbering the vertices of the Dynkin diagram of $W$, or equivalently  constructing an increasing sequence of Dynkin diagrams in which vertices are added one by one. Irreducible components in this setting correspond to connected components, which make the above algorithm very practical to run by hand, as we now illustrate.

\subsection{An example in detail}

We present here in detail the situation for a tower $T$ of $D_6$, whose groups $W_i$ are described by the labels of the columns in \eqref{EQ: Example of eigenvalues}; they have isomorphism types $$\big(\{\mathbbl{1}\},A_1,A_2,A_2\times A_1,A_3\times A_1, A_3\times A_2,D_6\big).$$

Each group $W_i$ is determined by a subdiagram of $D_6$, and the irreducible subgroup $G_i\leq W_i$ corresponds to the circled connected component. Notice that this is exactly the component that contains the new vertex added in the $i$-th column. The matrix $H$ defined above is shown in~\eqref{EQ: Example of eigenvalues}.
Each row corresponds to a vertex of the diagram and each column to an ``active" component $G_i$. The numbers of the $i$-th column are such so that the row-sums \emph{up to that column and only for those rows that correspond to the vertices of $G_i$} equal the coxeter number of $G_i$.

So, for example, in the fourth column of \eqref{EQ: Example of eigenvalues} the active component $G_4$ is of type $A_3$ with Coxeter number $4$. The three vertices of its type-$A_3$ diagram correspond to rows $1,2,$ and $4$ (which is when those vertices were added). The rows $1$ and $2$ already sum to the number $3$, so the entries $(1,4)$ and $(2,4)$ are both $1$. The new entry $(4,4)$ is $4$ and row $3$ gets a zero entry since the connected component containing the third added vertex is unaffected in this column. 

\newcommand{\myPic}[1]{\underline{\begin{minipage}[b]{16mm}\begin{center}{\includegraphics{#1}}\end{center}\end{minipage}}}

\begin{equation}
  \begin{blockarray}{c@{\hspace{20pt}}cccccc@{\hspace{20pt}}}
	  & \matindex{\myPic{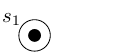}} & \matindex{\myPic{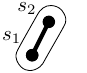}} & \matindex{\myPic{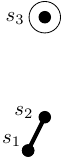}} & \matindex{\myPic{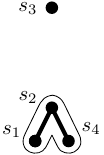}} & \matindex{\myPic{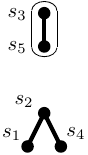}} & \matindex{\myPic{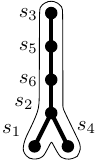}} \\
    \begin{block}{[c@{\hspace{20pt}}cccccc@{\hspace{20pt}}]}
      & {2} &  {1} & {0} & {1} & {0} & {6} &  \matindex{\ $\lambda_1$} \\  
      & {0} &  {3} & {0} & {1} & {0} & {6} &  \matindex{\ $\lambda_2$} \\  
      & {0} &  {0} & {2} & {0} & {1} & {7} &  \matindex{\ $\lambda_3$} \\  
      & {0} &  {0} & {0} & {4} & {0} & {6} &  \matindex{\ $\lambda_4$} \\  
      & {0} &  {0} & {0} & {0} & {3} & {7} &  \matindex{\ $\lambda_5$} \\  
      & {0} &  {0} & {0} & {0} & {0} & {10} & \matindex{\ $\lambda_6$} \\  
  \end{block}
  \end{blockarray}\label{EQ: Example of eigenvalues}
\end{equation}

In such a matrix then, each column corresponds to the (generalized) Jucys-Murphy element $J_i$ and hence the weight $\omega_i$, and each row to an irreducible subgroup $G_j$ and hence to an eigenvalue $\lambda_j(\bm\omega)$ of the $W$-Laplacian. For our case here, the eigenvalues will be given by
$$\{\lambda_j(\bm \omega)\quad |\ j=1\dots n \}=\begin{Bmatrix} 2\omega_1+\omega_2+\omega_4+6\omega_6,\\ 3\omega_2+\omega_4 +6\omega_6,\\2\omega_3+\omega_5 + 7\omega_6,\\ 4\omega_4+6\omega_6,\\ 3\omega_5+7\omega_6,\\ 10\omega_6 \end{Bmatrix}.$$

\section{Lie-like elements and equivalence of Theorems~\ref{thm:main} and~\ref{thm:equivalence}}
\label{sec:LieLike}

\subsection{Lie-like elements}
\label{subsec:LieLike}

Theorem~\ref{thm:equivalence} relates the virtual characters appearing in the Frobenius formula to the exterior powers $\bigwedge^k(V)$ of the reflection representation $V$. Burman and Zvonkine~\cite{BZ} were the first to observe, in the context of type $A_n$, that reflections give rise to ``Lie-like elements'' on those exterior powers. We now quickly recall this notion, which was formalized in a later manuscript of Burman~\cite{Burman} (see also the more recent \cite{burman2020lie} where further relations with matrix tree theorems are produced).

\begin{definition}[Lie-like element, \cite{BZ, Burman}]
	Let $W$ be a group, let $x \in \mathbb{C}[W]$ be an element of the group algebra, and let $U$ be a representation of $W$. We say that $x$ is a \emph{Lie-like element for the representation $U$} if for any $k \geq 0$, and for any element $v_1\wedge \dots \wedge v_k  \in \bigwedge^k U$, we have
\begin{align}
	\label{eq:LieLike}
	x \cdot (v_1 \wedge \dots \wedge v_k) = \sum_{i=1}^k v_1 \wedge \dots \wedge (x \cdot v_i) \wedge \dots \wedge v_k.
\end{align}
\end{definition}
To prevent any confusion, we stress that in the left-hand side of~\eqref{eq:LieLike} we use the classical action of the group algebra $\mathbb{C}[W]$ on $\bigwedge^k U$, given by
	$$g \cdot (v_1 \wedge \dots \wedge v_k) =(g \cdot v_1) \wedge \dots \wedge (g\cdot v_k) $$
	if $g \in W$ is an element of the group, and extended \emph{linearly} to the group algebra $\mathbb{C}[W]$. In particular, the action of some $x\in \CC[W]$ on $\bigwedge^k U$ and the action of the \emph{matrix} $\rho^{}_U(x)\in\op{End}(U)$ on $\bigwedge^k U$ are often different. Note that it is clear that Lie-like elements on $U$ form a vector space (for interesting questions related to this vector space, which is in fact a Lie-algebra, see~\cite{Burman}).

	The following three statements are straighforward adaptations of corresponding statements for type $A_n$  in~\cite{BZ}. We state them not only for general reflection groups, but for the more general situation of an arbitrary representation $U$. This will be useful for the complements given in Section~\ref{Sec: product formulas}.
\begin{lemma}[{\cite[Proposition~2.2]{BZ}}]
	Let $G$ be a group, $U$ a representation of $G$, and let $g \in G$ such that $g$ acts as a pseudo-reflection on $U$. Then the element $\mathbbl{1}-g\in \mathbb{C}[G]$ is Lie-like for $U$.
\end{lemma}
\begin{proof}
	The proof of~\cite[Proposition~2.2]{BZ} extends trivially to our setting, we adapt it here for completeness. 
	Let $k \geq 1$ and let $n$ be the dimension of $U$. Since $g$ acts as a pseudo-reflection, we can choose a basis $(x_1,\dots,x_n)$ of $U$ such that $g\cdot x_1 = \alpha x_1$ for some root of unity $\alpha$, and $g \cdot x_i = x_i$ for $2 \leq i\leq n$.  It is a direct check that the two actions $\bigwedge^k U \longrightarrow \bigwedge^k U$ given by:
	\begin{align*}
		v_1 \wedge \dots \wedge v_k \longmapsto 
	(\mathbbl{1}-g) \cdot ( v_1 \wedge \dots \wedge v_k ) \ \mbox{ and } 
v_1 \wedge \dots \wedge v_k \longmapsto	
	\sum_{i=1}^k v_1 \wedge \dots \wedge ((\mathbbl{1}-g) \cdot v_i) \wedge \dots \wedge v_k
	\end{align*}
both act as follows, for integers $1<i_2,\dots, i_{k+1} \leq n$:
\begin{align*}
	x_1 \wedge x_{i_2} \wedge \dots \wedge x_{i_k} &\longmapsto
	(1-\alpha) x_1 \wedge x_{i_2} \wedge \dots \wedge x_{i_k}, \\ 
	x_{i_2} \wedge x_{i_3} \wedge \dots \wedge x_{i_{k+1}} &\longmapsto 0. 
\end{align*}
	We have thus checked the equality~\eqref{eq:LieLike} on a basis of $\bigwedge^k U $.
\end{proof}
The next two statements and their proofs are  adaptations of \cite[Proposition~2.4]{BZ} to our more general situation.
\begin{lemma}\label{lemma:LieLikeEigenvalues}
	Let $G$ be a group, $U$ a representation of $G$ of dimension $n$, and let $x \in \mathbb{C}[W]$ which is Lie-like for $U$.
Then for any $k\geq 1$, the  ${n \choose k}$ eigenvalues of $x$ on the representation $\bigwedge^k U$ are given by the sums
	$$
	\sigma_{i_1}+\sigma_{i_2}+\dots+\sigma_{i_k}
	$$
	for all  $1\leq i_1<i_2<\dots<i_k\leq n$, where $\sigma_1, \sigma_2, \dots, \sigma_n$ are the eigenvalues of $x$ in the representation~$U$.
\end{lemma}
\begin{proof}
	Assume $x$ is diagonalisable over $U$, and let $(v_1, \dots, v_n)$ be a $U$-basis of eigenvectors of $x$ adapted to the eigenvalues $(\sigma_1,\dots,\sigma_n)$. Then for any $i_1<\dots<i_k$, the right-hand side of \eqref{eq:LieLike} shows that $v_{i_1}\wedge\dots \wedge v_{i_k} \in \bigwedge^k U$ is an eigenvector of $x$ of eigenvalue $\sigma_{i_1}+\sigma_{i_2}+\dots+\sigma_{i_k}$. This gives a complete basis of eigenvectors. If $x$ is not diagonalisable the argument is the same, starting from a $U$-basis in which $x$ is upper triangular and obtaining a basis of $\bigwedge^k U$ in which $x$ also is.
\end{proof}

Since Lie-like elements form a vector space, the two lemmas imply the following corollary which is a key element of the paper.
\begin{corollary}\label{cor:BZ}
	Let $W\subset\op{GL}(V)$ be a complex reflection group of rank $n$, equipped with an underlying weight function $\mathbf{w}_P$ as in Section~\ref{subsec:Laplacians}. Then the group algebra  Laplacian  $\mathbf{B}(\bm\omega)$
 is Lie-like for~$V$, and for $k\geq 1$ its ${n \choose k}$ eigenvalues  on the representation $\bigwedge^k V$ are given by the sums
	$$
	\sigma_{i_1}+\sigma_{i_2}+\dots+\sigma_{i_k}
	$$
	for all  $1\leq i_1<i_2<\dots<i_k\leq n$, where $\sigma_1, \sigma_2, \dots, \sigma_n$ are the eigenvalues  of the $W$-Laplacian $L_W(\bm\omega) \in  GL(V)$.
\end{corollary}

When all weights are equal to $1$, the previous corollary together with Proposition~\ref{Prop: L_W has all eigenvalues =h} implies that all eigenvalues of the (unweighted) element $\mathbf{B}$ are equal to $kh$, which gives another proof of 
Proposition~\ref{prop:wedgeCoxeterNumbers}.

\subsection{Equivalence of the two main results for a given group $W$}

We are now ready to show that, for a given group $W$, the statements of Theorem~\ref{thm:equivalence} and Theorem~\ref{thm:main} are equivalent. In the rest of this section we fix a complex reflection group $W$.

First, we assume that the statement of Theorem~\ref{thm:equivalence} holds for a given group $W$, and we will deduce Theorem~\ref{thm:main}. We use the notation of Section~\ref{subsec:Laplacians}. 
By Lemma~\ref{Thm: Frobenius a la Guillaume}, we have
$$
\fac_W^T(t,\bm{\omega})=\dfrac{|\mathcal{C}|}{|W|}\cdot \sum_{\chi\in \widehat{W}}\chi(c^{-1})\cdot \chi\big(e^{t\mathbf{A}(\bm\omega)}\big).
$$ 
We note that $\mathbf{A}(\bm\omega)= \sum_{i=1}^n \omega_i J_{T,i}$ clearly belongs to the Jucys-Murphy algebra $\mathbb{C}[\mathbf{J}_T]$. Since we assumed Theorem~\ref{thm:equivalence}, we can use tower equivalence and rewrite this expression as
	\begin{align*}
		\fac_W^T(t,\bm\omega)=\dfrac{|\mathcal{C}|}{|W|}\cdot \sum_{k=0}^n(-1)^k\chi_{\bigwedge^k(V)}\big(e^{t\mathbf{A}(\bm\omega)}\big).
	\end{align*}
	Now, by~\eqref{Eq: B(w)=wt(R)-A(w)}, and by Corollary~\ref{cor:BZ} we know the complete list of eigenvalues of $\mathbf{A}(\bm\omega)$ on $\bigwedge^k V$, which are given by the quantities 
	$
{\bf w}_T^{}(\mathcal{R}) - (\sigma_{i_1}+\dots + \sigma_{i_k}),
	$ 
	for $i_1<\dots<i_k$, where $(\sigma_1, \dots, \sigma_n)$ are the eigenvalues  of the $W$-Laplacian $L_W(\bm\omega) \in  GL(V)$. We directly obtain
	\begin{align*}
\fac_W^T(t,\bm\omega)
	&= \dfrac{|\mathcal{C}|}{|W|}\cdot  
		\sum_{k=0}^n (-1)^k  \sum_{1\leq i_1<\dots<i_k \leq n} e^{t ({\bf w}_T^{}(\mathcal{R}) -\sigma_{i_1}-\dots -\sigma_{i_k})} \\
			&= \frac{1}{h}e^{t{\bf w}_T^{}(\mathcal{R})} \prod_{i=1}^n (1-e^{-t \sigma_{i}}), 
	\end{align*}
	where we have used that $\dfrac{|\mathcal{W}|}{|\mathcal{C}|} =h$ from \eqref{EQ: |C|=|W|/h}. This is precisely the statement of Theorem~\ref{thm:main} for $W$.

Let us now prove the converse statement. Assume that Theorem~\ref{thm:main} holds for a given group $W$. By the same calculation, this is equivalent to saying that for any parabolic tower $T$, we have, in the notation of Section~\ref{subsec:Laplacians},
\begin{align}\label{eq:identity}
\sum_{\chi\in \widehat{W}}\chi(c^{-1})\cdot \chi\big(e^{t\mathbf{A}(\bm\omega)}\big)
=
\sum_{k=0}^n(-1)^k\chi_{\bigwedge^k(V)}\big(e^{t\mathbf{A}(\bm\omega)}\big).
\end{align}
Recall that $\mathbf{A}(\bm\omega)= \sum_{i=1}^n \omega_i J_{T,i}$. 
We now expand the identity~\eqref{eq:identity} in powers of $t$ and the parameters $\omega_1, \dots, \omega_n$. For $i_1,\dots,i_n\geq0$ we extract the coefficient of $t^{i_1+\dots+i_n}\omega_i^{i_1}\dots\omega_n^{i_n}$. We get
\begin{align*}
	\sum_{\chi\in \widehat{W}}\chi(c^{-1})\cdot \chi\big(
	J_{T,1}^{i_1}\dots J_{T,n}^{i_n}
	\big)
=
\sum_{k=0}^n(-1)^k\chi_{\bigwedge^k(V)}\big(J_{T,1}^{i_1}\dots J_{T,n}^{i_n}\big).
\end{align*}
Since the $J_{T,1}^{i_1}\dots J_{T,n}^{i_n}$ are a (vector space) basis of the Jucys-Murphy algebra $\mathbb{C}[\mathbf{J}_T]$, the last statement is equivalent to saying that 
\begin{align*}
	\sum_{\chi\in \widehat{W}}\chi(c^{-1})\cdot \chi\towerequiv
\sum_{\chi\in \widehat{W}}\sum_{k=0}^n(-1)^k\chi_{\bigwedge^k(V)},
\end{align*}
which is precisely the statement of Theorem~\ref{thm:equivalence} for the group $W$.

\section{Proof of Theorem~\ref{thm:equivalence} for the infinite families $G(r,1,n)$ and $G(r,r,n)$}
\label{sec:infiniteFamilies}

In this section, we prove Theorem~\ref{thm:equivalence} for all well generated groups in the infinite family $G(r,p,n)$. These are precisely the groups comprising the subfamilies $G(r,1,n)$ and $G(r,r,n)$; notice that the remaining groups are not well generated and hence do not have Coxeter elements (see Defn.~\ref{Defn: Coxeter element}).

\subsection{Induction and parabolic subgroups}

The proofs for $G(r,1,n)$ and $G(r,r,n)$ have a similar structure, that deeply exploits the recursive nature of towers.
They rely on an induction based on the following lemma. To simplify its statement we first introduce a notion of decomposition of virtual characters.

\begin{definition}\label{Defn. decomp-via-cox-nums}
Let $W$ be an irreducible complex reflection group, with Coxeter number $h$ and rank $n$, and let $\chi=\sum \lambda_i\chi_i$ be a virtual character of $W$ (a sum of irreducible characters $\chi_i\in\widehat{W}$ with $\lambda_i\in\mathbb{C}$). We define the {\color{blue}\emph{decomposition via Coxeter numbers}} of $\chi$ to be the expression 
\[
\chi=\sum_{s=0}^{hn}\chi_s,
\]
where each $\chi_s$ is the virtual character made up of all components of $\chi$ that have Coxeter number $s$. That is,
\[
\chi_s:=\sum_{c_{\chi_i}=s}\lambda_i\chi_i.
\]
Notice that we have used here Prop.~\ref{Prop: coxeter numbers are integers} which states that Coxeter numbers are integers and always belong to the interval $[0,hn]$.
\end{definition}

The following key lemma allows us to prove that two virtual characters $\chi$ and $\psi$ are tower equivalent by checking that all component $\chi_s$ and $\psi_s$ as defined above are themselves tower equivalent in a certain collection of smaller towers.

\begin{lemma}\label{lemma:subgroup}
	Let $W$ be a complex reflection group and let $\chi, \psi$ be two virtual characters of $W$, viewed as linear functions $W\longrightarrow \mathbb{C}$. Write $\chi=\sum_{s=0}^{hn}\chi_s$ and $\psi=\sum_{s=0}^{hn}\psi_s$ for the decompositions via Coxeter numbers of $\chi$ and $\psi$ as in Defn.~\ref{Defn. decomp-via-cox-nums}. Assume that for each number $s=0,\ldots, hn$ and each proper maximal parabolic subgroup $W_{sub}\subsetneq W$, the restrictions of $\chi_s$ and $\psi_s$ to $W_{sub}$ are tower equivalent in the sense of $W_{sub}$. Then, $\chi$ and $\psi$ are tower equivalent in the sense of $W$. 
\end{lemma}
\begin{proof}
    We will show that $\chi_s$ and $\psi_s$ are tower equivalent in the sense of $W$ (this appears to be stronger than what is required, but it is equivalent to the statement of the lemma after Prop.~\ref{Prop: formal stronger tower equivalence}).
	For any parabolic tower $T=\big(\{\mathbbl{1}\}=W_0\leq W_1\leq \cdots \leq W_n=W \big)$,  let $W_{sub}=W_{n-1}$ and consider the parabolic tower $T'=\big(\{\mathbbl{1}\}=W_0\leq W_1\leq \cdots \leq W_{n-1}=W_{sub} \big)$. Then the Jucys-Murphy algebra $\mathbb{C}[\bm J_T]$ is generated by $\mathbb{C}[\bm J_{T'}]$ and by the sum $\mathbf{R}$ of all reflections of $W$. Because $\mathbb{C}[\bm J_T]$ is commutative, it is enough to confirm that $\chi_s$ and $\psi_s$ agree on all monomials of the form $\mathbf{R}^k\cdot a$ where $a$ is any arbitrary element of $\mathbb{C}[\bm J_{T'}]$.
	
	Now, the \emph{virtual} characters $\chi_s$ and $\psi_s$ are not necessarily irreducible but they are constructed in a way that allows us to factor our the contribution of $\mathbf{R}$. To clarify, we have for any element $a$ in the algebra $\mathbb{C}[\bm J_{T'}]$ that
	\[
	\chi_s(\mathbf{R}^k\cdot a)=\big(|\mathcal{R}|-s\big)^k\cdot \chi_s(a).
	\]
	To see this, write $\chi_s=\sum_{i=1}^r\lambda_i\chi_{s,i}$ for the decomposition of $\chi_s$ as a sum of irreducible characters $\chi_{s,i}$. Indeed now, we will have
	\[
	\chi_s(\mathbf{R}^k\cdot a)=\sum_{i=1}^r\lambda_i\chi_{s,i}(\mathbf{R}^k\cdot a)=\sum_{i=1}^r\lambda_i\big(\widetilde{\chi}_{s,i}(\mathbf{R})\big)^k\cdot \chi_{s,i}(a)=\sum_{i=1}^r\lambda_i\big(|\mathcal{R}|-s\big)^k\cdot\chi_{s,i}(a)=\big(|\mathcal{R}|-s\big)^k\cdot\chi_s(a),
	\]
	where the second equality is because the $\chi_{s,i}$ are \emph{irreducible} characters and $\mathbf{R}$ is a central element, and where the third equality is by the definitions of the components $\chi_s$ (Defn.~\ref{Defn. decomp-via-cox-nums}) and of the generalized Coxeter numbers (Defn.~\ref{Defn: Coxeter numbers}). 
	
	The exact same relation will be true for $\psi_s$ and since by our assumption $\chi_s$ and $\psi_s$ agree on any element $a\in\mathbb{C}[\bm J_{T'}]$, they must now agree on all of $\mathbb{C}[\bm J_T]$ as well. This completes the proof.	
\end{proof}

In order to use the lemma, we need to know the maximal parabolic subgroups of the families we work with. We have (see for instance \cite[Appendix A]{broue-book-braid-groups}):
\begin{lemma}\label{lemma:parabolics}
	1)	Every maximal parabolic subgroup of $G(r,1,n)$ is conjugate to a standard parabolic subgroup, which is either:
	1a: the parabolic subgroup $S_n$, or 1b: the parabolic subgroup $G(r,1,a)\times S_{b}$ for $a=1,\cdots,n-1$ and $b=n-a$. Restriction on this subgroup is given following the chain:$$G(r,1,a)\times S_{b}\hookrightarrow G(r,1,a)\times G(r,1,b)\hookrightarrow G(r,1,n).$$

	2)   Every maximal parabolic subgroup of $G(r,r,n)$ is conjugate in $G(r,1,n)$ to a standard parabolic subgroup of $G(r,r,n)$, which is either:
2a: the parabolic subgroup $S_n$;
or 2b: the parabolic subgroup $G(r,r,a)\times S_{b}$ for $a=2,...,n-1$ and $b=n-a$. Restriction on this subgroup is given following the chain: $$G(r,r,a)\times S_{b}\hookrightarrow G(r,r,a)\times G(r,r,b)\hookrightarrow G(r,r,n).$$
\end{lemma}

\subsection{Partitions and representations}
 We will use the combinatorial representation theory of the groups $\mathfrak{S}_n, G(r,1,n)$, and $G(r,r,n)$. We refer to~\cite{Sagan} for  symmetric groups, and to~\cite{CS, Stembridge} for what we need of the case  the two other families.

Recall that a \emph{partition} of an integer $n$ is a  nonincreasing sequence of positive integers $\lambda=(\lambda_1 \geq \lambda_2 \geq ... \geq \lambda_\ell>0)$ summing up to $n$. We include in this definition the empty partition $\varnothing$ (of zero). We write $\lambda \vdash n$ if $\lambda$ is a partition of $n$.
Irreducible representations of $\mathfrak{S}_n$ are classically indexed by partitions of $n$, while representations of $G(r,1,n)$ are indexed by $r$-tuples of partitions of total size~$n$. By restriction, any irreducible representation of $G(r,1,n)$ is also a representation of $G(r,r,n)$, but it is not necessarily irreducible : it splits into as many irreducible components as the order of symmetry of the $r$-tuple of partitions indexing it, under cyclic shifts of its entries.
In the following discussion, we abuse terminology and often identify partitions, or $r$-tuples of partitions, with the irreducible representation of character they index.


The only irreducible representations we will need to consider in detail are the ones whose characters do not vanish on Coxeter elements, which turn out to be relatively few. For $n\geq 1$, define the following partitions of the integer $n$,
	\begin{itemize}
		\item $\hook{n}{k} := (n-k,1^k)$ be the \emph{hook} partition of height $k+1$, for $0\leq k <n$;
		\item $\qhook{n}{k}:=(n-k-1,2,1^{k-1})$ be the \emph{quasi-hook} of height $k+1$, for $1\leq k \leq n-3$.
	\end{itemize}

Throughout Section~\ref{sec:infiniteFamilies} we will use the following convention.
\begin{notation}\label{not:zero}
	A hook or quasihook character indexed by non-combinatorial parameters is defined to be zero, i.e. $\hook{u}{v}:=0$ if $v\geq u$ or $v<0$, and $\qhook{u}{v}:=0$ if $v\geq u-1$ or $v<1$.
\end{notation}
This convention has two advantages: first, it enables us to avoid distinguishing boundary cases when writing sums (see e.g.~\eqref{eq:induction1} or \eqref{eq:inductionGrrn}). Second, it is well adapted to the use of combinatorial rules in representation theory, for example when a generic operation of ''removing a box'' may not be possible, and thus not contribute to any term in the summations, for boundary values of some parameters.

\subsection{The case of $\mathfrak{S}_n$}

As a warmup, we consider the case of $\mathfrak{S}_n$, acting on $V=\mathbb{C}^{n-1}$.
It is well known (see e.g.~\cite{BZ}) that the only irreducible characters that do not vanish on the Coxeter elements are the hooks, and that the hook $\hook{n}{k}$ coincides with the exterior power $\bigwedge^k(V)$. This character takes the value $(-1)^k$ on the Coxeter element and has Coxeter number $hk$, with $h=n$, see e.g. \cite{CS}. Thus~\eqref{eq:filterCoxeterNumbers} holds, which is all we need to prove by  Corollary~\ref{cor:filterCoxeterNumbers}. Note that in the case of $\mathfrak{S}_n$ the equivalence $\towerequiv$ in~\eqref{eq:filterCoxeterNumbers} is, in fact, an equality.

\subsection{The case of $G(r,1,n)$}

\begin{lemma}[{\cite[Sec. 5.3]{CS}}]\label{lemma:listRepsGr1n}
	The only irreducible representations of $G(r,1,n)$ whose character does not vanish on the Coxeter element are the $$
\ghook{n}{k}{q} := (0,...,0,\hook{n}{k},0,...0) \text{ for } 0 \leq q < r \text{ and } 0 \leq k < n,
$$
where the hook $\hook{n}{k}$ appears in $q$-th position.
	Their characters satisfy $\chi_{\ghook{n}{k}{q}}(c^{-1})=(-1)^k \xi^{-q}$ for a primitive $r$-root of unity $\xi$, and their Coxeter numbers are $hk$ if $q=0$ and $h(k+1)$ if $q\neq 0$.
\end{lemma}
\begin{proof}
	The list of representations and their character is given in \cite[Sec 5.3, Step 1-2]{CS}. 
	The Coxeter number can be computed by the explicit evaluation of the normalized character on each conjugacy class of reflections, given in\cite[Sec 5.3, Step 3]{CS}.
\end{proof}

As an additional point of clarity (which we will not use however) we remind the reader that for $G(r,1,n)$, the exterior powers of the reflection representation are indexed by the following $r$-tuples of partitions:
\[
{\bigwedge}^k(V)=\big(n-k,1^k,0,\cdots,0\big).
\] 

Now, the previous  lemma shows that the only characters we need to consider have Coxeter numbers which are multiples of $h$. By Lemma~\ref{lemma:subgroup}, Theorem~\ref{thm:equivalence} for $G(r,1,n)$ is therefore equivalent to the following statement (relying also on Lemma~\ref{lemma:listRepsGr1n} and Prop.~\ref{prop:wedgeCoxeterNumbers} to compare Coxeter numbers). 
For the purpose of using induction, we include the case $n=1$.
\begin{proposition}\label{prop:Gr1n}
	For each $n\geq 1$ and for each $ 0 \leq k < n$ consider the following virtual character of $G(r,1,n)$: 
\begin{align*}
	\mathfrak{P}_{n,k}:=\ghook{n}{k}{0} -  \sum_{0<q<r} \xi^{-q} \cdot \ghook{n}{k-1}{q},
\end{align*}
	where we recall Notation~\ref{not:zero}.
Then we have
\begin{align}\label{eq:inductionGr1n}
	\mathfrak{P}_{n,k} \towerequiv {\bigwedge}^k (V),
\end{align}
where $V=\mathbb{C}^n$ denotes the reflection representation of $G(r,1,n)$.
\end{proposition}
\begin{proof}
	We will apply induction on $n$.

	For $n=1$, $G(r,1,1)$ is isomorphic to the cyclic group $C_r$. Since it is of rank $1$, the weighting system is trivial, and Theorem~\ref{thm:main} is equivalent to the unweighted Chapuy-Stump formula~\cite{CS}. Because Theorems~\ref{thm:main} and~\ref{thm:equivalence} are equivalent (see Section~\ref{sec:LieLike}), we are done.

	Let $n\geq 2$. We will check the conditions of Lemma~\ref{lemma:subgroup} with $\chi_s$ and $\psi_s$ being the two sides of~\eqref{eq:inductionGr1n} (corresponding to Coxeter numbers $s=kh$). Let $W_{sub}$ be a maximal proper parabolic subgroup of $W=G(r,1,n)$.
	The conditions of Lemma~\ref{lemma:subgroup} being conjugation invariant, we can assume that $W_{sub}$ is given by either case 1a or 1b of Lemma~\ref{lemma:parabolics}.

	We start with case 1b and we let
	$
W_{sub}=G(r,1,a)\times \mathfrak{S}_{b}
$
for some $a\in [1..n-1]$ and $b=n-a$.
	Irreducible representations of $W_{sub}$ are indexed by pairs $U\otimes V$ where $U$ is an $r$-tuple of partitions of total size $a$, and $V$ is a partition of size $b$. We now examine the restriction to $W_{sub}$ of both sides of~\eqref{eq:inductionGr1n}.

We start with the L.H.S. By the ``type B'' Littlewood-Richardson rule (see Appendix~\ref{appsec:LRGr1n}), we have for $\ell\geq 0$
$$
\ghook{n}{\ell}{q}\Big\downarrow_{W_{sub}} = \sum_{\alpha \vdash a \atop \beta \vdash b} c_{\hook{n}{\ell}}^{\alpha, \beta} (0,...,0,\alpha,0,...0) \otimes \beta,
$$
	where $c_{\hook{n}{\ell}}^{\alpha, \beta}$ is the $\mathfrak{S}_n$-Littlewood-Richardson (LR) coefficient, and where the partition $\alpha$ appears in position $q$ inside the $r$-tuple. Using the LR-rule to determine these coefficients (see~\eqref{appeq:LRhook2} in Appendix~\ref{appsec:LRhooks}) we get
	\begin{align}
	\ghook{n}{\ell}{q}\Big\downarrow_{W_{sub}} &=
	\sum_{i,j\geq 0, \epsilon \in \{0,1\} \atop i+j+\epsilon=\ell }
	(0,...,0,\hook{a}{i},0,...0) \otimes  \hook{b}{j} 
	= \sum_{i,j\geq 0, \epsilon \in \{0,1\} \atop i+j+\epsilon=\ell }  \ghook{a}{i}{q} \otimes  \hook{b}{j} \nonumber. 
\end{align}
	Using this identity for $\ell$ equal to $k$ and $k-1$ (respectively with $q=0$ and $q\neq 0$, and making the change of index $i+1\mapsto i$ in the second case) we obtain by taking a linear combination the ``inductive'' relation
\begin{align}\label{eq:induction1}
	\mathfrak{P}_{n,k}\Big\downarrow_{W_{sub}}  
	&= \sum_{i,j\geq 0, \epsilon \in \{0,1\} \atop i+j+\epsilon=k }  \mathfrak{P}_{a,i} \otimes \hook{b}{j}. 
\end{align}

We now consider the R.H.S. of~\eqref{eq:inductionGr1n}. We could again apply the LR-rule, but this is not necessary. Indeed it is clear that
\begin{align}\label{eq:induction2}
	{\bigwedge}^k ( \mathbb{C}^n )\Big\downarrow_{W_{sub}} 
	&= \sum_{i\geq 0} 
	{\bigwedge}^i ( \mathbb{C}^{a}) \otimes {\bigwedge}^{k-i} (\mathbb{C}^{b}).
\end{align}
	The reflection representation of $\mathfrak{S}_{b}$ is $\mathbb{C}^{b-1}$. We have, as $\mathfrak{S}_{b}$ representations,
	\begin{align}\label{eq:wedges}
		{\bigwedge}^{k-i} (\mathbb{C}^{b}) ={\bigwedge}^{k-i} \big(\mathbb{C}^{b-1} \oplus \mathbb{C}\big) 
		={\bigwedge}^{k-i-1}(\mathbb{C}^{b-1}) + {\bigwedge}^{k-i}(\mathbb{C}^{b-1})= \hook{b}{k-i-1}+\hook{b}{k-i}.
	\end{align}
Thus~\eqref{eq:induction2} rewrites
\begin{align}\label{eq:induction3}
	{\bigwedge}^k ( \mathbb{C}^n )\Big\downarrow_{W_{sub}} 
	&=
	\sum_{i,j\geq 0, \epsilon \in \{0,1\} \atop i+j+\epsilon=k }
	{\bigwedge}^i ( \mathbb{C}^{a}) \otimes\hook{b}{j}.
\end{align}

Assuming the induction hypothesis, and comparing \eqref{eq:induction1} and \eqref{eq:induction3}, we obtain that both sides of~\eqref{eq:inductionGr1n} are tower equivalent in the sense of $W_{sub}$. Moreover, both  have Coxeter number $kh$, hence they take the same value on the sum of all reflections. 

	Therefore in order to apply Lemma~\ref{lemma:subgroup}, it only remains to check the case 1a of Lemma~\ref{lemma:parabolics}, i.e. $W_{sub}=\mathfrak{S}_n$. But we have, as $\mathfrak{S}_n$-representations
$$
	\ghook{n}{\ell}{q}\Big\downarrow_{\mathfrak{S}_n} = \hook{n}{\ell} = {\bigwedge}^\ell (\mathbb{C}^{n-1}).
$$
	Therefore using~\eqref{eq:wedges} again we have 
	$$\mathfrak{P}_{k,n}\Big\downarrow_{\mathfrak{S}_n}={\bigwedge}^k (\mathbb{C}^{n-1})+{\bigwedge}^{k-1} (\mathbb{C}^{n-1}) = {\bigwedge}^k (\mathbb{C}^{n}).$$

In conclusion, the hypotheses of Lemma~\ref{lemma:subgroup} hold for any choice of the subgroup $W_{sub}$, and we conclude that $\mathfrak{P}_{n,k}	\towerequiv \Lambda^k \mathbb{C}^n$  in the sense of the group $G(r,1,n)$, which concludes the induction step. 
\end{proof}

To conclude this section, we note that we could have given a different proof of Theorem~\ref{thm:main} (hence Theorem~\ref{thm:equivalence}) in the case of $G(r,1,n)$. Indeed, in this case there exists a more refined formula (see Section~\ref{Sec: product formulas}) that does not use the induction structure of towers. We chose the proof given here because it is similar in structure to the one we will give in the next section for $G(r,r,n)$.

\subsection{The case of $G(r,r,n)$}
We start by importing the following information from~\cite[Section 5.4]{CS}.
\begin{lemma}[{\cite{CS}}]\label{lemma:listRepsGrrn}
	The only irreducible representations of $G(r,r,n)$ whose character does not vanish on the Coxeter element are the following:
	\begin{itemize}[itemsep=0pt,topsep=0pt]
		\item The representation $\parA$. It has Coxeter number $0$.
		\item The representation $\parB$. It has Coxeter number $r(n-1)n$.
		\item The representation $\parC$ for $1\leq k \leq n-2$. It has Coxeter number $r(n-1)(k+1)$.
		\item The representation $\parD$ for $0\leq k \leq n-2$ and $1\leq q <r$, where the ``$1$'' appears in position $q$. It has Coxeter number $r(n-1)(k+1)$.
	\end{itemize}
\end{lemma}
\begin{proof}
	The list of representations and their character is given in \cite[Sec 5.4, Step 1-2]{CS}. 
	The Coxeter number can be computed by the explicit evaluation of the normalized character on reflections, given in\cite[Sec 5.4, Step 3]{CS}.
\end{proof}

As an additional point of clarity (which we will not use however) we remind the reader that for $G(r,r,n)$, the exterior powers of the reflection representation are indexed by the following $r$-tuples of partitions (the same ones as for $G(r,1,n)$):
\[
{\bigwedge}^k(V)=\big(n-k,1^k,0,\cdots,0\big).
\]

Among the representations appearing in Lemma~\ref{lemma:listRepsGrrn}, $\parA$ is the only one of Coxeter number zero, and it is the trivial representation, also equal to ${\bigwedge}^0 (\mathbb{C}^{n})$. Similarly, $\parB$ is the only one of Coxeter number $r(n-1)n$, and it is the sign (determinant) representation, which is equal to ${\bigwedge}^n (\mathbb{C}^{n})$.
These two cases already account for the coefficients of $t^0$ and $t^n$ in~\eqref{eq:filterCoxeterNumbers}.
Moreover, since the characters $\parC$ and $\parD$ respectively take the value $(-1)^k$ and $\xi^q(-1)^k$ on the inverse of a Coxeter element (\cite[Lemma 5.7]{CS}), the only thing remaining to prove~\eqref{eq:filterCoxeterNumbers} (hence Theorem~\ref{thm:equivalence}) is the following statement (relying also on Lemma~\ref{lemma:listRepsGrrn} and Prop.~\ref{prop:wedgeCoxeterNumbers} to compare Coxeter numbers).
For the purpose of using induction, we include the case $n=2$.
\begin{proposition}\label{prop:Grrn}
	For each $n,r\geq 2$ and for each $ 0 \leq k < n$, let
	$$\mathfrak{Q}_{n,k}:=\parC+\sum_{0<q<r} \xi^{-q} \cdot \parD.$$
	Then we have, as virtual characters of $G(r,r,n)$: 
\begin{align}\label{eq:inductionGrrn}
	\mathfrak{Q}_{n,k}	\towerequiv - {\bigwedge}^{k+1} (V),
\end{align}
where $V=\mathbb{C}^n$ denotes the reflection representation of $G(r,r,n)$.
\end{proposition}

\begin{proof}	We use the same structure as in the proof of Proposition~\ref{prop:Gr1n}, and apply induction on $n$.

	We start with $n=2$. Theorem~\ref{thm:main} for the dihedral group $G(r,r,2)$ will be proved in Section~\ref{sec:dihedral}. Admitting this for now, using that Theorems~\ref{thm:main} and~\ref{thm:equivalence} are equivalent (see Section~\ref{sec:LieLike}), and Corollary~\ref{cor:filterCoxeterNumbers}, we are done for this case.

	Let $n\geq 3$. We will check the conditions of Lemma~\ref{lemma:subgroup} with $\chi_s$ and $\psi_s$ being the two sides of~\eqref{eq:inductionGrrn} (corresponding to Coxeter numbers $s=(k+1)h$). Let $W_{sub}$ be a maximal proper parabolic subgroup of $W=G(r,r,n)$. By Lemma~\ref{lemma:parabolics} we can assume that $W_{sub}$ is given by case 2a or 2b of that lemma.

	We start with case 2b, and assume that
$
W_{sub}=G(r,r,a)\times \mathfrak{S}_{b}
$
for some $a\in [2..n]$ and $b=n-a$. We now examine the restrictions to this subgroup.
 
We start with $\parC$. By the ``type D'' Littlewood-Richardson rule (see Appendix~\ref{appsec:LRGrrn}), 
and by the computations in Appendix~\ref{appsec:LRquasihooks}, 
	we have, with $\delta_\epsilon$ the indicator function of $\{\epsilon=0\}$,
\begin{align}
	(\qhook{n}{k},...) \Big\downarrow^{G(r,r,n)}_{(G(r,r,a) \times \mathfrak{S}_{b}} &=
	-(\hook{a}{0},...)\otimes(\hook{b}{k}+\hook{b}{k+1}) +
	\sum_{i,j\geq 0, \epsilon \in \{0,1\} \atop i+j+\epsilon=k } (\qhook{a}{i}, ...) \otimes \hook{b}{j} 
	+(\hook{a}{i}, ...) \otimes \qhook{b}{j} \nonumber \\
	& \quad \quad \quad \quad \quad \quad+\sum_{i,j\geq 0, \epsilon \in \{-1,0,1\} \atop i+j+\epsilon=k } (1+\delta_\epsilon) \cdot (\hook{a}{i}, ...) \otimes \hook{b}{j}.
	\label{eq:h1}
\end{align}

We now consider $\parD$, for which both steps of the restriction
$$G(r,r,a)\times S_{b}\hookrightarrow G(r,r,a)\times G(r,r,b)\hookrightarrow G(r,r,n)$$
	come into play. First, by the splitting rule for the restriction $ G(r,r,k)\times G(r,r,n-k)\hookrightarrow G(r,r,n)$ (see \eqref{appeq:parabolicLR2}  and Appendix~\ref{appsec:LRGrrn}) we have
	\begin{align}
		\parD\Big\downarrow^{G(r,r,n)}_{G(r,r,a)\times G(r,r,b)}
		=& \sum_{i,j\geq 0, \epsilon \in \{0,1\} \atop i+j+\epsilon=k }(\hook{a-1}{i}, ..., 1, ...) \otimes (\hook{b}{j} , ...) 
		+
		\sum_{i,j\geq 0, \epsilon \in \{0,1\} \atop i+j+\epsilon=k }
		(\hook{a}{i}, ...) \otimes (\hook{b-1}{j} , ..., 1, ...) \label{eq:restrict1}
	\end{align}
	where the two sums correspond to the two possible splittings of the paritition `1` into itself and the empty partition, and where each sum describes all the splittings of the hook $\hook{n-1}{k}$ (see Appendix~\ref{appsec:LRhooks}).
	Then, we apply the splitting rule for the restriction $G(r,r,a)\times S_{b}\hookrightarrow G(r,r,a)\times G(r,r,b)$
	(see \eqref{appeq:parabolicLR3} and Appendix~\ref{appsec:LRGrrn}) that consists in merging
	all partitions in each $r$-tuple appearing to the right of a tensor product into a single partition of $b$.
Only the second sum in~\eqref{eq:restrict1} is affected nontrivially by this phenomenon.
Namely, there are at most three ways to add a box to the hook $\hook{n-1}{k}$, resulting in the  three possible partitions
$\hook{n}{k+1}, \hook{n}{k}$ and $\qhook{n}{k}$. We obtain
	\begin{align}
		\parD\Big\downarrow^{G(r,r,n)}_{G(r,r,a)\times S_{b}}
		& = \sum_{i,j\geq 0, \epsilon \in \{0,1\} \atop i+j+\epsilon=k }(\hook{a-1}{i}, ..., 1, ...) \otimes \hook{b}{j}
	 + \sum_{i,j\geq 0, \epsilon \in \{0,1\} \atop i+j+\epsilon=k }(\hook{a}{i}, ...) \otimes 
		( \hook{b}{j+1} +\hook{b}{j}+\qhook{b}{j})
\nonumber \\
=& \sum_{i,j\geq 0, \epsilon \in \{0,1\} \atop i+j+\epsilon=k }(\hook{a-1}{i}, ..., 1, ...) \otimes \hook{b}{j} \nonumber \\
		&+ \sum_{i,j\geq 0, \epsilon \in \{-1,0,1\} \atop i+j+\epsilon=k } (1+\delta_\epsilon) \cdot (\hook{a}{i}, ...) \otimes \hook{b}{j} 
		+ \sum_{i,j\geq 0, \epsilon \in \{0,1\} \atop i+j+\epsilon=k }  (\hook{a}{i}, ...) \otimes \qhook{b}{j}
		\label{eq:h2}.
	\end{align}
where the last equality is obtained by sum and indices manipulations.

	We observe that there are common terms, independent of $q$, in \eqref{eq:h1} and \eqref{eq:h2}. Moreover, because $\sum_{0<q<r} \xi^q =-1$, we get by linear combination that
	\begin{align}
		\mathfrak{Q}_{n,k} \Big\downarrow^{G(r,r,n)}_{G(r,r,a) \times \mathfrak{S}_{b}} 
		&=-\hook{a}{0}\otimes(\hook{b}{k}+\hook{b}{k+1}) 
	 +\sum_{i,j\geq 0, \epsilon \in \{0,1\} \atop i+j+\epsilon=k } 
		\Big((\qhook{a}{i}, ...) +\sum_{0<q<r}\xi^{-q}(\hook{a-1}{i},...,1,...) 
		\Big)\otimes \hook{b}{j} \nonumber\\
		&=
		 -\hook{a}{0}\otimes(\hook{b}{k}+\hook{b}{k+1}) + \sum_{i,j\geq 0, \epsilon \in \{0,1\} \atop i+j+\epsilon=k } 
		\mathfrak{Q}_{a,i}\otimes \hook{b}{j}.
		\end{align}
	Using the induction hypothesis $\mathfrak{Q}_{a,i}\equiv -{\bigwedge}^{i+1} (\mathbb{C}^n)$, the fact that $\hook{b}{j}+\hook{b}{j-1}={\bigwedge}^{j}(\mathbb{C}^{b})$ as $\mathfrak{S}_{b}$-representations for any $j\geq 0$, and the fact that $\hook{a}{0}$ is the trivial representation also equal to ${\bigwedge}^0 (\mathbb{C}^a)$, the last equation can be rewritten
	\begin{align}
		\mathfrak{Q}_{n,k} \Big\downarrow^{G(r,r,n)}_{G(r,r,a) \times \mathfrak{S}_{b}} 
		&\equiv
		 -\sum_{i,j\geq 0,  \atop i+j=k+1 } {\bigwedge}^i(\mathbb{C}^a) \otimes {\bigwedge}^j(\mathbb{C}^{b})\\
		 &=-{\bigwedge}^{k+1}(\mathbb{C}^n).
\end{align}
We conclude that $\mathfrak{Q}_{n,k} \equiv -{\bigwedge}^{k+1} (\mathbb{C}^n)$ as virtual characters of the subgroup $W_{sub}=G(r,r,a) \times \mathfrak{S}_{b}$.

It remains to examine case 2a, that is $W_{sub}=\mathfrak{S}_n$. The restriction of $\parC$ is $\qhook{n}{k}$, and the restriction of $(\hook{n-1}{k}, ... 1, ...)$ is the sum of the three terms $\hook{n}{k+1}, \hook{n}{k}$ and $\qhook{n}{k}$, corresponding to the three ways to add a box to the hook $\hook{n-1}{k}$. Using  $\sum_{0<q<r}\xi^q =-1$ we get as $\mathfrak{S}_n$-representations
$$
\mathfrak{Q}_{n,k}\Big\downarrow^{G(r,r,n)}_{\mathfrak{S}_n} = \hook{n}{k+1} + \hook{n}{k} = {\bigwedge}^{k+1}(\mathbb{C}^n).
$$

Thus the hypotheses of Lemma~\ref{lemma:subgroup} hold for any choice of the subgroup $W_{sub}$, and we conclude that $\mathfrak{Q}_{n,k}	\equiv {\bigwedge}^{k+1} (\mathbb{C}^n)$ for the group  $G(r,r,n)$, which concludes the induction step. 

\end{proof}

\subsection{The dihedral group $G(r,r,2)$.}
\label{sec:dihedral}

In this section we consider the group $G(r,r,2)$. We prove an explicit formula for the generating function of factorizations with an arbitrary weight system (see~\eqref{eq:explicitDihedral} below), and then specialize it to prove Theorem~\ref{thm:main}.

The  group $W=G(r,r,2)\subset GL_2(\mathbb{C})$, which is isomorphic to the dihedral group $I_2(r)$,  consists of the $2r$ elements 
$$
t_i=\left(\begin{array}{cc} 0 & \zeta^i \\ \zeta^{-i} & 0 \end{array}\right) \ , \  
\rho_i=\left(\begin{array}{cc} \zeta^i &0  \\0& \zeta^{-i} \end{array}\right), \ 0\leq i \leq r-1,
	$$
where $\zeta$ is a primitive $r$-th root of unity. It contains $r$ reflections $t_0, \dots, t_{r-1}$, its Coxeter number is $h=r$,  and the two Coxeter elements are $\rho_1$ and $\rho_{-1}$.
Elements of $W$ obey the relations $t_i t_j = \rho_{i-j}$ and $\rho_i t_j = t_{i+j}$ (here and below indices are considered modulo $r$), consequently
\begin{align}
	\label{eq:dihedralRelation}
	t_{i_1} t_{i_2}\dots t_{i_{2\ell}} = \rho_{i_1-i_2+i_3\dots -  i_{2\ell}}.
\end{align}
Now, consider the most general weight system assigning the weight $\omega_i$ to the reflection $t_i$, for each $0 \leq i< r$, and define the following formal power series in $\mathbb{Q}[u,u^{-1},\omega_0,\dots,\omega_{r-1}][[t]]$:  
\begin{align*}
	H(t,u) &:= \sum_{\ell \geq 0}
	\frac{t^{2\ell}}{(2\ell)!}
	 \big((u^0 \omega_0 + u^1 \omega_1 + \dots+ u^{r-1} \omega_{r-1}) (u^{-0} \omega_0 + u^{-1} \omega_1 + \dots+ u^{-(r-1)} \omega_{r-1})\big)^\ell\\
	 &= \mathrm{cosh} \left(t\sqrt{(u^0 \omega_0 + u^1 \omega_1 + \dots+ u^{r-1} \omega_{r-1}) (u^{-0} \omega_0 + u^{-1} \omega_1 + \dots+ u^{-(r-1)} \omega_{r-1})
		 }\right).
\end{align*}
	Then by~\eqref{eq:dihedralRelation} the weighted exponential generating function of factorizations of Coxeter elements in $G(r,r,2)$ is equal to 
	$$
	\sum_{k \in \mathbb{Z} \atop k \equiv \pm 1 \mod r} [u^k] H(t,u),
	$$
where $[u^k]$ denotes coefficient extraction, and where we observe that the sum over $k$ is finite for the coefficient of each fixed power of $t$. This quantity can be expressed  as a sum over $r$-th roots of unity, and we obtain the explicit expression:
\begin{align}\label{eq:explicitDihedral}
	\fac_{G(r,r,2),\mathcal{C}}^{\mathcal{R},P}(t,\bm\omega) &= \frac{1}{r}
	\sum_{i=0}^{r-1} \left( \zeta^i +\zeta^{-i} \right) H(t,\zeta^i), 
\end{align}
where we use the notation of Section~\ref{Subsection: A frobenius Lemma for weighted enumeration}, and $P$ denotes the  partition of the set $\mathcal{R}$ of reflections into singletons, with associated weight function $\mathbf{w}_P(t_i)=\omega_i$.

Let us now prove Theorem~\ref{thm:main} for this group. Every parabolic tower is conjugated to the parabolic tower $\{1\} \subset W_1\subset  W $ with $W_1=\{1,t_0\}$. This corresponds to the weight system giving weight $\omega_0$ to $t_0$ and $\omega_1$ to all other reflections. The $W$-Laplacian in this case is thus
$$
L= \omega_0 (1-t_0) + \omega_1 \sum_{i=1}^{r-1} (1-t_i)= \left(\begin{array}{cc}  \omega_0+(r-1)\omega_1&\omega_1-\omega_0 \\ \omega_1-\omega_0 & \omega_0+(r-1)\omega_1\end{array}\right),
$$
with eigenvalues $r\omega_1$ and $2\omega_0+(r-2)\omega_1$. Now applying~\eqref{eq:explicitDihedral} with  $\omega_i=\omega_1$, for $i\neq 0$, directly gives
\begin{align*}
\fac_W^T(t,\bm\omega)& =
	\frac{1}{r}
	\left(2\mathrm{cosh}(t(\omega_0+(r-1)\omega_1)) 
	+\sum_{i=1}^{r-1} \left( \zeta^i +\zeta^{-i} \right) \mathrm{cosh}(t( \omega_0 -  \omega_1))\right)\\
	&=\frac{1}{r}\left(e^{t(\omega_0+(r-1)\omega_1)} +e^{-t(\omega_0+(r-1)\omega_1)} - e^{t(\omega_0-\omega_1)} - e^{-t(\omega_0-\omega_1)} \right),\\
		&= \frac{1}{r} e^{t(\omega_0+(r-1)\omega_1)}\left(1 - e^{-tr\omega_1}\right)\left(1 - e^{-t(2\omega_0+(r-2)\omega_1)}\right),
\end{align*}
which is precisely Theorem~\ref{thm:main} for this group. There are in fact finer weight systems for which the generating function~\eqref{eq:explicitDihedral} factors, see~\S~\ref{Sec: end: dihedrals}.

\section{The $\mathcal{A}$-Laplacian and applications in Coxeter combinatorics}
\label{Section: Matrix-Forest}

As we mention in the Introduction (Corol.~\ref{cor:reduced}) a consequence of our Thm.~\ref{thm:main} is a version of the Matrix tree theorem for reflection groups. It is well known that the whole characteristic polynomial $\chi\big(L_G(\bm\omega)\big)$ of the (weighted) Laplacian of a graph $G$ has a combinatorial interpretation  (not just its determinant). This is usually referred to as the (weighted) Matrix Forest theorem; we give it below for the complete graph $K_n$ and calculating $\chi\big(-L_{K_n}(\bm\omega)\big)$ for simplicity \cite[Lemma~4]{arxiv_MT}:

\begin{equation}\op{det}\big(x+L^{}_{K_n}(\bm\omega)\big)=\sum_{\mathcal{F}}\mathbf{w}(\mathcal{F})\cdot x^{c(\mathcal{F})},\label{Eq: Usual M-F theorem}\end{equation} where the sum is over all forests $\mathcal{F}$ of rooted trees, with vertices in $[n]$ and $c(\mathcal{F})$-many connected components (and therefore roots), and where $\bf{w}$ is the usual weight function assigning the weight $\omega_{ij}$ to the edge $(ij)$.

In this section we give an analog of the Matrix Forest theorem for reflection groups (Thm.~\ref{Thm: WMFT}), where we interpret the characteristic polynomial of the $W$-Laplacian $L^T_W(\bm\omega)$ (with Jucys-Murphy weights) via factorizations of parabolic Coxeter elements. We prove it by combining Corol.~\ref{cor:reduced} with a recursive formula for $\chi\big(L_W(\bm\omega)\big)$ which however holds for a much more general class of Laplacian matrices. As we are convinced of the many possible applications of these latter objects (see for instance Remark~\ref{Rem: Brieskorn's lemma} and Sections~\ref{Sec: identity multiset Cox nums} and~\ref{Sec: root polytopes}) we build the general framework.

\begin{definition}\label{Defn: A-Laplacian}
Let $\mathcal{A}$ be a (possibly complex) hyperplane arrangement in a space $V\cong\CC^n$ and choose for each of its $N$-many hyperplanes $H_i$ a (Hermitian-)orthogonal vector $r_i$ (of arbitrary norm). Now, let $\bm\omega=(\omega_i)_{i=1}^N$ be a weight system for $\mathcal{A}$ and define the $\mathcal{A}$-Laplacian matrix as $$ \op{GL}(V)\ni L_{\mathcal{A}}(\bm\omega):=\sum_{i=1}^N\omega_i\cdot\big( \mathbf{I}_n - S_{r_i}\big), $$ where $\mathbf{I}_n$ is the $(n\times n)$ identity matrix and $S_{r_i}(v)=v-\langle v,r_i\rangle\cdot r_i$.

\noindent Notice that $\mathcal{L}_{\mathcal{A}}(\bm\omega)$ is a sum of rank-1 operators and that it depends on the norms $\langle r_i, r_i\rangle$ of the vectors we chose. In particular, for a real arrangement $\mathcal{A}$, if $\langle r_i,r_i\rangle=2$ then $S_{r_i}$ is the reflection across the hyperplane $H_i$. As with the usual Laplacian, we can rewrite $L_{\mathcal{A}}(\bm\omega)$ as $R\Omega R^T$, where $R$ is the $n\times N$ matrix that encodes the $r_i$'s (see proof of Lemma~\ref{Lem:Burman M-T} below).  
\end{definition}

Burman et al. developed \cite{BPT} an extensive theory for such and more general matrices. In particular, they gave an abstract Matrix-forest theorem that computes the characteristic polynomial $\chi\big(L_{\mathcal{A}}(\bm\omega)\big)$ via certain Grammian determinants. Its proof is quite short and we reproduce it below for completeness.

\begin{lemma}\cite[Corol.~2.4]{BPT}\label{Lem:Burman M-T}
For a hyperplane arrangement $\mathcal{A}$ of dimension $n$ and its Laplacian matrix $L_{\mathcal{A}}(\bm\omega)$ defined above, the characteristic polynomial $\chi\big(-L_{\mathcal{A}}(\bm\omega)\big)$ is given by
$$\op{det}\big(x+L_{\mathcal{A}}(\bm\omega)\big)=\sum_{\{r_{i_1},\cdots,r_{i_k}\}}\omega_{i_1}\cdots\omega_{i_k}\cdot\op{det}\big(\langle r_{i_s},r_{i_t}\rangle \big)_{s,t=1}^k\cdot x^{n-k},$$ where the sum is over all linearly independent \emph{sets} of normal vectors $r_i$ (and $k=0\dots n$).

\end{lemma}
\begin{proof}
	Pick an orthonormal basis $(e_1,\dots,e_n)$ of $V$, and let $R$ be the $n\times N$ matrix such that $r_j=\sum_{i=1}^n R_{ij} e_i$, for $1\leq j \leq N$. We have $\langle e_i,r_j\rangle=\overline{R}_{ij}$ and $$
	L_{\mathcal{A}}(\bm\omega) (e_i) = \sum_{j=1}^N \omega_j \langle e_i, r_j\rangle r_j = \sum_{j=1}^N \sum_{u=1}^n \omega_j  \overline{R}_{ij} R_{uj} e_u.
	$$
	Therefore the matrix of $L_{\mathcal{A}}(\bm\omega)$ in the basis $(e_i)$ is given by $R \Omega \overline{R}^T$
	where $\Omega:=\mathrm{diag}(\omega_1,\dots,\omega_N)$. For a matrix $M$ and $I,J$ subsets of lines and columns, write $M_{I}^{J}$ for the submatrix formed by keeping lines in $I$ and columns in $J$.	We expand
	\begin{align*}
		\det (x+R\Omega \overline{R}^T) &= \sum_{k=0}^n x^{n-k}
		\sum_{1\leq i_1<\dots<i_k \leq n} \det (R\Omega \overline{R}^T)_{i_1,\dots,i_k}^{i_1,\dots,i_k}\\
&	=\sum_{k=0}^n x^{n-k}
		\sum_{1\leq i_1<\dots<i_k \leq n}\  
		\sum_{1\leq j_1 < \dots <j_k \leq N}
		\omega_{j_1}\dots \omega_{j_k}
\det (R)_{i_1,\dots,i_k}^{j_1,\dots,j_k}
\det (\overline{R})_{i_1,\dots,i_k}^{j_1,\dots,j_k}\\
&	=\sum_{k=0}^n x^{n-k}
		\sum_{1\leq j_1 < \dots <j_k \leq N}
		\omega_{j_1}\dots \omega_{j_k}
\det (R^T\overline{R})_{j_1,\dots,j_k}^{j_1,\dots,j_k}
.
	\end{align*}
	where the second and third equalities are applications of the Cauchy-Binet formula.
The proof is complete, since 
	$\langle r_i, r_j\rangle = \sum_{u=1}^n R_{ui} \overline{R}_{uj} = (R^T \overline{R})_{ij}$.
\end{proof}

By grouping together the terms of the right hand side in this lemma with respect to the linear space spanned by the $r_{i_j}$, we immediately get the following recursive formula. Recall that the \emph{pseudodeterminant} of an operator is the product of its nonzero eigenvalues and that $\mathcal{A}_X$ denotes the \emph{localization} of an arrangement $\mathcal{A}$ over a flat $X$; that is, $\mathcal{A}_X:=\{H\in\mathcal{A}:\ X\subset H\}$. Moreover, we keep the choices (see Defn.~\ref{Defn: A-Laplacian}) of the vectors $r_i$ we made for $\mathcal{A}$, also for the $\mathcal{A}_X$-Laplacian $L_{\mathcal{A}_X}(\bm\omega)$.

\begin{proposition}[$\mathcal{A}$-Laplacian recursion]\label{prop:arr_rec}\ \newline
If $\mathcal{A}$ is a hyperplane arrangement and $L_{\mathcal{A}}(\bm\omega)$ its Laplacian matrix as above, we have
$$\op{det}\big(x+L_{\mathcal{A}}(\bm\omega)\big)=\sum_{X\in\mathcal{L}_{\mathcal{A}}}\op{pdet}\big(L_{\mathcal{A}^{}_X}(\bm\omega)\big)\cdot x^{\op{dim}(X)},$$ where $\mathcal{L}_{\mathcal{A}}$ denotes the intersection lattice of $\mathcal{A}$ and $\op{pdet}$ stands for \emph{pseudodeterminant}.
\end{proposition}
\begin{proof}
Every subset of the $r_i$'s determines a collection of hyperplanes $H_i$ whose intersection is some flat $X$; then these $r_i$ span the orthogonal complement $X^{\perp}$. By grouping together all those collections $\{r_{i_j}\}$ with respect to their associated flat $X$, we can rewrite lemma~\ref{Lem:Burman M-T} as $$\op{det}\big(x+L_{\mathcal{A}}(\bm\omega)\big)=\sum_{X\in\mathcal{L}_{\mathcal{A}}}\ \ \sum_{\op{sp}(r_{i_j})=X^{\perp}}\omega_{i_1}\cdots \omega_{i_k}\cdot\op{det}\big(\langle r_{i_s},r_{i_t}\rangle\big)_{s,t=1}^k\cdot x^k,$$ where of course $k=\op{dim}(X)$. Notice now that the inner summation is exactly equal to $\op{pdet}\big(L_{\mathcal{A}^{}_X}(\bm\omega)\big)$. Indeed, this is again lemma~\ref{Lem:Burman M-T} but applied to the arrangement $\mathcal{A}_X$, keeping in mind that the pseudodeterminant of some operator $A$ equals the smallest degree nonzero coefficient of the characteristic polynomial $\op{det}(x+A)$.
\end{proof}

\begin{remark}\label{Rem: Brieskorn's lemma}
The previous proposition is reminiscent of Brieskorn's lemma \cite[Lem.~5.91]{OT} which relates the Poincare polynomial of the complement of $\mathcal{A}$ with those for the localizations $\mathcal{A}_X$: $$\op{Poin}(V\setminus\mathcal{A},t)=\sum_{X\in\mathcal{L}_{\mathcal{A}}}\op{rank}\big(H^{\op{top}}(V\setminus\mathcal{A}_X)\big)\cdot t^{\op{dim}(X)}.$$ After the identification with the characteristic polynomial of the arrangement this is trivial, but the original lemma in fact shows a natural decomposition of the corresponding cohomology spaces. Could proposition~\ref{prop:arr_rec} be interpreted in a similar way?
\end{remark}

\subsection{A Matrix-Forest theorem for well generated complex reflection groups}

The $\mathcal{A}$-Laplacian we introduced in this section is a natural generalization of the $W$-Laplacian of reflection groups, even when the order of the unitary reflections is greater than $2$. Recall that $\mathcal{A}_W$ denotes the reflection arrangement of $W$.

\begin{proposition}\label{Prop: W-Laplacian <= A-Laplacian}
The (weighted) $W$-Laplacian can always be realized as an $\mathcal{A}_W$-Laplacian.\end{proposition}
\begin{proof}
Comparing the two Definitions~\ref{Defn: W-Laplacian matrix} and~\ref{Defn: A-Laplacian}, notice that if $\zeta$ is the unique $\zeta\neq 1$ eigenvalue of $\rho^{}_V(\tau)$ and $r^{}_{\tau}$ a $\zeta$-eigenvector of norm $\langle r^{}_{\tau},r^{}_{\tau}\rangle=1$, we have an identity between rank-1 operators $$\mathbf{I}_n-\rho^{}_V(\tau)=(1-\zeta)\cdot\big(\mathbf{I}_n-S_{r^{}_{\tau}}\big).$$
Now, we can pick a generator $\tau_H$ of each $W_H$ as in \eqref{EQ: decomposition of R into cyclic groups} and for an arbitrary weight system $\wt$, we may write $$L_W(\bm\omega)=\sum_{H\in\mathcal{R}^*}\sum_{i=1}^{e_H-1}\wt(\tau_H^i)\big(\mathbf{I}_n-\rho_V(\tau_H^i)\big)=\sum_{H\in\mathcal{R}^*}\Big(\sum_{i=1}^{e_H-1}\wt(\tau_H^i)\cdot (1-\zeta_H^i)\Big)\cdot\big( \mathbf{I}_n-S_{r^{}_H}\big).$$ Here again we have chosen vectors $r^{}_H$ orthogonal to $H$ with $\langle r^{}_H,r^{}_H\rangle=1$ and for the second equality we are using that the matrices $\rho^{}_V(\tau_H^i)$, $i=1\dots e_H-1$, all have common eigenspaces. This is a direct way to assign weights for the $\mathcal{A}_W$-Laplacian with norm-1 vectors so that it agrees with a weighted $W$-Laplacian. 

In the case that the weight system is defined via some tower $T$ \eqref{Eq: parabolic tower}, then all reflections $\tau_H^i$, $i=1\dots e_H-1$, are assigned the same weight, so the above construction would be simplified either by setting weights $\omega_H=e_H\cdot\mathbf{w}^{}_T(\tau_H)$ or by having the exact same weights $\omega_H=\mathbf{w}^{}_T(\tau_H)$ but choosing orthogonal vectors with norms $\langle r^{}_H,r^{}_H\rangle=e_H$. 
\end{proof}

The following theorem is now a direct application of the recursion of Prop.~\ref{prop:arr_rec} and Corol.~\ref{cor:reduced}.

\begin{theorem}[$W$ Matrix Forest Theorem]\label{Thm: WMFT}
For a well generated complex reflection group $W$ and a parabolic tower $T$, the characteristic polynomial of its $T$-weighted $W$-Laplacian is given via
$$\op{det}\big(x+L_W^T(\bm\omega)\big)=\sum_{\tau_1\cdots\tau_{n-k}=c^{}_X\atop  X\in\mathcal{L}_W,\ k=\op{dim}(X)}|C^{}_{W_X}(c^{}_X)|\cdot\mathbf{w}^{}_T(\tau_1)\cdots\mathbf{w}_T^{}(\tau_{n-k})\cdot\dfrac{x^k}{(n-k)!},$$ where the sum is over all (reduced) reflection factorizations of \emph{any} Coxeter element $c^{}_X$ of \emph{any} parabolic subgroup $W_X$ of $W$ and $C_{W_X}(c_X)$ is the centralizer of $c_X$ in $W_X$.
\end{theorem}
\begin{proof}
As we showed in Prop.~\ref{Prop: W-Laplacian <= A-Laplacian} (see also the last paragraph in the proof), we may identify the $W$-Laplacian $L_W^T(\bm\omega)$ with the $\mathcal{A}_W$-Laplacian $L_{\mathcal{A}_W}(\bm\omega)$ with the same weight system if we choose the orthogonal vectors to have norms $\langle r_H,r_H\rangle=e_H$.
After Prop.~\ref{prop:arr_rec} we need only show that $$\op{pdet}\big(L_{\mathcal{A}_X}(\bm\omega)\big)=|C_{W_X}(c^{}_X)|\cdot\big(\sum_{\tau_1\cdots\tau_{n-k}=c^{}_X}\mathbf{w}^{}_T(\tau_1)\cdots \mathbf{w}^{}_T(\tau_{n-k})\big)\cdot\dfrac{1}{(n-k)!},$$ where the sum is over all Coxeter elements $c^{}_X$ of the parabolic subgroup $W_X$.

This required identity is just Corol.~\ref{cor:reduced} for the possibly reducible parabolic subgroup $W_X$. Notice that by standard properties of exponential generating functions, if $W_X=W_1\times\cdots \times W_s$ is the decomposition into irreducibles, we will have (with a slight abuse of notation) that $$\mathcal{F}_{W_X}^T(t,\bm\omega)=\prod_{i=1}^s\mathcal{F}_{W_i}^T(t,\bm\omega).$$ Moreover, the construction of Prop.~\ref{Prop: W-Laplacian <= A-Laplacian} is compatible with localizations; that is, $L_{\mathcal{A}_X}(\bm\omega)=L^T_{W_X}(\bm\omega)$ where we just take the restriction of both weight systems on $\mathcal{A}_X$ and $\mathcal{R}\cap W_X$ respectively (the latter is still a tower system, induced by $T\cap W_X$). Now, we only need to check that the factor $|C_{W_X}(c_X)|$ is equal to the product of the Coxeter numbers associated to the irreducible components of $W_X$ (as that would instead be the factor prescribed by Corol.~\ref{cor:reduced}). Indeed, this is an immediate consequence of the fact that the centralizers of Coxeter elements $c$ in irreducible groups $W$ are the cyclic groups $\langle c\rangle$, see \eqref{EQ: |C|=|W|/h}.\end{proof}

\begin{remark}\label{Rem: W-M-F thm}
The previous theorem might appear intuitive, but it is highly non-trivial. There is no reason that the combinatorics of \emph{reduced} reflection factorizations of parabolic Coxeter elements should be encoded by the same object (the eigenvalues of the $W$-Laplacian) that also encodes the combinatorics of arbitrary length factorizations of Coxeter elements. In particular, the partial products in such \emph{higher genus} factorizations need not be parabolic Coxeter elements.
\end{remark}

\subsection{An identity among parabolic Coxeter numbers}
\label{Sec: identity multiset Cox nums}

The recursion of Prop.~\ref{prop:arr_rec} for the \emph{unweighted} $W$-Laplacian $L_W$ gives a remarkable relation between the Coxeter numbers of $W$ and its parabolic subgroups $W_X$. Of course, there is not a \emph{single} Coxeter number associated to $W_X$; even though $W_X$ is indeed a reflection group, it might be reducible in which case the irreducible components might have different Coxeter numbers.

In fact, this relation is more easily phrased in terms of the {\color{blue} \emph{multiset $\{h_i(W)\}$ of Coxeter numbers}} of a reflection group $W$, which we construct as the collection of each Coxeter number $h_i$ of an irreducible component $W_i$ of $W$ taken with multiplicity equal to $\op{rank}(W_i)$. Equivalently (after Prop.~\ref{Prop: L_W has all eigenvalues =h}) the multiset $\{h_i(W)\}$ consists of the eigenvalues, counted with multiplicity, of the $W$-Laplacian $L_W$.

Now, as we discussed in the proof of Thm.~\ref{Thm: WMFT}, if we choose for the reflection arrangement $\mathcal{A}_W$ orthogonal vectors with norms $\langle r_H,r_H\rangle=e_H$, $H\in\mathcal{A}_W$, then the (unweighted) Laplacians $L_W$ and $L_{\mathcal{A}_W}$ are equal. The induced localized Laplacians $L_{(\mathcal{A}_W)_X}$ are then also equal to the $W_X$-Laplacians $L_{W_X}$ and Prop.~\ref{prop:arr_rec} immediately gives the following.

\begin{theorem}\label{Prop: identity among coxeter numbers}
For any irreducible complex reflection group $W$ with Coxeter number $h$, we have
$$(h+x)^n=\sum_{X\in\mathcal{L}_W}\prod_{i=1}^{\op{codim}(X)}h_i(W_X)\cdot x^{\op{dim}(X)},$$ where $\{ h_i(W_X)\}$ denotes the \emph{multiset} of Coxeter numbers of $W_X$ as described previously.
\end{theorem}

\begin{remark}
In a separate article \cite{chapuy_theo_lapl} by the same authors, the previous Theorem~\ref{Prop: identity among coxeter numbers} is the main ingredient for a \emph{uniform} (type independent), \emph{combinatorial} proof of the unweighted version of Corol.~\ref{cor:reduced}. This is of course the Arnol'd-Bessis-Chapoton formula \eqref{eq:introDeligne} of the introduction, or equivalently the enumeration of maximal chains in the non-crossing lattice $NC(W)$. 
\end{remark}

\begin{remark}
If we chose orthogonal vectors with norms $\langle r_H,r_H\rangle=1$ or $\langle r_H,r_H\rangle=e_H-1$, we would get analogs of the previous proposition involving instead of Coxeter numbers $h$, the quantities $|\mathcal{R}|/n$ and $|\mathcal{R^*}|/n$ respectively. The first one is precisely the number $t/2$ that appeared in \cite[\S~6.5]{OT} and was a main ingredient in calculating the size of the restricted arrangements $\mathcal{A}_W^H$ (for real $W$ we have $h=t$ and this was further used to show that $\mathcal{A}^H$ is free).
\end{remark}

\subsection{Calculating the volumes of root polytopes}
\label{Sec: root polytopes}

The determinant of the $\mathcal{A}$-Laplacian can be seen after the \emph{unweighted version} of Lemma~\ref{Lem:Burman M-T} as a \emph{scaled} sum over $\mathbb{C}$-bases $\bm{r}$ formed by the orthogonal vectors $r_i$ of Defn.~\ref{Defn: A-Laplacian}. Indeed, each such basis is scaled by the Grammian $\op{det}(\langle r_i,r_j\rangle)$ or equivalently by the squared volume $|\op{det}(\bm{r})|^2$. This setting is almost, but not quite, what one would need to compute the volume of the zonotope associated to $\mathcal{A}$. We show here how, particularly for Weyl groups $W$, we can indeed calculate the volume of the associated root polytope. 

Recall that, given a collection of \emph{real} vectors $X=\{r_i\}$, we define the {\color{blue} zonotope} $B(X)$ as the Minkowski sum of the segments $[0,r_i]$ (see for instance the classical references \cite[\S~7.3]{ziegler}, \cite[Ch.~9]{Beck_Sinai}, or \cite[\S~2.3]{dCP}). It is a fundamental theorem of Shephard \cite[Thm.~54]{shephard} that each zonotope has a paving by parallelepipeda indexed by all possible bases $\bm{r}$ formed by elements of $X$. In particular, this implies the formula \begin{equation}
\op{Vol}\big(B(X)\big)=\sum_{\bm r}|\op{det}(\bm r)|,\label{EQ: vol of zonotopes}
\end{equation} by Shephard and McMullen \cite[(57)]{shephard}. Now, if the zonotope $B(X)$ is unimodular (meaning that all bases $\bm r$ satisfy $|\op{det}(\bm r)|=1$), then we can immediately calculate its volume via \begin{equation}\op{Vol}\big(B(X)\big)=\op{det}\big(L_{\mathcal{A}(X)}\big)=\sum_{\bm r}|\op{det}(\bm r)|^2,\label{EQ: vol via A-Lapl}\end{equation} where $\mathcal{A}(X)$ is the arrangement of hyperplanes orthogonal to the vectors in $X$. Indeed, this is a well-known observation \cite[Prop.~2.52]{dCP} where we usually see the matrix $L_{\mathcal{A}(X)}$ defined as $XX^t$ (recall the comments after Defn.~\ref{Defn: A-Laplacian}).

For crystallographic reflection groups $W$, one defines the {\color{blue} root zonotope $Z_W$} as the zonotope $B(X)$ associated to the collection $X$ of positive roots of $W$.  It is a translation \cite[Thm.~2.56]{dCP} of the convex hull $P_W(\rho)$ of the $W$-orbit of the half-sum $\rho$ of positive roots (such polytopes $P_W(x)$ are called Coxeter (or $W$-) permutahedra \cite[Defn.~2.1]{hohlweg} or weight polytopes \cite[\S~4]{postnikov}). 

	For the symmetric group $S_n$, we recover the classical permutahedron \cite[Example~0.10]{ziegler} for which we know \cite[Ex.~4.64]{EC1} the (normalized) volume to be $n^{n-2}$. Indeed, for $S_n$ the root zonotope is essentially unimodular (we have $|\op{det}(\bm r)|=\sqrt{n}$ for all bases $\bm r$ ) so that its volume can easily be computed after \eqref{EQ: vol via A-Lapl} and via the $S_n$-Laplacian, whose determinant is $n^{n-1}$ (there is some analogy with the situation in~\cite{Kalai} here).

For the other crystallographic types however, the root systems are not unimodular and the volumes of the corresponding zonotopes do not factor nicely. Still, there is some level of control; the contribution as in \eqref{EQ: vol of zonotopes} of each basis $\bm r$ from $X$ is completely determined by the reflection subgroup $W'\leq W$ generated by the reflections $S_{r_i}$, $r_i\in \bm r$. Indeed, Baumeister and Wegener have shown \cite[Cor.~1.2]{BW} that if a collection of reflections $\{S_{r_i}\}$, $i=1\dots n$, generates a Weyl group $W'$ of rank $n$, then the roots $r_i$ and coroots $\widehat{r_i}$ form, respectively, $\mathbb{Z}$-bases of the root and coroot lattices $Q'$ and $\widehat{Q'}$ of $W'$. Now, all $\mathbb{Z}$-bases $\bm r$ of a lattice determine parallelepipeda of the same volume $|\op{det}(\bm r)|$, which is for this reason called the \emph{volume} of the lattice. In other words, and for zonotopes $Z_W$, the Shephard-McMullen formula \eqref{EQ: vol of zonotopes} can be written as
\begin{equation}\op{Vol}(Z_W)=\sum_{W'\leq_{\op{max}}W}g(W')\cdot \op{vol}(Q'),\label{EQ: vol(Z_W)}\end{equation} where the summation is over all \emph{maximal rank} reflection subgroups $W'$ of $W$, $Q'$ denotes the root lattice of the (Weyl) group $W'$, and where $g(W')$ is the number of $n$-sets $\{S_{r_i}\}$ that generate $W'$.

In the following, we take advantage of this property and give a formula for the volume of root zonotopes $Z_W$ in terms of the Coxeter numbers and root theoretic data of its maximal rank reflection subgroups. Even though these volumes have been calculated before (see \cite[Thm.~4.2]{postnikov}, \cite[Thm.~3.4]{dCP_zon}, and \cite[Prop.~4.2]{eur}) our approach yields an arguably more explicit answer. We note however that the previously mentioned works do much more and calculate the volumes of arbitrary $W$-permutahedra $P_W(x)$.

\begin{theorem}\label{Thm: volume of Z_W}
The volume of the associated root zonotope $Z_W$ of a Weyl group $W$ is given by
$$\op{Vol}(Z_W)=\sum_{W''\leq W}\ \Big(\prod_{i=1}^nh_i(W'')\Big) \sum_{W''\leq W'\leq W}\mu(W',W'')\cdot\dfrac{1}{\op{vol}(\widehat{Q'})},$$ where the sum is over all \emph{maximal rank} reflection subgroups $W''$ of $W$, ordered by inverse inclusion, and where $\widehat{Q'}$ is the coroot lattice of $W'$ and $\{h_i(W'')\}$ the multiset of Coxeter numbers of $W''$.
\end{theorem}
\begin{proof}
The quantity $g(W')$ of \eqref{EQ: vol(Z_W)}, which counts sets of $n$ reflections that generate $W'$, can be re-expressed in terms of Coxeter numbers. Indeed, after the proof of Prop.~\ref{Prop: W-Laplacian <= A-Laplacian}, if we normalize the root vectors so that $\langle \tilde{r}_i,\tilde{r}_i\rangle=2$, the corresponding $\mathcal{A}_W$-Laplacian equals the $W$-Laplacian. Then,
$$ \prod_{i=1}^nh_i(W)=\op{det}(L_{W})=\op{det}(L_{\mathcal{A}_W})=\sum_{\bm r}\op{det}\big(\langle\tilde{r}_i,\tilde{r}_j\rangle\big)_{i,j=1}^n=\sum_{W'\leq_{\op{max}}W}g(W')\cdot I(W'),$$
where the first equality is Prop.~\ref{Prop: L_W has all eigenvalues =h} and the third equality comes after Lemma~\ref{Lem:Burman M-T}. For the last equality, we apply first the Baumeister-Wegener theorem discussed previously and notice that the matrix $\big(\langle\tilde{r}_i,\tilde{r}_j\rangle\big)_{i,j=1}^n$ and the Cartan matrix $\big(\langle r_i,\widehat{r_j}\rangle\big)_{i,j=1}^n$ of $W'$ have the same determinant, which is denoted $I(W')$ and is known as the \emph{connection index} of $W'$. A Mobius inversion on the poset of maximal rank reflection subgroups of $W'$ gives then that $$g(W')=\dfrac{1}{I(W')}\cdot\sum_{W''\leq_{\op{max}}W'}\mu(W',W'')\cdot\prod_{i=1}^nh_i(W''),$$ which plugged in \eqref{EQ: vol(Z_W)} immediately gives the theorem after writing $I(W')=\op{vol}(Q')\cdot \op{vol}(\widehat{Q'})$.
\end{proof}

\begin{remark}
Following Stanley's \cite[Thm.~2.2]{stanley} generalization of the Shephard-McMullen formula \eqref{EQ: vol of zonotopes}, one can easily extend Thm.~\ref{Thm: volume of Z_W} above and calculate the whole Ehrhart polynomial of $Z_W$ (but we will not elaborate on this here). In the classical types, where $\CC$-bases of root systems correspond to signed trees, this has been done for instance in \cite{ardila2020}.
\end{remark}

\begin{example}[The volume of the root polytope for $E_6$]\ \newline
The case of $E_6$ is a particularly nice example where our Theorem~\ref{Thm: volume of Z_W} gives a much faster calculation of the root polytope volume than existing techniques. In \cite{DPR}, the authors have compiled substantial information for the reflection subgroups of finite Coxeter groups. In particular for $E_6$, this is sufficient to determine its poset of maximal rank reflection subgroups; it contains a minimal element ($E_6$ itself) and only atoms: $36$ groups of type $A_1\times A_5$ and $40$ groups of type $A_2\times A_2\times A_2$.

The multisets of Coxeter numbers of the groups $E_6$, $A_1A_5$, and $A_2^3$, are respectively $\{12^6\}$, $\{2,6^5\}$, and $\{3^6\}$ (where exponents denote multiplicities). Their connection indices are $3$, $12=2\cdot 6$, and $27=3^3$, respectively, and they all satisfy (being simply laced) $\op{vol}(Q')=\op{vol}(\widehat{Q'})=\sqrt{I(W')}$. After the previous theorem, the volume of the root zonotope $Z_{E_6}$ will then be given as \begin{alignat*}{2}\op{Vol}(Z_{E_6})=&12^6\cdot\dfrac{1}{\sqrt{3}}+36\cdot (2\cdot 6^5)\cdot\big(\dfrac{1}{\sqrt{12}}-\dfrac{1}{\sqrt{3}}\big) + 40\cdot 3^6\cdot \big(\dfrac{1}{\sqrt{27}}-\dfrac{1}{\sqrt{3}}\big)&\\
=&\sqrt{3}\cdot 895536,&
\end{alignat*}which precisely agrees with the data in \cite[\S~3,~p.~512]{dCP_zon} (DeConcini and Procesi normalize their numbers so that the volume of the root lattice is $1$, which in the simply laced types has the effect of dividing the final answer by $\sqrt{I(W)}$, which in the case of $E_6$ is $\sqrt{3}$).
\end{example}

\subsection{Trees versus Coxeter factorizations in reflection groups}
\label{sec: Trees vs Facns}

In the traditional setting, the (pseudo-)determinant of the unweighted graph Laplacian $L(K_n)$ equals $n^{n-2}$ and counts labeled trees on $n$ vertices as well as minimal length factorizations of a given long cycle $c\in S_n$ in transpositions. We have seen in this section that the two combinatorial objects have generalizations for reflection groups which are both counted somehow by the $W$-Laplacian, even allowing certain weights. 

On one side, the factorizations of long cycles generalize to reduced reflection factorizations of Coxeter elements and Corol.~\ref{cor:reduced}, allowing Jucys-Murphy weights, gives us 
\begin{equation}
\sum_{\tau_1\cdot \tau_2\cdots\tau_n =c \atop \tau_i\in\mathcal{R},\ c\in \mathcal{C}} \mathbf{w}^{}_T(\tau_1) \mathbf{w}^{}_T(\tau_2)\cdots\mathbf{w}^{}_T(\tau_n)\cdot h\cdot\dfrac{1}{n!}=
	\det L_W^T(\bm\omega), \label{EQ: sec.8 L_W<->facns}
\end{equation} where $n$ is the rank of $W$ and $h$ its Coxeter number. On the other hand, Burman's Lemma~\ref{Lem:Burman M-T} expresses the $W$-Laplacian (and after Prop.~\ref{Prop: W-Laplacian <= A-Laplacian}) as a weighted sum over all $\CC$-bases $\bm r$ formed by vectors $r_i$ orthogonal to the reflection hyperplanes $H\in \mathcal{R}^*$ of $W$. These $\CC$-bases $\bm r$ can be considered natural versions of {\color{blue} $W$-trees} and the following a Laplacian enumeration for them:
$$\sum_{\bm r:=\{r_{i^{}_1},\dots,r_{i^{}_n}\}}\omega_{i^{}_1}\cdots\omega_{i^{}_n}\cdot\op{det}\big(\langle r_{i^{}_j}, r_{i^{}_k}\rangle\big)=\op{det}\big(L_W(\bm\omega)\big),$$
where moreover, there is no restriction on the weights $\omega_i$. The Grammian determinant $\op{det}\big(\langle r_{i^{}_j}, r_{i^{}_k}\rangle\big)$ might look unappealing here, but at least in the case of Weyl groups we already saw that it can be further interpreted. There, the previous equation becomes
\begin{equation}
\sum_{\bm r:=\{r_{i^{}_1},\dots,r_{i^{}_n}\}}\omega_{i^{}_1}\cdots\omega_{i^{}_n}\cdot I\big(W(\bm r)\big)=\op{det}\big(L_W(\bm\omega)\big),\label{EQ: sec.8 L_W<->trees}
\end{equation}
where $W(\bm r)$ is the (maximal rank, reflection) subgroup of $W$ generated by those reflections with associated roots $r_i$ and $I(W')$ denotes the connection index of a Weyl group $W'$.

These two equations \eqref{EQ: sec.8 L_W<->facns} and \eqref{EQ: sec.8 L_W<->trees} form thus direct generalizations of the two interpretations of the number $n^{n-2}$ discussed above; each with its own advantages and disadvantages. The former gives a direct enumerative analog but doesn't allow arbitrary weights, while the latter doesn't exactly count $W$-trees but a certain nice statistic on $W$-trees.

In the case of the symmetric group $S_n$ the left hand sides of \eqref{EQ: sec.8 L_W<->facns} and \eqref{EQ: sec.8 L_W<->trees} are directly related (and for arbitrary weights) after an argument of Denes \cite{Denes}. A key property here is that, in $S_n$, starting with a $\CC$-basis $\bm r$, the product of the corresponding reflections $\tau_{r_i}$ \emph{in any order} will always be some long cycle. Moreover, in $S_n$, \emph{there are no proper rank-$n$ reflection subgroups and the connection index $I(S_n)$ equals the Coxeter number of $S_n$}. 

For general reflection groups this is no longer true; the Coxeter number does not equal the connection index, there are many proper maximal rank reflection subgroups, and a $W$-tree does not always give you a Coxeter factorization. Still however, \emph{factorizations and $W$-trees}, both are counted via the $W$-Laplacian! Maybe we can ask then, for parabolic weight systems $\mathbf{w}^{}_T$ (and this would give a uniform proof of Corol.~\ref{cor:reduced} at least for Weyl groups $W$):

\begin{question}\label{Q: Cox-facs vs W-trees}
Is there a direct way to assign (perhaps multiple) factorizations of Coxeter elements to $\CC$-independent tuples of roots so that the left hand sides of \eqref{EQ: sec.8 L_W<->facns} and \eqref{EQ: sec.8 L_W<->trees} agree?
\end{question}

\section{Some further discussion}
\label{sec:further}

We would like to record here some further thoughts about this work, open problems and relations to other areas, as well as possible generalizations. The first, most obvious comment addresses the reliance of our main theorems on the classification. Even though we feel that the tower equivalence of Thm.~\ref{thm:equivalence} and the results of Section~\ref{sec:LieLike} give a proper explanation of the product formula of Thm.~\ref{thm:main}, it is still unclear \emph{why} the equivalence of Thm.~\ref{thm:equivalence} holds between the virtual characters.

\begin{open}\footnote{While this paper was going through the review process, Jean Michel has informed us that he was able to extend his work \cite{Michel} and give a proof of our Theorem~\ref{thm:equivalence} for \emph{Weyl groups} that does not rely on the classification \cite{JM_uniform}.}
Give a case-free proof of Thm.~\ref{thm:main} or, equivalently, Thm.~\ref{thm:equivalence}.
\end{open} 

Recently, case-free arguments have been provided for the Chapuy-Stump formula \cite{CS}, which is the unweighted version of Thm.~\ref{thm:main}, for Weyl groups by Jean Michel \cite{Michel}, and for the general case by the second author \cite{Douvr}. In both papers, the exterior powers of the reflection representation play an important role, similar to our Thm.~\ref{thm:equivalence} and after Corol.~\ref{cor:filterCoxeterNumbers}.

\subsection{Relations with the Okounkov-Vershik approach to the symmetric group}
\label{Section End: Okounkov-Vershik approach}

Our initial discovery of Thm.~\ref{thm:main} was motivated to a large extent by the work of Okounkov and Vershik \cite{OV}. They made a complete analysis of the representation theory of the symmetric group $S_n$ by studying the commutative algebra $\CC[\bm J]$ generated by the usual Jucys-Murphy elements $$J_i:=(1i)+(2i)+\cdots + (i-1,i).$$ It was because of their work that we considered weights defined by parabolic towers $T$ and we paid particular attention to the eigenspace decomposition \eqref{Eq: Gelfand-Tsetlin decomposition} induced by $T$. Various authors have tried to generalize the Okounkov-Vershik approach to other reflection groups, notably \cite{marin} and \cite{Stem} with different levels of success.

As it happens, not all properties of the usual Jucys-Murphy elements survive in our general setting. An obvious one is that the restriction of characters down the tower $T$ is not always multiplicity free \eqref{Eq: mult(towerchi)}, which in particular means that the Gelfand-Tsetlin algebra $GZ(T)$ associated to $T$ is not always a maximal commutative subalgebra of $W$. Stembridge in fact notes \cite[Rem.~2.3]{Stem} that in $E_8$ there is no tower of \emph{reflection} subgroups for which the restriction is multiplicity free.  Furthermore however, our algebra $\CC[\bm J_T]$ itself, generated by the (generalized) Jucys-Murphy elements, does not always agree with $GZ(T)$. That is, both inclusions below may be proper ones
\begin{equation}\CC[\bm J_T]\subsetneq \underbrace{\langle Z\big(\CC[W_1]\big),\cdots,Z\big(\CC[W_n]\big)\rangle}_{GZ(T)}\subsetneq \bigoplus_{\chi\in\widehat{W}} \CC^{\op{dim}(\chi)}.\label{EQ: Ivan Marin}\end{equation}

In fact, our Theorem~\ref{thm:equivalence} shows that, outside of $S_n$, at least one of these inclusions \emph{must} be proper. Indeed, if $\CC[\bm J_T]$ was a maximal commutative subalgebra, it would include the center $Z(\CC[W])$ and no two (different) virtual characters could be equal on it (as in our Thm.~\ref{thm:equivalence}). Beyond our study of their spectrum in Prop.~\ref{Prop: bounds on spectrum of JM elements}, there are therefore various questions one might ask regarding these generalized Jucys-Murphy algebras. For instance:
\begin{enumerate}
\item What is, for a given tower $T$, the dimension of the Jucys-Murphy algebra $\CC[\bm J_T]$?
\item What is the intersection of $\CC[\bm J_T]$ with the algebra $GZ(T)$? or only the center $Z(\CC[W])$? 
\item What is the dimension of the space of class functions modulo tower equivalence?
\end{enumerate}

\subsection{Our Thm.~\ref{thm:main} and the combinatorics of the $W$-Laplacian in \S~\ref{sec:combinatorial} }
\label{Sec: end: W-Lapl-combinatorics}

The combinatorial interpretation (in \S~\ref{sec:combinatorial}) of the eigenvalues of the $W$-Laplacian $L_W^T(\bm\omega)$ allows for various nice applications of our Thm.~\ref{thm:main}. For instance, one could identify many of the weights $\omega_i$, in effect considering non-maximal towers $T$ \eqref{Eq: parabolic tower}. We could even pick a parabolic subgroup $W_X$ and assign only two weights, $\omega_1$ and $\omega_2$, recording respectively if a reflection $\tau\in\mathcal{R}$ belongs to $W_X$ or not. Then, if $F_{W,X}$ counts the number of \emph{reduced} reflection factorizations of Coxeter elements using only reflections $\tau\in\mathcal{R}\cap W\setminus W_X$, Corol.~\ref{cor:reduced} and Prop.~\ref{Prop: comb. descr. of Spec of L_W^T(w)} give us (setting $\omega_1=0$ and $\omega_2=1$) \begin{equation}
F_{W,X}=\dfrac{1}{h}\cdot \Big( \prod_{i=1}^{\op{dim}(X)}\big(h-h_i(W_X)\big)  \Big)\cdot h^{\op{codim}(X)}\cdot n!,\end{equation} where as usual $\{ h_i(W_X)\}$ denotes the multiset of Coxeter numbers of $W_X$ (see \S~\ref{Sec: identity multiset Cox nums}). This could be seen as a (broad but partial) generalization of \emph{star factorizations} in the symmetric group; namely factorizations of elements of $S_n$ in the transpositions $(in)$, $i=1\dots n-1$. These were introduced by Pak \cite{Pak_star} and have been of interest to the community; see in particular \cite{Feray_star}. 

Furthermore, Griffeth has shown \cite[Thm.~A.1.(a)]{griffeth} that for any irreducible \emph{reflection} subgroup $W'\lneq W$, there is a strict inequality between the Coxeter numbers $h(W)>h(W')$ (in fact, with a uniform proof!). The previous discussion implies then that there exist (even reduced) factorizations of Coxeter elements $c\in W$ using none of the reflections in some (in this case \emph{parabolic}) subgroup $W'$. On the other hand, no Coxeter element can live in a proper reflection subgroup of $W$ (see the proof of \cite[Corol.~3.9]{Douvr}) so that we immediately get the following statement.

\begin{corollary}
For any well generated reflection group $W$ and parabolic subgroup $W_X\leq W$, the reflections in $W\setminus W_X$ generate all of $W$.
\end{corollary}

\noindent Notice that this is not true if for instance one replaces \emph{parabolic} with \emph{reflection} subgroup. A first counterexample is already $B_2$ with the reflection subgroup $A_1^2$.

Another immediate consequence of Corol.~\ref{cor:reduced} and Prop.~\ref{Prop: comb. descr. of Spec of L_W^T(w)} is that the weighted sum \begin{equation}\sum_{\tau_1\cdot \tau_2\cdots\tau_n =c \atop \tau_i\in\mathcal{R},\ c\in \mathcal{C}} \wt(\tau_1) \wt(\tau_2)\cdots\wt(\tau_n)\label{EQ: end sum mult n!}\end{equation} is always a multiple of $n!$ if the weight system $\wt$ comes from a parabolic tower $T$ \eqref{Eq: parabolic tower} (by this we mean that each monomial that appears has a coefficient that is a multiple of $n!$). For certain  groups $W$ this is natural and holds for arbitrary weights $\wt$ as well. This is the case when any permutation of the terms in a reduced reflection factorization of a Coxeter element produces a new factorization of, again, a Coxeter element (but possibly different). 

For real reflection groups for instance, this holds precisely for the coincidental types $A_n$, $B_n$, and $H_3$ (this is by the standard interpretation of factorizations as trees for the first two and by an easy computer check for $H_3$, $D_4$, and the rest of the exceptional groups). For general reflection groups $W$ though, this is no longer true and the previous sum is not a multiple of $n!$ for arbitrary weights $\wt$. We ask therefore (compare also with Ques.~\ref{Q: Cox-facs vs W-trees}):
\begin{question}\label{Q: multiple of n!}
Is there a conceptual explanation why the sum \eqref{EQ: end sum mult n!}, with Jucys-Murphy weights \eqref{Eq: parabolic tower}, is always a multiple of $n!$?
\end{question}

\subsection{On possible analogs of the  Aldous conjecture for reflection groups}

In the early 90's Aldous stated a conjecture in probability theory that was only resolved about 20 years later by Caputo et al. \cite{proof_of_Aldous}. It demanded that the spectral gaps of two different processes on a finite graph, the random walk and the interchange process, are equal (in fact the original conjecture was for unweighted graphs, but Caputo et al. proved the weighted case).

It turns out that Aldous' conjecture can be rephrased in terms of the \emph{group algebra-} and $W$-Laplacians of the symmetric group $S_n$ (denoted $\mathbf{B}_{S_n}(\bm\omega)$ and $L_{S_n}(\bm\omega)$ respecively, see \S~\ref{subsec:Laplacians}). It is explained in \cite[Section~3]{cesi} that the conjecture is equivalent to showing that the smallest positive eigenvalue of $\mathbf{B}_{S_n}(\bm\omega)$ (i.e. under the regular representation) is always, for any \emph{positive} weights $\bm\omega$, one of the (and by necessity the smallest) eigenvalues of $L_{S_n}(\bm\omega)$.

Now, while the interchange process can easily be generalized to any group with an associated (perhaps generating) subset, the random walk process doesn't have immediate analogs. One might hope that the representation theoretic version for the two Laplacians still holds, but this is in fact incorrect. For $W=G_2=I_2(6)$, and weights $\bm\omega:=(\omega_1,\omega_2,\omega_3)$ defined by the tower of \emph{reflection} subgroups $A_1\leq A_1\times A_1\leq G_2$, the results of \S~\ref{Section:weight function by group tower} (especially \eqref{Eq: Spec_chi(A(w)) description} after simple calculations) give the following set of eigenvalues of $\mathbf{B}_{G_2}(\bm\omega)$ for the reflection representation $V_{ref}$ and the unique other 2-dimensional representation $U$ of $G_2$: 
\begin{equation}
\op{Spec}\big(L_{G_2}(\bm\omega)\big)=\begin{Bmatrix} 2\omega_1+2\omega_2+2\omega_3 \\ 6\omega_3 \end{Bmatrix}\quad , \quad \op{Spec}_U\big(B_{G_2}(\bm\omega)\big)=\begin{Bmatrix} 2\omega_1 + 4\omega_3 \\ 2\omega_2 + 4\omega_3 \end{Bmatrix}.
\end{equation}
It is now easy to see that for suitable choices of the $\omega_i$'s (for instance $\bm\omega=(3,1,2)$) the smallest eigenvalue of $\mathbf{B}_{G_2}(\bm\omega)$ appears in $U$.

On a different direction, Cesi has recently suggested in \cite{cesi_B} that for $W=B_n$ one might replace the random walk process with the coset permutation action for the subgroup $B_{n-1}\leq B_n$ (this is indeed the \emph{canonical} action of $B_n$ on the $2n$-element set $\{-n,\cdots,-1,1,\cdots,n\}$). He proves in the same work an analog of the Aldous' conjecture for $W=B_n$, where however he forces all $\binom{n}{2}$-many reflections of the form $(i\bar{j})$ to have zero weight \cite[Thm.~1.2]{cesi_B}.

For general reflection groups the coset actions for maximal parabolics do not behave as nicely as for types $A$ and $B$. However, it turns out that the (non trivial) irreducible components that appear in $\CC[B_n/B_{n-1}]$ all have (generalized) Coxeter numbers equal to $2n$, the Coxeter number of $B_n$ and common eigenvalue of the \emph{unweighted} Laplacian $L_{B_n}$. For well generated groups $W$, an easy check shows that the Coxeter number $h$ is in fact the smallest possible (generalized) Coxeter number for non-trivial irreducible representations (this is no longer true in the non well generated groups, already in $G_7$). This leads us to ask the following question:
\begin{question}
Let $W$ be a well generated complex reflection group and $\bm\omega:=(\omega_i)$ an arbitrary weight system of \emph{positive} real scalars on its reflections (that agrees on $\tau$ and $\tau^{-1}$). Is it true that the smallest eigenvalue of the representation Laplacian $\mathbf{B}_W(\bm\omega)$ will always appear in some representation $U$ of (generalized) Coxeter number equal to the (standard) Coxeter number $h$ of $W$?

\noindent If not, is the stronger assumption that the weights are induced by a parabolic tower $T$ as in \eqref{Eq: parabolic tower} sufficient?
\end{question}

\subsection{The category $\mathcal{O}$ of rational Cherednik algebras}

The (generalized) Coxeter numbers $c_{\chi}$ play an important role in defining the highest-weight structure for the category $\mathcal{O}$ of the rational Cherednik algebra of a reflection group $W$. The irreducible characters $\chi$ of $W$ index the standard objects of the category $\mathcal{O}$ and the numbers $c_{\chi}$ define an order $<$ on them via $\mu > \chi$ if $c_{\mu}-c_{\chi}\in\mathbb{Z}_{\geq 0}$.

Griffeth et al. \cite{GGJL} showed that this order may be replaced by a much coarser one that involves the same Coxeter-theoretic data of parabolic towers that we have considered in our work. Under this new coarser order, we have that $\mu > \chi$ if for all parabolic towers $T$ there are always two tuples $\towerchi:=(\chi_0,\chi_1,\cdots,\chi_n)\in\op{Res}_T(\chi)$ and $\underline{\bm\mu}:=(\mu_0,\mu_1,\cdots,\mu_n)\in\op{Res}_T(\mu)$ (recall Notation~\ref{Not: chi restricted down T}) for which $c_{\mu_i}-c_{\chi_i}\in \mathbf{Z}_{\geq 0}$ for all indices $i=1\dots n$.

A main ingredient in \cite{Douvr} for the proof of the Chapuy-Stump formula \eqref{EQ: CS-formula} was an argument saying that only characters $\chi$ with Coxeter numbers $c_{\chi}$ that are multiples of $h$ needed to be considered in the Frobenius lemma. We wonder whether, in this way or otherwise, one might use the category $\mathcal{O}$ structure to give a uniform proof of our Theorem~\ref{thm:equivalence}.

\subsection{Product formulas for finer weight systems}
\label{Sec: product formulas}

As we explain in \S~\ref{sec:LieLike}, the product structure of formula \ref{eq:mainThm} in our Thm.~\ref{thm:main} is a consequence of the agreement of the virtual characters $\sum_{\chi\in\widehat{W}}\chi(c^{-1})\cdot \chi$ and $\sum_{k=0}^n(-1)^k\chi_{\wedge^k V_{ref}}$ over the subalgebras $\CC[\bm J_T]$ for any choice of parabolic tower $T$. It turns out that in certain special cases we can replace the reflection representation $V_{ref}$ with other $n$-dimensional representations so that the corresponding virtual characters agree on larger subalgebras. Then, in the same way as in \S~\ref{sec:LieLike}, these representation-theoretic relations force product formulas for finer weight systems.

\subsubsection{The hyperoctahedral groups $W=B_n$}
In the case of $W=B_n$ for example, it is an easy check after \cite[\S~5]{CS}, that for the representation $U=\big((n-1,1),\emptyset\big)\oplus\big(\emptyset,(n)\big)$ we have in fact an equality $$\sum_{\chi\in\widehat{B_n}}\chi(c^{-1})\cdot \chi=\sum_{k=0}^n(-1)^k\chi_{\wedge^kU},$$ as virtual characters (that is, they agree on the whole group algebra $\CC[B_n]$). Moreover, for any reflection $\tau$, the element $(1-\tau)$ is \emph{Lie-like} for the representation $U$ also. This means that in the case of $B_n$ we have the following \emph{full generalization} of the Burman and Zvonkine formula:
\begin{equation}
\mathcal{F}_{B_n}^P(t,\bm\omega)=\dfrac{e^{t\cdot\wt(\mathcal{R})}}{2n}\prod_{i=1}^n(1-e^{-t\lambda_i}),
\end{equation} where $\wt$ is \emph{any} weight system (that is, induced by \emph{any} partition $P$ of the set of reflections $\mathcal{R}$, including the one in singletons) and where $\lambda_i(\bm\omega)$ are the complex eigenvalues of the $n\times n$ matrix $\rho_U\big(\sum_{\tau\in\mathcal{R}}\wt(\tau)\cdot (1-\tau)\big)$. It is an easy check that for generalized Jucys-Murphy weights these $\lambda_i(\bm\omega)$ agree with the eigenvalues of the $W$-Laplacian for $B_n$.

\subsubsection{The infinite family $W=G(r,1,n)$}
The case of $W=G(r,1,n)$ is slightly more complicated. There, the character equality becomes: $$\sum_{\chi\in\widehat{G(r,1,n)}}\chi(c^{-1})\cdot\chi=\sum_{k=0}^n(-1)^k\Big[\xi^{-1}\chi_{\wedge^kU_1}+\xi^{-2}\chi_{\wedge^k U_2}+\cdots +\xi^{1-r}\chi_{\wedge^k U_{r-1}}\Big],$$ where $\xi$ is a primitive $r$-th root of unity and  the $n$-dimensional representations $U_i$ are defined via $$U_i=\big((n-1,1),\emptyset,\cdots,\emptyset\big)\oplus\big(\emptyset,\cdots,\underbrace{(n)}_i,\cdots,\emptyset\big)$$ with the indexing starting at $i=0$. This does not immediately give a product formula, but it is easy to see that the representations $U_i$ all agree on the sums $(\tau+\tau^2+\cdots + \tau^{e_H-1})$ for any reflection $\tau$ of order $e_H$. We then have a simpler equivalence over the subalgebra generated by these sums and (since the elements $(1-\tau)$ are again \emph{Lie-like} for any $U_i$) it forces the product formula \begin{equation}
\mathcal{F}_{G(r,1,n)}^P(t,\bm\omega)=\dfrac{e^{t\cdot \wt(\mathcal{R})}}{rn}\prod_{i=1}^n(1-e^{-t\lambda_i}),
\end{equation} where $\wt$ is any weight system that respects the reflection hyperplanes (i.e. such that $\wt(\tau)=\wt(\tau')$ if $V^{\tau}=V^{\tau'}$) and where $\lambda_i(\bm\omega)$ are the eigenvalues of the matrix $\rho_{U_i}\big(\sum_{\tau\in\mathcal{R}}\wt(\tau)\cdot(1-\tau)\big)$ (for any $i$). Once more, it is an easy check that for parabolic Jucys-Murphy weights (which do indeed respect reflection hyperplanes) these agree with the eigenvalues of the $W$-Laplacian.

\subsubsection{The dihedral groups $I_2(m)$}
\label{Sec: end: dihedrals}

Here we have an example where weight systems induced by towers of \emph{reflection} subgroups (i.e. not necessarily parabolic) always give a product formula for $\mathcal{F}_W(t,\bm\omega)$. For the dihedral group $W=I_2(m)$, any such tower corresponds to a chain in the poset of divisors of $m$. That is, for any positive integers $\bm m:=(m_1,m_2,\cdots, m_k=m)$ such that $m_i|m_{i+1}$ we consider the tower $$T_{\bm m}:=\big(\{\mathbbl{1}\}\leq I_2(m_1)\leq I_2(m_2)\leq \cdots \leq I_2(m)\big).$$ As in \eqref{Eq: parabolic tower} it determines a weight system with complex weights $\bm\omega:=(\omega_i)_{i=1}^k$ via $\mathbf{w}_{\bm m}(\tau)=\omega_i$ if and only if $\tau\in I_2(m_i)\setminus I_2(m_{i-1})$. Then, it is not very difficult after \S~\ref{sec:dihedral} to see that the exponential generating function of $\mathbf{w}_{\bm m}$-weighted factorizations has a product formula \begin{equation}
\mathcal{F}_{I_2(m)}^{\bm m}(t,\bm \omega)=\dfrac{e^{t\cdot \mathbf{w}_{\bm m}(\mathcal{R})}}{m}\big(1-e^{-t\lambda_1(\bm\omega)}\big)\cdot \big(1-e^{-t\lambda_2(\bm\omega)}\big),
\end{equation} where $\lambda_1(\bm\omega)=m\cdot \omega_k$ and $\lambda_2(\bm\omega)=2m_1\cdot \omega_1 +\big(\sum_{i=2}^{k-1}2(m_i-m_{i-1})\cdot\omega_i\big) +(m-2m_{k-1})\cdot\omega_k$.

\subsubsection{Are there any further \emph{uniform} results?}
\label{sec: end: better version?}

Our main Theorem~\ref{thm:main} is in some sense a \emph{maximally uniform} result; if we allow finer weights, the enumeration might still be given by factoring formulas but the exponentials appearing in those will not be the eigenvalues of the $W$-Laplacian matrix. For that matter, there does not seem to be a way to choose a \emph{single} representation $U$, somehow uniformly defined for reflection groups (in fact, even if we only restrict to dihedral groups) and other than the reflection representation $V$, that would give finer results via the Lie-like elements method. One could choose however different qualitative characteristics to maximize. It is natural to ask:
\begin{enumerate}
\item Do we get a product formula when weights are defined by arbitrary reflection towers (i.e. not necessarily parabolic)? If yes, is there a uniform combinatorial description of the exponents?
\item What are the finest partitions of the set of reflections for a given group $W$ that lead to product formulas? Can they be characterized somehow?
\end{enumerate} 

%

\section*{Acknowledgments}
We would like to thank Ivan Marin for a very illuminating discussion regarding the Gelfand-Tsetlin decomposition for parabolic towers, especially the comments surrounding equation \eqref{EQ: Ivan Marin}. Also Daniel Juteau and Arun Ram for their remarks vis a vis Cherednik algebras and the category $\mathcal{O}$.

We are deeply grateful to Jean Michel who noticed a gap in Lemma~\ref{lemma:subgroup} in an earlier version of this paper. Moroever, after we fixed the gap, Jean Michel was able to extend his work \cite{Michel} and give a case-free proof of our Theorem~\ref{thm:equivalence} for Weyl groups \cite{JM_uniform}.

An extended abstract of this work was accepted as a poster for the 32nd International Conference on Formal Power Series and Algebraic Combinatorics (FPSAC) 2020 (that took place online because of the COVID-19 pandemic) and was published in the conference proceedings as \cite{FPSAC}.

\appendix

\section{Littlewood-Richardson rules for $\mathfrak{S}_n$,  $G(r,1,n)$ and $G(r,r,n)$}
\label{app:LRtheory}

In this appendix we recall the Littlewood-Richardson rules that enable one to compute restrictions of irreducible representations for the infinite families, via their combinatorial data. In Appendix~\ref{app:LRinstances} we will apply these rules to certain special cases that are used in the main body of the text. 

 We represent partitions by their Ferrer diagram in the \emph{French} convention, i.e. row lengths are nondecreasing from bottom to top. 

\subsection{The case of $\mathfrak{S}_n$}

Let $n \geq 1$, let $a\in \{1,...,n-1\}$ and $b=n-a$. For $\lambda \vdash n$ a partition, the irreducible representation $\lambda$ is {\it a priori} not irreducible when restricted to the subgroup $\mathfrak{S}_a\times \mathfrak{S}_{b}$. It splits, with possible multiplicities, into irreducible components of the form $\alpha \otimes \beta$, where $\alpha \vdash a$ and $\beta\vdash b$. The \emph{Littlewood-Richardson coefficient (LR-coefficient)} is the multiplicity in this restriction
$$
c^{\alpha, \beta}_\lambda := \left\langle \lambda \big|_{\mathfrak{S}_a\times \mathfrak{S}_{b}}, \alpha\otimes \beta \right\rangle_{\mathfrak{S}_a\times \mathfrak{S}_{b}}.
$$

The \emph{Littlewood-Richardson rule (LR-rule)} (see for example \cite[Chapter~7,~A1.3]{EC2}) enables one to compute the LR-coefficient $c^{\alpha, \beta}_\lambda$ as follows.

First, the coefficient $c^{\alpha, \beta}_\lambda$ is zero unless $\alpha$ is a subpartition of $\lambda$, meaning that the Ferrers diagram of $\alpha$ fits inside the one of $\lambda$. If this is the case, a semistandard tableau of (skew) shape $\lambda\setminus \alpha$ is an \emph{LR-tableau} if it has the property that reading the filling of its rows from bottom to top and from right to left, gives a Yamanouchi word -- that is to say, a word starting with the letter $1$, and such that in any prefix, the number $i$ occurs at least as many times as $i+1$, for any $i\geq 1$. Then the LR-rule says that the coefficient $c^{\alpha, \beta}_\lambda$ is equal to the number of LR-tableaux of shape $\lambda\setminus \alpha$ and content $\beta$.
See Appendix~\ref{appsec:LRapplications} for concrete applications of the rule.

\subsection{The case of $G(r,1,n)$}
\label{appsec:LRGr1n}

Let $n,r\geq 1$, $a\in [1..n]$, let $b=n-a$ and let $\vec \lambda=(\lambda^{(1)}, \lambda^{(2)}, ..., \lambda^{(r)})$ be an $r$-tuple of partitions of total size $n$. Thus $\vec \lambda$ indexes an irreducible representation of $G(r,1,n)$. 
The goal of this section is to describe the restriction of this representation  over the chain of subgroups:
$G(r,1,a)\times S_{b}\hookrightarrow G(r,1,a)\times G(r,1,b)\hookrightarrow G(r,1,n)$.

We first consider the subgroup $G(r,1,a)\times G(r,1,b)\hookrightarrow G(r,1,n)$.
Irreducible representations of that subgroup are of the form $\vec \alpha \otimes \vec \beta$
where $\vec \alpha$ and $\vec \beta$ are two $r$-tuples of partitions of respective total sizes $a$ and $b$.
Using the description of characters of $G(r,1,a)$ it is easy to see  (see e.g. the appendix of \cite{CS} for this description, and~\cite{Stembridge} for the derivation in the case $r=2$) that the restriction of $\vec \lambda$ is given by
\begin{align}\label{appeq:parabolicLR2}
\vec\lambda\Big\downarrow^{G(r,r,n)}_{G(r,1,a)\times G(r,1,b)}
	=\sum_{|\vec \alpha|=a\atop |\vec \beta|=b}
	\left(\prod_{i=1}^{r} c^{\alpha^{(i)}, \beta^{(i)}}_{\lambda^{(i)}} \right) 
	\big(\alpha \otimes \vec \beta\big)
\end{align}
where the $c$'s are the classical LR-coefficients described the previous section, and where the sum is taken over $r$-tuples of partitions
$\vec \alpha=(\alpha^{(1)}, \alpha^{(2)}, ..., \alpha^{(r)})$ and 
$\vec \beta=(\beta^{(1)}, \beta^{(2)}, ..., \beta^{(r)})$
of respective total sizes $a$ and $b$, such that $|\alpha^{(i)}|+|\beta^{(i)}|=|\lambda^{(i)}|$ for each $i\in\{1,\dots,r\}$.
In words: {\it along the restriction $G(r,1,a)\times G(r,1,b)\hookrightarrow G(r,1,n)$, each coordinate of the $r$-tuple of partitions splits in to an $a$ and a $b$ component, according to the classical LR-rule.}

We now consider the subgroup $G(r,1,a)\times \mathfrak{S}_{b}\hookrightarrow G(r,1,a)\times G(r,1,b)$, whose irreducible representations are of the form $\vec \alpha\otimes \gamma$ with $\vec \alpha$ as above and $\gamma$ a partition of $b$.
Let $\vec \alpha=(\alpha^{(1)}, \alpha^{(2)}, ..., \alpha^{(r)})$ and 
$\vec \beta=(\beta^{(1)}, \beta^{(2)}, ..., \beta^{(r)})$
be two $r$-tuples of partitions of respective total size $a$ and $b$
Again using the description of characters of $G(r,1,a)$ (see the same references)
it is easy to see that the restriction of $\vec \alpha \otimes \vec \beta$ is given by
\begin{align}\label{appeq:parabolicLR3}
	\big( \vec\alpha \otimes \vec \beta\big)\Big\downarrow^{G(r,1,a)\times G(r,1,b)}_{G(r,1,a)\times \mathfrak{S}_{b}}
	=\sum_{\gamma\vdash b}  c^{\vec\beta}_{\gamma}  \cdot \big( \vec\alpha\otimes \gamma\big),
\end{align}
where 
\begin{align}\label{eq:LRgen}
	c^{\vec\beta}_{\gamma}=\left\langle \vec\beta \Big\uparrow^{\mathfrak{S}_b}_{\mathfrak{S}_{|\beta^{(1)}|}\times \mathfrak{S}_{|\beta^{(1)}|}\times ... \times \mathfrak{S}_{|\beta^{(r)}|}} , \gamma\right\rangle _{\mathfrak{S}_b}
	=\left\langle \gamma \Big\downarrow^{\mathfrak{S}_b}_{\mathfrak{S}_{|\beta^{(1)}|}\times \mathfrak{S}_{|\beta^{(1)}|}\times ... \times \mathfrak{S}_{|\beta^{(r)}|}} , \vec \beta \right\rangle _{\mathfrak{S}_{|\beta^{(1)}|}\times \mathfrak{S}_{|\beta^{(1)}|}\times ... \times \mathfrak{S}_{|\beta^{(r)}|}}
\end{align}
is a generalization of the LR-coefficient that describes restrictions from $\mathfrak{S}_b$ to the $r$-component Young subgroup $\mathfrak{S}_{|\beta^{(1)}|}\times ... \times \mathfrak{S}_{|\beta^{(r)}|}$.

Or course the generalized LR coefficient~\eqref{eq:LRgen} can be expressed in terms of the LR coefficients of the previous section via successive restrictions. However we will not need to do that, because the $r$-tuples $\vec \beta$ that we consider in this paper always have at most two non-empty coordinates. For the record, we note that for partitions $\rho^{(1)}, \rho^{(2)}, \rho^{(3)}$ such that $|\rho^{(1)}|=n$ and $|\rho^{(2)}|+|\rho^{(3)}|=n$ we have
$$c^{(\rho^{(1)}, \emptyset, ..., \emptyset)}_{\gamma} = \mathbf{1}_{\{\rho^{(1)}=\gamma\}} 
\mbox{ and }
c^{(\rho^{(2)}, \emptyset, ..., \rho^{(3)},\emptyset, ...)}_{\gamma} = c^{\rho^{(2)}, \rho^{(3)}}_{\gamma}.
$$

\subsection{The case of $G(r,r,n)$}
\label{appsec:LRGrrn}

For the group $G(r,r,n)$, we will not need any new rule of restriction. We just note that $G(r,r,n)$ being a subgroup of $G(r,1,n)$ the restrictions to $G(r,r,a)\times G(r,r,b)$ and $G(r,r,a)\times\mathfrak{S}_{b}$ are still given by \eqref{appeq:parabolicLR2} and \eqref{appeq:parabolicLR3}. The only difference is that the representations appearing in these formulas are not necessarily irreducible, since the $r$-tuples of partitions may have nontrivial cyclic symmetries.
However this will \emph{not} be a problem for us. The interested reader may consult~\cite{Stembridge} for details on how irreducible representations of $G(r,1,n)$ split when restricted to $G(r,r,n)$ for the case $r=2$ (higher values of $r$ are similar), but we will not need this.

\section{Instances of the LR-rule for $\mathfrak{S}_n$}\label{appsec:LRapplications}
\label{app:LRinstances}

In this appendix, we will use everywhere Notation~\ref{not:zero}, given in page~\pageref{not:zero}.

\subsection{Hooks}\label{appsec:LRhooks}

We compute the restriction of the hook $\hook{n}{k}$ from $\mathfrak{S}_n$ to $\mathfrak{S}_a \times \mathfrak{S}_{b}$ using the LR-rule. We could  easily compute it directly (using the exterior product point of view) but using the LR-rule will serve as a warm-up for the next section.

The character $\hook{n}{k}$ splits into a sum of characters of the form $\alpha\otimes \beta$, with multiplicities. They satisfy that $\alpha$ and $\beta$ are subpartitions of $\hook{n}{k}$, which implies that $\alpha$ and $\beta$ are hooks. For a hook $\alpha=\hook{a}{i}$,  the shape $\hook{n}{k}\setminus \alpha$ corresponds to at most two LR-tableaux, which are described on Figure~\ref{fig:LRhooks}. We obtain
\begin{align}\label{appeq:LRhook1}
\hook{n}{k} \Big\downarrow_{\mathfrak{S}_a \times \mathfrak{S}_{b}}
= \sum_{i\geq 0}\hook{a}{i} \otimes \big( \hook{b}{k-i-1} + \hook{b}{k-i} \big).
\end{align}
Note that we can rewrite this in the more symmetric form:
\begin{align}\label{appeq:LRhook2}
\hook{n}{k} \Big\downarrow_{\mathfrak{S}_a \times \mathfrak{S}_{b}}
= \sum_{i,j\geq 0, \epsilon \in \{0,1\} \atop i+j+\epsilon=k }\hook{a}{i} \otimes \hook{b}{j}.
\end{align}

\begin{figure}[h]
	\begin{center}
		\includegraphics[width=0.6\linewidth]{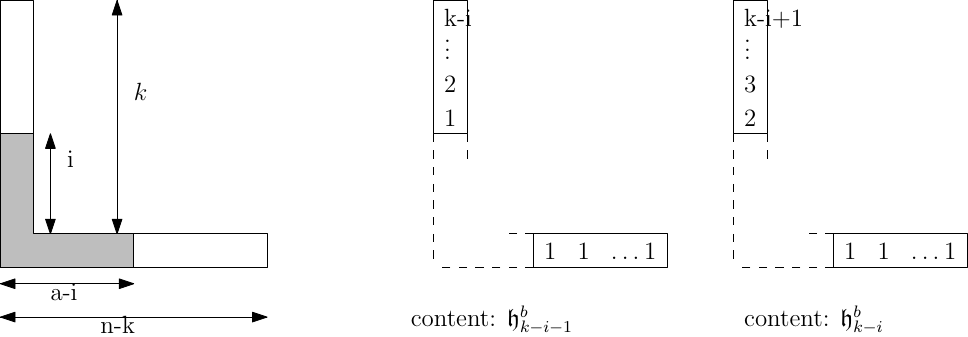}
		\caption{Application of the LR rule to compute the restriction of the hook $\hook{n}{k}$ to $\mathfrak{S}_a \times \mathfrak{S}_{b}$.
		}\label{fig:LRhooks}
	\end{center}
\end{figure}

\subsection{Quasi-hooks}\label{appsec:LRquasihooks}

We now study the restriction of the quasihook $\qhook{n}{k}$ from $\mathfrak{S}_n$ to $\mathfrak{S}_a \times \mathfrak{S}_{b}$.
It splits into terms of the form $\alpha\otimes \beta$, where $\alpha$ and $\beta$, being subpartitions of the quasihook, have to be either hooks or quasihooks.

We start with the case where $\alpha$ is a quasihook, say $\alpha=\qhook{a}{i}$. In this case, the shape $\qhook{n}{n}\setminus \alpha$, illustrated on Figure~\ref{fig:LRquasihooks}-top, is similar as the one studied in Section~\ref{appsec:LRhooks}. It gives rise, up to a change of $(n,a)$ to $(n-1,a-1)$, to the same LR-tableaux as described on Figure~\ref{fig:LRhooks} and we obtain for this first case a contribution of
$$
T_1= \sum_{i,j\geq 0, \epsilon \in \{0,1\} \atop i+j+\epsilon=k }\qhook{a}{i} \otimes \hook{b}{j}.
$$
By symmetry of LR coefficients, we will also obtain a term $T_2$, analogous to $T_1$ but with the quasihooks to the right of the tensor product. Finally we also see that we find no contributions where both $\alpha$ and $\beta$ are quasihooks, and we conclude that
$
\qhook{n}{k} \Big\downarrow_{\mathfrak{S}_a \times \mathfrak{S}_{b}} = T_1+T_2+T_3$, where
 $T_3$ is a linear combination of tensors involving \emph{hooks} only. 

It remains to compute $T_3$, so we let $\alpha=\hook{a}{i}$ be a hook and we want to list all the LR-tableaux of shape $\qhook{n}{n}\setminus \alpha$ whose content itself forms a hook. For $i>0$  there are at most four of them, two of them with the same content,  (Figure~\ref{fig:LRquasihooks}) while for $i=0$ there are at most two of them (Figure~\ref{fig:LRquasihooks2}). We get
\begin{align*}
	T_3&= \sum_{i\geq 1,j\geq 0, \epsilon \in \{-1,0,1\} \atop i+j+\epsilon=k } (1+\delta_\epsilon) \cdot \hook{a}{i} \otimes \hook{b}{j}
	+ \hook{a}{0} \otimes (\hook{b}{k-1}+\hook{b}{k})\\
	&= \sum_{i,j\geq 0, \epsilon \in \{-1,0,1\} \atop i+j+\epsilon=k } (1+\delta_\epsilon) \cdot \hook{a}{i} \otimes \hook{b}{j}
	- \hook{a}{0} \otimes (\hook{b}{k}+\hook{b}{k+1}).
\end{align*}

To summarize, we finally obtain:
\begin{proposition}The restriction of a quasihook representation of $\mathfrak{S}_n$ to the maximal Young subgroup $\mathfrak{S}_a \times \mathfrak{S}_{b}$ is given by
	\begin{align}
		\qhook{n}{k} \Big\downarrow_{\mathfrak{S}_a \times \mathfrak{S}_{b}} &=
		\sum_{i,j\geq 0, \epsilon \in \{0,1\} \atop i+j+\epsilon=k }\qhook{a}{i} \otimes \hook{b}{j} 
		+\sum_{i,j\geq 0, \epsilon \in \{0,1\} \atop i+j+\epsilon=k }\hook{a}{i} \otimes \qhook{b}{j} 
		+\sum_{i,j\geq 0, \epsilon \in \{-1,0,1\} \atop i+j+\epsilon=k }(1+\delta_\epsilon) \cdot\hook{a}{i} \otimes \hook{b}{j}\nonumber\\
		&\quad\quad- \hook{a}{0} \otimes (\hook{b}{k}+\hook{b}{k+1}).
	\end{align}
	\end{proposition}

\begin{figure}[h]
	\begin{center}
		\includegraphics[width=0.7\linewidth]{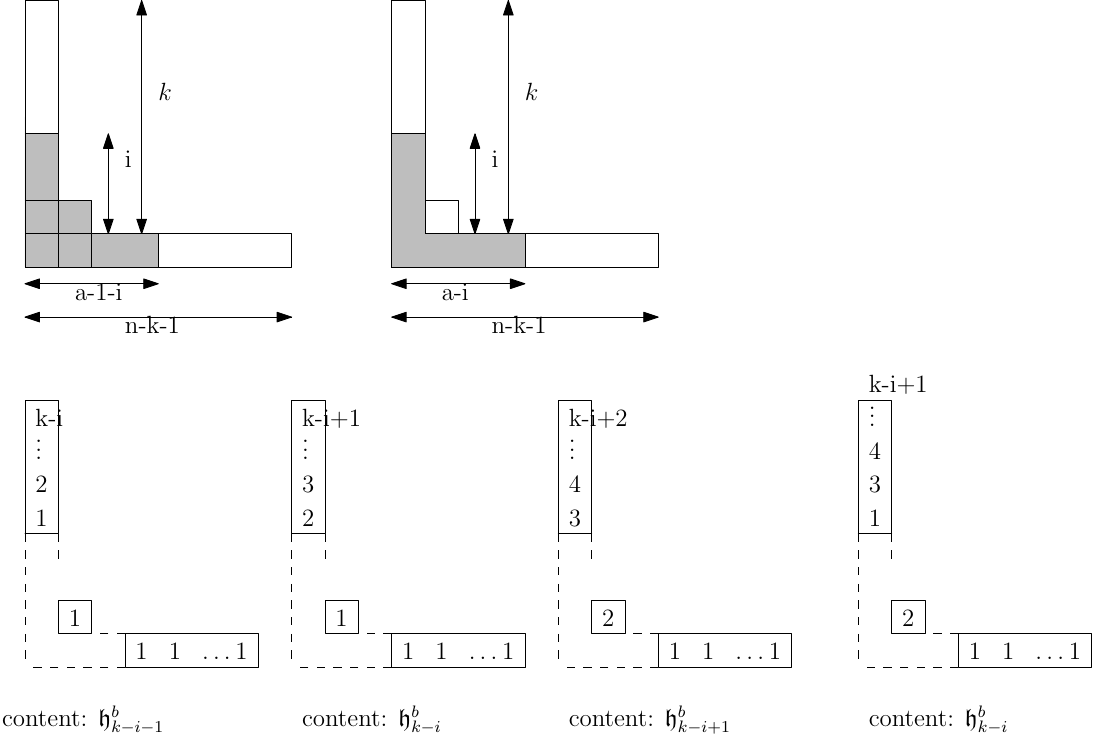}
		\caption{Application of the LR rule to compute the restriction of the quasihook $\qhook{n}{k}$ to $\mathfrak{S}_a \times \mathfrak{S}_{b}$. Top-Left: when $\alpha$ is a quasi-hook, we obtain the same skew shape as on Figure~\ref{fig:LRhooks}. Top-Right and Bottom: when $\alpha$ is a hook, there are at most four LR-tableaux such that $\beta$ is also a hook.
		}\label{fig:LRquasihooks}
	\end{center}
\end{figure}

\begin{figure}[h]
	\begin{center}
		\includegraphics[width=0.7\linewidth]{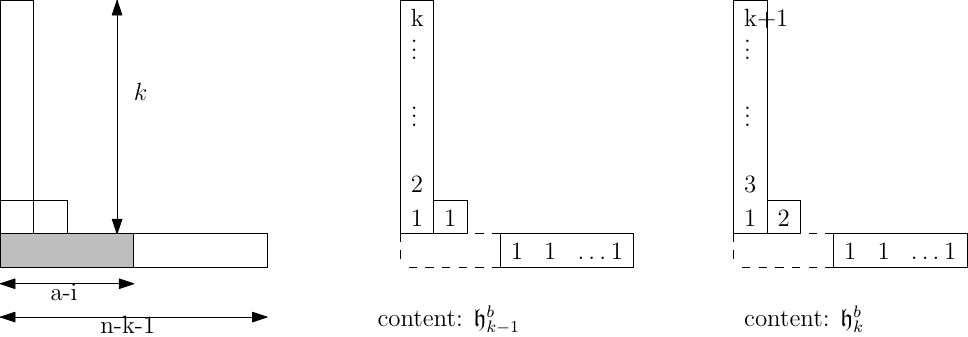}
		\caption{Application of the LR rule to compute the restriction of the quasihook $\qhook{n}{k}$ to $\mathfrak{S}_a \times \mathfrak{S}_{b}$, in the special case when $\alpha$ is the hook $\hook{a}{0}$. The picture displays the LR-tableaux such that $\beta$ is also a hook.		}\label{fig:LRquasihooks2}
	\end{center}
\end{figure}

\newpage
\bibliographystyle{alpha}
\bibliography{biblio}

\Address

\end{document}